\documentclass[11pt]{article}

\usepackage{amsmath,amssymb,amsfonts,amsthm,enumitem,stmaryrd}
\usepackage{pstricks,pst-node,xy,float}
\xyoption{all}

\numberwithin{equation}{section}
\newtheorem{thm}{Theorem}[section]
\newtheorem{prp}[thm]{Proposition}
\newtheorem{lmm}[thm]{Lemma}   
\newtheorem{crl}[thm]{Corollary}

\newtheorem{mythm}{Theorem}

\newtheorem{myconj}{Conjecture}

\newtheorem{thmlet}{Theorem}

\theoremstyle{definition}
\newtheorem{rmk}[thm]{Remark}

\newcounter{zr}
\def\zr#1{\stepcounter{zr}\bibitem{#1} A.~Zinger,}
\def\sf#1{\textsf{#1}}

\def\BE#1{\begin{equation}\label{#1}}  
\def\EE{\end{equation}}
\def\e_ref#1{(\ref{#1})}
\def\ti#1{\tilde{#1}}
\def\ov#1{\overline{#1}}
\def\un#1{\underline{#1}}
\def\ti#1{\tilde{#1}}
\def\wt#1{\widetilde{#1}}
\def\lan{\langle}  \def\ran{\rangle}
\def\lr#1{\lan{#1}\ran}
\def\blr#1{\big\lan{#1}\big\ran}
\def\bblr#1{\bigg\lan{#1}\bigg\ran}
\def\lrbr#1{\llbracket{#1}\rrbracket}
\def\LRbr#1{\left\llbracket{#1}\right\rrbracket}

\def\smsize#1{\begin{small}#1\end{small}}

\def\lra{\longrightarrow}
\def\Lra{\Longrightarrow}
\def\Llra{\Longleftrightarrow}

\def\al{\alpha}
\def\be{\beta}
\def\de{\delta}

\def\ga{\gamma}

\def\la{\lambda}
\def\om{\omega}
\def\si{\sigma}
\def\th{\theta}

\def\ze{\zeta}

\def\De{\Delta}
\def\Ga{\Gamma}

\def\Si{\Sigma}

\def\bA{\mathbb A}
\def\c{\mathbf c}
\def\C{\mathbb C}
\def\cC{\mathcal C}
\def\nc{\textnormal{c}}
\def\ntc{\tilde{\textnormal{c}}}
\def\ntC{\tilde{\textnormal{C}}}

\def\bD{\mathbf D}

\def\bfd{\mathbf d}

\def\E{\mathbf e}
\def\tnE{\textnormal{E}}

\def\bG{\mathbb G}
\def\fL{\mathfrak L}
\def\H{\mathcal H}
\def\bfH{\mathbf H}

\def\I{\mathfrak i}
\def\cI{\mathcal I}
\def\L{\mathcal L}
\def\M{\mathfrak M}
\def\bM{\mathbf M}
\def\cM{\mathcal M}
\def\N{\mathcal N}
\def\O{\mathcal O}
\def\P{\mathbb P}
\def\cP{\mathcal P}
\def\Q{\mathbb Q}
\def\R{\mathbb R}

\def\fR{\mathfrak R}
\def\bfS{\mathbf S}
\def\cS{\mathcal S}
\def\bft{\mathbf t}
\def\T{\mathbb T}
\def\U{\mathfrak U}
\def\V{\mathcal V}
\def\cY{\mathcal Y}
\def\Z{\mathbb Z}
\def\cZ{\mathcal Z}

\def\a{\mathbf{a}}
\def\b{\mathbf{b}}
\def\d{\mathfrak d}
\def\tnd{\textnormal{d}}
\def\p{\mathbf{p}}

\def\q{\mathbf{q}}

\def\x{\mathbf{x}}
\def\0{\mathbf{0}}

\def\nua{\nu_{\a}}

\def\i{\infty}
\def\hb{\hbar}

\def\tnd{\textnormal{d}}

\def\dim{\textnormal{dim}}
\def\ev{\textnormal{ev}}
\def\val{\textnormal{val}}
\def\Edg{\textnormal{Edg}}
\def\Ver{\textnormal{Ver}}

\def\1{\mathbf 1}

\def\dset{\llfloor d\rrfloor}
\def\nset{\llfloor n\rrfloor}
\def\eset{\emptyset}

\def\Res#1{\underset{#1}{\fR}}

\begin{document}

\thispagestyle{empty}

\title{The Genus 0 Gromov-Witten Invariants\\ of Projective Complete Intersections}
\author{Aleksey Zinger\thanks{Partially supported by DMS grant 0846978}}
\date{\today}

\maketitle

\begin{abstract}
We describe the structure of mirror formulas for genus 0 Gromov-Witten invariants
of projective complete intersections with any number of marked points
and provide an explicit algorithm for obtaining the relevant structure coefficients. 
The structural description alone suffices for some qualitative applications, 
such as vanishing results and the bounds on the growth of these invariants
predicted by R.~Pandharipande. 
\end{abstract}

\tableofcontents

\section{Introduction}
\label{intro_sec}

\noindent
Gromov-Witten invariants of a smooth projective variety~$X$
are certain counts of curves in~$X$.
In many cases, these invariants are known or conjectured to possess
rather amazing structure which is often completely unexpected from
the classical point of view.
For example, the genus~0 GW-invariants of a quintic threefold,
i.e.~a degree~5 hypersurface in~$\P^4$, are related by a so-called
\sf{mirror formula} to hypergeometric series.
This relation was explicitly predicted in~\cite{CdGP}
and mathematically confirmed in~\cite{Be}, \cite{Ga}, 
\cite{Gi}, \cite{Le}, and~\cite{LLY}.
In fact, the prediction of~\cite{CdGP} has been shown to be 
a special case of closed formulas for 1-pointed genus~0 GW-invariants 
(counts of curves passing through one constraint)
of complete intersections of sufficiently small 
total multi-degree (\cite{Gi2}, \cite{LLY}).
It is shown in~\cite{bcov0} that closed formulas for 2-pointed genus~0 
GW-invariants of hypersurfaces are explicit transforms of the 1-pointed formulas;
this is extended to projective  complete intersections in \cite{Ch} and~\cite{PoZ}.\\

\noindent
The classical localization theorem of~\cite{ABo} reduces the computation 
of genus~0 GW-invariants of projective complete intersections to a sum over decorated graphs.
In this paper, we use the method of~\cite{bcov1} for breaking such graphs 
at special nodes to show that closed formulas for $N$-pointed genus~0
GW-invariants of projective complete intersections
 are explicit transforms of the 1-pointed formulas,
with the key link provided by the transform for the 2-pointed invariants 
obtained in~\cite{PoZ}.
We show that closed formulas for $N$-pointed genus~0 
GW-invariants of  projective  complete intersections, with $N\!\ge\!3$,
are linear combinations of $N$-fold products of derivatives of 1-pointed formulas
with coefficients that are polynomials of total degree at most~$N\!-\!3$.
While we describe two explicit ways of computing the coefficients of these polynomials,
the final formulas become rather complicated as~$N$ increases.
Nevertheless, our qualitative description of generating functions for $N$-pointed 
GW-invariants as linear combinations of $N$-fold products of derivatives
leads to some simple-to-state qualitative results concerning these invariants; 
see Theorems~\ref{GWbound_thm} and~\ref{GWeq0_thm} below.\\


\noindent
Throughout the paper $N\!\ge\!3$, $n\!\ge\!2$, and $l\!\ge\!0$ will be 
fixed integers and 
$$\a\equiv(a_k)_{k=1,2,\ldots,l}\equiv(a_1,\ldots,a_l)$$
a tuple of positive integers,
with $N$ and $\a$ denoting the number of marked points and  
the multi-degree of a fixed  complete intersection $X_{\a}\!\subset\!\P^{n-1}$,
respectively.
Let
\begin{gather*}
[N]=\{1,2,\ldots,N\},\\
|\a|=\sum_{k=1}^{k=l}a_k\,,\quad 
\|\a\|=\sum_{k=1}^{k=l}ka_k\,,\quad
\lr{\a}=\prod_{l=1}^{l=k}a_k\,, 
\quad \a^{\a}=\prod_{k=1}^{k=l}a_k^{a_k}, 
\quad \a!=\prod_{k=1}^{k=l}a_k!\,, 
\quad  \nua= n-|\a|.
\end{gather*}
For any nonnegative integer $d$, we denote by $\ov\M_{0,N}(X_{\a},d)$
the moduli spaces of genus~0 degree~$d$ $N$-marked stable maps to~$X_{\a}$. 
For each $s\!=\!1,\ldots,N$, let
$$\ev_s\!: \ov\M_{0,N}(X_{\a},d)\lra X_{\a}, \qquad
\psi_s\equiv c_1(L_s^*)\in H^2\big(\ov\M_{0,N}(X_{\a},d)\big),$$
be the evaluation map 
and the first chern of  the universal tangent line bundle at the $s$-th marked point.
Denote by  $H\!\in\!H^2(\P^{n-1})$ denote the hyperplane class.\\

\noindent
The main theorem of this paper, Theorem~\ref{main_thm} in Section~\ref{multipt_subs}, 
provides closed formulas for generating functions for genus~0 GW-invariants 
\BE{GWXadfn_e}\blr{\tau_{b_1}H^{c_1},\ldots,\tau_{b_N}H^{c_N}}_{0,d}^{X_{\a}}
\equiv \int_{[\ov\M_{0,N}(X_{\a},d)]^{vir}}    
\prod_{s=1}^{s=N}\!\!\big(\psi_s^{b_s}\ev_s^*H^{c_s}\big)\EE
of a complete intersection $X_{\a}\!\subset\!\P^{n-1}$ of multi-degree~$\a$
with $|\a|\!\le\!n$.
The precise statement of these formulas is quite involved
and is thus deferred until Section~\ref{mainthm_subs}.
We instead begin by describing some qualitative corollaries of Theorem~\ref{main_thm},
Theorems~\ref{GWbound_thm} and~\ref{GWeq0_thm}, and 
special cases, Theorems~\ref{cy_thm} and~\ref{proj_thm}.

\begin{mythm}\label{GWbound_thm}
If $n\!\in\!\Z^+$, $\a\!\in\!(\Z^+)^l$, and $X_{\a}\!\subset\!\P^{n-1}$ is a complete intersection
of multi-degree~$\a$, there exists $C_{\a}\!\in\!\R^+$ such~that 
$$\big|\blr{b_1!\,\tau_{b_1}H^{c_1},\ldots,b_N!\,\tau_{b_N}H^{c_N}}_{0,d}^{X_{\a}}\big|
\le  N!\,C_{\a}^{N+d}
\qquad\forall\,N\!\in\!\Z^+,~d,b_1,\ldots,b_N,c_1,\ldots,c_N\!\in\!\Z.$$
\end{mythm}

\noindent
This bound holds for $d\!=\!0$, since
\begin{equation*}\begin{split}
\blr{\tau_{b_1}H^{c_1},\ldots,\tau_{b_N}H^{c_N}}_{0,0}^{X_{\a}}
&=\lr\a \bigg(\int_{\P^{n-1}}\!\!H^{c_1+\ldots+c_N+l}\bigg)
\bigg(\int_{\ov\cM_{0,N}}\!\psi_1^{b_1}\ldots\psi_N^{b_N}\bigg)\\
&=\lr\a\de_{c_1+\ldots+c_N,n-1-l}\binom{N\!-\!3}{b_1,\ldots,b_N}
\qquad\forall~c_1,\ldots,c_N\ge0, 
\end{split}\end{equation*}
where $\ov\cM_{0,N}$ is the Deligne-Mumford moduli space of genus 0 curves
with $N$ marked points.
For $d\!\in\!\Z^+$,
this theorem is proved in  Section~\ref{GWbound_app}, 
with the hardest cases ($|\a|\!\le\!n$, $N\!\ge\!3$) deduced from Theorem~\ref{main_thm}.
It in particular implies that for every Calabi-Yau complete intersection threefold
$X_{\a}\!\subset\!\P^{n-1}$ ($|\a|\!=\!n$, $l\!=\!n\!-\!4$)
there exists $C\!\in\!\R^+$ such that 
$$\big|\blr{}_{0,d}^{X_{\a}}\big|\le \frac{N!C^N}{d^N}\cdot C^d
\qquad\forall~d,N\!\in\!\Z^+\,;$$
for $N\!\le\!2$, this bound also follows from the one-point mirror formulas.
According to~\cite{MaP}, Theorem~\ref{GWbound_thm} and 
\cite[Theorem~1]{Gi01} should imply such bounds in all genera via~\cite{MaP0};
in turn, the latter imply that generating functions for GW-invariants of
any genus have positive radii of convergence, as expected from physical considerations.

\begin{mythm}\label{GWeq0_thm}
Suppose $n,N\!\in\!\Z^+$ with $N\!\ge\!3$, $\a\!\in\!(\Z^+)^l$,
$X_{\a}\!\subset\!\P^{n-1}$ is a complete intersection of multi-degree~$\a$,
and $(b_s)_{s\in[N]}$ and $(c_s)_{s\in[N]}$ are $N$-tuples of nonnegative integers.
If there exists $S\!\subset\![N]$ such that $b_s\!+\!c_s\!<\!\nua$
for every $s\!\in\!S$ and $\displaystyle\sum_{s\in S}b_s\!>\!N\!-\!3$, then
$$\blr{\tau_{b_1}H^{c_1},\ldots,\tau_{b_N}H^{c_N}}_{0,d}^{X_{\a}}=0.$$
\end{mythm}

\noindent
This theorem is  an immediate consequence of Theorem~\ref{main_thm}; 
see Remark~\ref{GWeq0_rmk}.
Because of the condition on $b_s$, the assumptions of this theorem are never
satisfied if $|\nua|\!=\!0,1$ (Calabi-Yau and borderline Fano cases).
For the same reason, it is most useful if $|\a|\!=\!0$ (projective case). 
For example,
\BE{GWeq0proj_e}\blr{\underset{N-2}{\underbrace{\tau_bH^{n-b},\ldots,\tau_bH^{n-b}}},
\cdot,\cdot}_{0,d}^{\P^n}=0
\qquad \forall\,N\!\ge\!3,\,b\!=\!1,2,\ldots,n.\EE
The $\P^1$-case of~\e_ref{GWeq0proj_e} follows from the dilaton relation \cite[p527]{MirSym}. 
For $n\!\ge\!2$, $\tau_bH^{n-b}$ is not a divisor on $\ov\M_{0,N}(\P^n,d)$ 
and there appears to be no direct geometric reason for the vanishing~\e_ref{GWeq0proj_e}.\\

\noindent
Theorems~\ref{GWbound_thm} and~\ref{GWeq0_thm} suggest the following conjectures.
Theorem~\ref{GWbound_thm} establishes the first conjecture for 
projective complete intersections and $g\!=\!0$.
The approach of~\cite{MaP} should remove the genus restriction and establish
the dependence of~$C_{X,g}$ on~$g$ and even on~$X$.
Theorem~\ref{GWeq0_thm} establishes the second conjecture 
for projective complete intersections.

\begin{myconj}\label{GWbound_conj}
If $(X,\om)$ is a compact symplectic manifold, $g\!\in\!\Z$, and 
$H_1,\ldots,H_k\!\in\!H^*(X)$, then there exists $C_{X,g}\!\in\!\R^+$ such~that 
$$\big|\blr{b_1!\,\tau_{b_1}H_{c_1},\ldots,b_N!\,\tau_{b_N}H_{c_N}}_{g,\be}^X\big|
\le  N!\,C_{X,g}^{N+\lr{\om,\be}}
\qquad\forall\,\be\!\in\!H_2(X),\,N,b_s\!\ge\!0,\,c_s\!\in\![k].$$
\end{myconj}

\begin{myconj}\label{GWeq0_tonj}
Let $(X,\om)$ be a compact monotone symplectic manifold with minimal 
chern number~$\nu$.\footnote{Thus, $c_1(X)\!=\!\la[\om]\!\in\!H^2(X;\R)$ for
some $\la\!\in\!\R^+$ and $\nu$ is the minimal value of $c_1(X)$
on the homology classes representable by non-constant $J$-holomorphic maps $S^2\!\lra\!X$
for every $\om$-compatible almost complex structure on~$X$.}
If $N\!\ge\!3$, $(b_s)_{s\in[N]}$ and $(c_s)_{s\in[N]}$ are $N$-tuples of nonnegative integers,
and $H_s\!\in\!H^{2c_s}(X)$ for every $s\!\in\![N]$,
then
$$\blr{\tau_{b_1}H_1,\ldots,\tau_{b_N}H_N}_{0,\be}^X=0$$
if there exists $S\!\subset\![N]$ such that $b_s\!+\!c_s\!<\!\nua$
for every $s\!\in\!S$ and $\displaystyle\sum_{s\in S}b_s\!>\!N\!-\!3$.
\end{myconj}

\noindent
The genus 0 GW-invariants of a complete intersection $X_{\a}\!\subset\!\P^{n-1}$
are related to certain twisted GW-invariants of~$\P^{n-1}$.
Let
$$\xymatrix{\U \ar[d]^{\pi} \ar[r]^{\ev} & \P^{n-1} \\
\ov\M_{0,N}(\P^{n-1},d)}$$
be the universal curve over $\ov\M_{0,N}(\P^{n-1},d)$.
The GW-invariants of~\e_ref{GWXadfn_e} then satisfy
\BE{genus0_e}
\blr{\tau_{b_1}H^{c_1},\ldots,\tau_{b_N}H^{c_N}}_{0,d}^{X_{\a}}
= \int_{\ov\M_{0,N}(\P^{n-1},d)}
\prod_{k=1}^{k=l}e\big(\pi_*\ev^*\O_{\P^{n-1}}(a_k)\big)
\prod_{s=1}^{s=N}\!\!\big(\psi_s^{b_s}\ev_s^*H^{c_s}\big).\EE
Since the moduli space $\ov\M_{0,N}(\P^{n-1},d)$ is a smooth stack (orbifold) and
$$\bigoplus_{k=1}^{k=l}\pi_*\ev^* \O_{\P^{n-1}}(a_k) \lra\ov\M_{0,N}(\P^{n-1},d)$$
is a locally free sheaf, i.e.~the sheaf of sections of  
a vector orbi-bundle $\V_d$ over $\ov\M_{0,N}(\P^{n-1},d)$,
the right-hand side of~\e_ref{genus0_e} is well-defined;
its computation will be the main focus of this paper.
In~\e_ref{Zdfn_e}, we combine all GW-invariants \e_ref{genus0_e} with fixed $N$
into a generating function.
We show that for $N\!\ge\!3$ this generating function  is a certain transform
of the $N\!=\!1$ generating function.\\

\noindent
The general principles of this paper are valid for all~$\a$, but
the explicit expressions for the transforms apply only for $\nua\!\ge\!0$.
As the known 1-pointed (and 2-pointed) formulas, they involve the hypergeometric series 
\BE{Fdfn_e}
F(w,q)\equiv \sum_{d=0}^{\i}q^d 
\frac{w^{\nua d} \prod\limits_{k=1}\limits^{k=l}\prod\limits_{r=1}\limits^{a_kd}
(a_kw\!+\!r)}{\prod\limits_{r=1}\limits^{r=d}((w\!+\!r)^n-w^n)}\,.
~\footnote{For the purposes of Theorems~\ref{cy_thm} and~\ref{proj_thm}, 
the term $w^n$ can be dropped from the definition of~$F$.} 
\EE
This power series in $q$ is an element of $1\!+\!q\Q(w)[[q]]$ such that 
the coefficient of each power of~$q$ is holomorphic at $w\!=\!0$.
The subgroup 
$$\cP\subset 1+q\Q(w)[[q]]$$
of such power series is preserved by the operator
$$\bM\!:1+q\Q(w)[[q]]\lra 1+q\Q(w)[[q]], \qquad 
\{\bM H\}(w,q)=\bigg\{1+\frac{q}{w}\frac{\tnd}{\tnd q}\bigg\}
\bigg(\frac{H(w,q)}{H(0,q)}\bigg).$$
We define $I_c\!\in\!1\!+\!q\Q[[q]]$ for  $c\!=\!0,1,\ldots$ and $J\!\in\!q\Q[[q]]$ by
\BE{Icdfn_e}\begin{split}
I_c(q)&=\begin{cases} 1,&\hbox{if}~|\a|\!<\!n;\\
\{\bM^cF\}(0,q),&\hbox{if}~|\a|\!=\!n;\end{cases} \\
J(q)&= \begin{cases}
0,&\hbox{if}~|\a|\!\le\!n\!-\!2;\\
\a!q,&\hbox{if}~|\a|\!=\!n\!-\!1;\\
\frac{1}{I_0(q)}\sum\limits_{d=1}^{\i}q^d\left(\frac{\prod\limits_{k=1}^{k=l}(a_kd)!}{(d!)^n}
\left(\sum\limits_{k=1}^{k=l}\sum\limits_{r=d+1}^{a_kd}\frac{a_k}{r}\right)\right) 
,&\hbox{if}~|\a|\!=\!n.\end{cases}
\end{split}\EE
The power series $J(q)$ is the coefficient of $w$ in the power series expansion of 
$F(w,q)/I_0(q)$ at $w\!=\!0$; thus, $I_1(q)\!=\!1\!+\!q\frac{\tnd}{\tnd q}J(q)$ if $|\a|\!\neq\!n\!-\!1$.
Similarly to the 1- and 2-pointed cases, 
the explicit expressions of Theorem~\ref{main_thm} for generating functions
for $N\!\ge\!3$ involve the power series $I_0,I_1,\ldots,I_{n-l}$ and~$J$;
see Section~\ref{CY_subs} for some examples.\\

\noindent
The author would like to thank D.~Maulik, R.~Pandharipande, and V.~Shende for 
many enlightening discussions related to this paper.

\subsection{The Calabi-Yau case}
\label{CY_subs}

\noindent
If $|\a|\!=\!n$, $X_{\a}$ is a Calabi-Yau $(n\!-\!1\!-\!l)$-fold and
the virtual dimension of~$\ov\M_{0,N}(X_{\a},d)$ and 
of the space of $N$-marked rational curves in~$X_{\a}$,
$$\dim^{vir}\ov\M_{0,N}(X_{\a},d)= n\!-\!4\!-\!l+N,$$
is independent of $d$.
If $c_1,\ldots,c_N$ are nonnegative integers such that 
$$c_1+\ldots+c_N=n\!-\!4\!-\!l+N,$$
the corresponding genus~0 degree~$d$ GW-invariant of $X_{\a}$,
\BE{GWofCYdfn_e} N_d^{X_{\a}}(c_1,\ldots,c_N)\equiv 
\int_{[\ov\M_{0,N}(X_{\a},d)]^{vir}} 
\big(\ev_1^*H^{c_1}\big)\ldots\big(\ev_N^*H^{c_N}\big),\EE
is a rational number. 
These numbers define BPS states of $X_{\a}$ via \cite[(2)]{KP},
that are intended to be virtual counts of curves (rather than maps)
and are conjectured to be integer (see also Footnote~\ref{GWint_ftnt}).
For a sufficiently small value of the degree~$d$, 
the corresponding BPS number is known to be the number of 
rational degree~$d$ curves in a general complete intersection 
of multi-degree~$\a$ that
pass through general linear subspaces of codimensions $c_1,\ldots,c_N$.\\

\noindent
Theorem~\ref{main_thm} yields fairly simple closed formulas for 
the numbers~\e_ref{GWofCYdfn_e} with $N\!=\!3,4$.
Theorem~\ref{cy_thm} below follows immediately from  \e_ref{genus0_e},
\e_ref{Zdfn_e}, \e_ref{mainthm_e},  \e_ref{Depdfn_e}, \e_ref{Flpdfn_e},  
\e_ref{coeff3CY_e}, \e_ref{coeff4CY_e}, \e_ref{bAcases_e}, \e_ref{pdual_e},
\e_ref{pmoddfn_e}, \e_ref{Ldfn_e2}, \e_ref{Iprop_e1}, 
and~\e_ref{Iprop_e3}.\footnote{\e_ref{coeff3CY_e} is needed for \e_ref{pt3_e} only;
\e_ref{coeff4CY_e}, \e_ref{bAcases_e}, \e_ref{pdual_e},
\e_ref{pmoddfn_e}, and \e_ref{Iprop_e3} are needed for \e_ref{pt4_e} only}

\begin{mythm}\label{cy_thm}
Suppose $n\!\in\!\Z^+$, $X_{\a}\!\subset\!\P^{n-1}$ is a nonsingular Calabi-Yau 
complete intersection of multi-degree~$\a$,  
$I_c$ and $J$ are given by~\e_ref{Icdfn_e}, and $Q=q\cdot e^{J(q)}\in q\Q[[q]]$.
If $c_1,c_2,c_3$ are nonnegative integers such that 
$c_1\!+\!c_2\!+\!c_3\!=\!n\!-\!1\!-\!l$, then
\BE{pt3_e} \sum_{d=0}^{\i}Q^dN_d^{X_{\a}}(c_1,c_2,c_3)
=\frac{\lr{\a}}{(1\!-\!\a^{\a}q)I_0(q)^2
\prod\limits_{s=1}^{s=3}\prod\limits_{c=1}^{c=c_s}\!I_c(q)}\,.\EE
If $c_1,c_2,c_3,c_4$ are nonnegative integers such that 
$c_1\!+\!c_2\!+\!c_3\!+\!c_4\!=\!n\!-\!l$, then
\BE{pt4_e}\begin{split}
\sum_{d=0}^{\i}Q^dN_d^{X_{\a}}(c_1,c_2,c_3,c_4)
=\frac{\lr{\a}}{(1\!-\!\a^{\a}q)I_0^2(q)
\prod\limits_{s=1}^{s=4}\prod\limits_{c=1}^{c=c_s}\!I_c(q)} \Bigg\{
\frac{n\!-\!l\!-\!2c_4}{2}\bigg(\frac{\a^{\a}q}{1\!-\!\a^{\a}q}
-2\frac{I_0'(q)}{I_0(q)}\bigg)\quad&\\
+\sum_{s=1}^{s=4}\frac{\cS_{c_s}'(q)}{\cS_{c_s}(q)}
-\frac{\cS_{c_1+c_2}'(q)}{\cS_{c_1+c_2}(q)}
-\frac{\cS_{c_1+c_3}'(q)}{\cS_{c_1+c_3}(q)}
-\frac{\cS_{c_2+c_3}'(q)}{\cS_{c_2+c_3}(q)}\Bigg\}\,&,
\end{split}\EE  
where $'$ denotes the operator $q\frac{\tnd}{\tnd q}$ and
$\cS_c=I_1^{c-1}I_2^{c-2}\ldots I_c^0$.
\end{mythm}

\noindent
Since $J(q)\!\in\!q\Q[[q]]$, there exists 
$\ti{J}(Q)\in Q\Q[[Q]]$ such that $q=Qe^{\ti{J}(Q)}$.
Thus, the relations~\e_ref{pt3_e} and~\e_ref{pt4_e} determine
the numbers $N_d^{X_{\a}}(c_1,c_2,c_3)$ and $N_d^{X_{\a}}(c_1,c_2,c_3,c_4)$, 
respectively.
Since
\BE{cSsym_e} \frac{\cS_c'(q)}{\cS_c(q)}=\frac{\cS_{n-l-c}(q)}{\cS_{n-l-c}(q)}
-\frac{n\!-\!l\!-\!2c}{2}\bigg(\frac{\a^{\a}q}{1\!-\!\a^{\a}q}
-2\frac{I_0'(q)}{I_0(q)}\bigg)
\qquad\forall\,c\!=\!0,1,\ldots,n\!-\!l\EE
by~\e_ref{Iprop_e1}-\e_ref{Iprop_e3} and \e_ref{Ldfn_e2}, \e_ref{pt4_e} is equivalent to 
\BE{pt4_e2}\begin{split}
\sum_{d=0}^{\i}Q^dN_d^{X_{\a}}(c_1,c_2,c_3,c_4)
=\frac{\lr{\a}}{(1\!-\!\a^{\a}q)I_0^2(q)
\prod\limits_{s=1}^{s=4}\prod\limits_{c=1}^{c=c_s}\!I_c(q)} \Bigg\{
c_1\bigg(\frac{\a^{\a}q}{1\!-\!\a^{\a}q}-2\frac{I_0'(q)}{I_0(q)}\bigg)\qquad&\\
+\sum_{s=1}^{s=4}\frac{\cS_{c_s}'(q)}{\cS_{c_s}(q)}
-\frac{\cS_{c_1+c_2}'(q)}{\cS_{c_1+c_2}(q)}
-\frac{\cS_{c_1+c_3}'(q)}{\cS_{c_1+c_3}(q)}
-\frac{\cS_{c_1+c_4}'(q)}{\cS_{c_1+c_4}(q)}\Bigg\}\,&.
\end{split}\EE  
By~\e_ref{cSsym_e}, RHS of~\e_ref{pt4_e} is symmetric in $c_1,c_2,c_3,c_4$, as expected.
By~\e_ref{pt4_e2}, 
$$N_d^{X_{\a}}(c_1,c_2,c_3,c_4)=0 \qquad\hbox{if}\quad 0\in\{c_1,c_2,c_3,c_4\},$$
as expected.
By \e_ref{pt3_e}, \e_ref{Iprop_e1}, \e_ref{Iprop_e2}, and~\e_ref{Ldfn_e2}, 
$$N_d^{X_{\a}}(c_1,c_2,c_3)=
\begin{cases} 
\lr{\a},&\hbox{if}~d\!=\!0;\\ 0,&\hbox{if}~d\!>\!0;
\end{cases}  \qquad\hbox{if}\quad 0\in\{c_1,c_2,c_3\}.$$
Since $I_1(q)\!=\!1\!+\!q\frac{\tnd}{\tnd q}J(q)$, \e_ref{pt3_e} and~\e_ref{pt4_e2}
immediately give
$$dN_d(c_2,c_3,c_4)=N_d^{X_{\a}}(1,c_2,c_3,c_4),$$
as expected from the divisor relation \cite[p527]{MirSym}.
By the divisor relation, \e_ref{pt3_e}, \e_ref{Iprop_e1}, \e_ref{Iprop_e2}, and~\e_ref{Ldfn_e2}, 
\begin{alignat*}{1}
\lr{\a}+\sum_{d=1}^{\i}Q^ddN_d^{X_{\a}}(c_1,c_2)&=\lr{\a}\frac{I_{c_1+1}(q)}{I_1(q)}
\qquad\hbox{if}\quad c_1\!+\!c_2\!=\!n\!-\!2\!-\!l;\\
\lr{\a}+\sum_{d=1}^{\i}Q^dd^2N_d^{X_{\a}}(n\!-\!3\!-\!l)&=\lr{\a}\frac{I_2(q)}{I_1(q)}  \,;
\end{alignat*}
these identities are equations (1.5) and (1.6) in \cite{PoZ}.\\

\noindent
The first true cases of \e_ref{pt3_e} and \e_ref{pt4_e} occur for Calabi-Yau
6-folds and 7-folds:
$$\big(n,\a,c_1,c_2,c_3\big)=\big(8,(8),2,2,2\big)
\quad\hbox{and}\quad 
\big(n,\a,c_1,c_2,c_3,c_4\big)=\big(9,(9),2,2,2,2\big).$$
Tables~\ref{H222_tbl}-\ref{X10_tbl} show 
some low-degree BPS counts obtained from \e_ref{pt3_e} and~\e_ref{pt4_e} 
via \cite[(2)]{KP} for all complete intersections $X_{\a}\!\subset\!\P^{n-1}$,
with $n\!\le\!10$, of suitable dimensions, with $H^{c_i}$ indicating that one of 
the constraints is a general linear subspace of $\P^{n-1}$ of codimension~$c_i$.
All degree~1 and~2 numbers agree with the corresponding lines and conics counts
obtained via classical Schubert calculus computations 
(the 3-pointed numbers for hypersurfaces can be found in~\cite{Ka}, which
also describes the classical methods).
The degree~3 numbers for the hypersurfaces $X_8$ and~$X_9$ agree with~\cite{ES};
the remaining degree~3 numbers can presumably be confirmed by similar computations.
The most noteworthy is the agreement of the 4-pointed numbers, since these do not 
naturally arise in the physics view of mirror symmetry as originally presented 
in~\cite{GMP}.\footnote{This viewpoint is extended to arbitrary number of marked
points in~\cite{Ba}.}
There are currently no direct methods of counting curves of degree~4 or higher
on projective complete intersections; so the numbers in these degrees
obtained from~\e_ref{pt3_e} and~\e_ref{pt4_e} cannot be compared to 
anything at this time.
Finally, all BPS counts computed from \e_ref{pt3_e} and~\e_ref{pt4_e} 
via \cite[(2)]{KP} for $d\!\le\!100$ and all compete intersections 
$X_{\a}\!\subset\!\P^{n-1}$ with $n\!\le\!10$ are integers,
as the case should be.\footnote{The genus 0 GW-invariants of CYs
with at least 3 marked points are integers; see 
\cite[Section~7.3]{McS} and~\cite{RT}.
Since the GW-BPS transform of \cite[(2)]{KP} is always lower-triangular
with 1's on the diagonal and integers everywhere else if the number
of marked points is at least~3, it follows that the BPS numbers
are integers as well in this case.\label{GWint_ftnt}}

\begin{table}[H]
\begin{center}\begin{small}\begin{tabular}{c||c|c|c|c}
\hline
$d$& 1& 2& 3& 4\\
\hline
$X_8$& 59021312& 821654025830400& 12197109744970010814464& 186083410628492378226388631552\\
$X_{27}$& 19133912& 52069545843672& 150771900962422866056& 448721851648931529402358688\\
$X_{36}$& 9303984& 9656915909184& 10669913703022812624& 12119013327306237518117376\\
$X_{45}$& 6536800& 4306289363200& 3019921285456823200& 2177140100777199737600000\\
$X_{226}$& 7036416& 4323279882240& 2819049510852887040& 1889305224389886741405696\\
$X_{235}$& 3936600& 1091194853400& 321105896368043400& 97128823290992207460000\\
$X_{244}$& 3252224& 699998060544& 159942140236292096& 37565431180080918822912\\
$X_{334}$& 2589408& 396151430400& 64359976334347296& 10748812573405031454720\\
\hline
\end{tabular}\end{small}\end{center}
\caption{Low-degree genus 0  BPS numbers $(H^2,H^2,H^2)$ for some Calabi-Yau 6-folds}
\label{H222_tbl}
\end{table}

\begin{table}[H]
\begin{center}\begin{tabular}{c||c|c|c}
\hline
$d$& 1& 2& 3\\
\hline
$X_9$& 1579510449& 506855012110118424& 174633921378662035929052320\\
$X_{28}$& 466477056& 25865899481481216& 1538349758855955308748800\\ 
$X_{37}$& 200848599& 3684692607275358& 72513809257771729565550\\
$X_{46}$& 122812416& 1209608310822912& 12780622639872867502080\\
$X_{55}$& 104480625& 841277146035000& 7266883194629367785000\\
\hline
\end{tabular}\end{center}
\caption{Low-degree genus 0 BPS numbers $(H^2,H^2,H^3)$ for some Calabi-Yau 7-folds}
\label{H223_tbl}
\end{table}

\begin{table}[H]
\begin{center}\begin{tabular}{c||c|c|c}
\hline
$d$& 1& 2& 3\\
\hline
$X_9$& 2395066806& 1718927099008463268& 957208127608222375829677128\\
$X_{28}$& 702562304& 86939314932416512& 8348345278919524413816832\\
$X_{37}$& 302321376& 12364886269091538& 392695531026064094763648\\
$X_{46}$& 184771584& 4056318495977472& 69156291871338627290112\\
$X_{55}$& 157178750& 2820556380767500& 39310596116635041745000\\
\hline
\end{tabular}\end{center}
\caption{Low-degree genus 0 BPS numbers $(H^2,H^2,H^2,H^2)$ for some Calabi-Yau 7-folds}
\label{X9_tbl}
\end{table}

\begin{table}[H]
\begin{center}\begin{small}\begin{tabular}{c||c|c|c}
\hline
$d$& 1& 2& 3\\
\hline
$(H^2,H^3,H^3)$&  51415320000& 444475303469701680000& 4089048226644406809222184680000\\
$(H^2,H^2,H^4)$&  38922224000& 295035175517918176000& 2467449594491156931046837776000\\
$(H^2,H^2,H^2,H^3)$& 75062592000& 1394799570099498816000& 20109980886063766606715932224000\\
\hline
\end{tabular}\end{small}\end{center}
\caption{Low-degree genus 0 BPS numbers for $X_{10}$  in $\P^9$}
\label{X10_tbl}
\end{table}

\subsection{The projective case}
\label{proj_subs}

\noindent
Throughout the paper,  we denote by $\bar\Z^+$ the set of nonnegative integers.
If $N,d,n\!\in\!\bar\Z^+$, let 
\begin{equation*}\begin{split}
\cP_N(d)&=\big\{\bfd\!\equiv\!(d_1,d_2,\ldots,d_N)\!\in\!(\bar\Z^+)^N\!:~
\sum_{s=1}^{s=N}d_s\!=\!d\big\};\\
\cP_N^n(d)&=\big\{\bfd\!\equiv\!(d_1,d_2,\ldots,d_N)\!\in\!\cP_N(d)\!:~
d_s\!<\!n~\forall\,s\!\in\![N]\big\}.
\end{split}\end{equation*}
For any $\p\!\in\!\cP_N^n(d)$, let
$$\llparenthesis\p\rrparenthesis=\min\big\{p_s\!+\!1,n\!-\!1\!-\!p_s\!:\,s\!\in\![N]\big\}.$$
If $(c_s)_{s\in[N]}\in(\bar\Z^+)^N$, let
$$\bblr{\prod_{s=1}^{s=N}\frac{H^{c_s}}{\hb_s\!-\!\psi}}^{\!\!\P^{n-1}}_{\!\!0,d}
=\sum_{b_1,b_2,\ldots,b_N\ge0}
\Bigg(\prod_{s=1}^{s=N}\hb_s^{-1-b_s}\Bigg)
\blr{\tau_{b_1}H^{c_1},\ldots,\tau_{b_N}H^{c_N}}^{\P^{n-1}}_{0,d}\,.$$
Theorem~\ref{main_thm} yields fairly simple closed formulas for 
the genus~0 GW-invariants of projective spaces with~3 and~4 insertions.
Theorem~\ref{proj_thm} below follows immediately from 
\e_ref{Zdfn_e}, \e_ref{mainthm_e}, \e_ref{coeffnon0_e},
\e_ref{coeff3CY_e}, \e_ref{coeff4proj_e}, 
\e_ref{Depdfn_e}, \e_ref{Flpdfn_e}, \e_ref{bDdfn_e}, 
and~\e_ref{Fdfn_e}.\footnote{in this case, 
$\ntc_{p,s}^{(d)}\!=\!\de_{0,d}\de_{p,s}$ in \e_ref{Flpdfn_e} and~\e_ref{coeff3CY_e};
\e_ref{coeff4proj_e} is needed for the second identity 
in Theorem~\ref{proj_thm} only}

\begin{mythm}\label{proj_thm}
The 3- and 4-pointed Gromov-Witten invariants of $\P^{n-1}$ are described~by
\begin{equation*}\begin{split}
&\sum_{p_1,p_2,p_3\ge0}\!\!\!
 \bblr{\prod_{s=1}^{s=3}\frac{H^{n-1-p_s}}{\hb_s\!-\!\psi}}^{\!\!\P^{n-1}}_{\!\!0,d}
\!\!\!\!H_1^{p_1}H_2^{p_2}H_3^{p_3}
=\sum_{d'=0}^{d'=1}
\sum_{\begin{subarray}{c}\bfd\in\cP_3(d-d')\\ \p\in\cP_3^n((2-d')n-2)\end{subarray}}
\prod_{s=1}^{s=3} \frac{(H_s\!+\!d_s\hb_s)^{p_s}}
{\hb_s\!\!\prod\limits_{r=1}^{d_s}(H_s\!+\!r\hb_s)^n}\,,\\
&\sum_{p_1,p_2,p_3,p_4\ge0}\!\!\!
 \bblr{\prod_{s=1}^{s=4}\frac{H^{n-1-p_s}}{\hb_s\!-\!\psi}}^{\!\!\P^{n-1}}_{\!\!0,d}
\!\!\!\!H_1^{p_1}H_2^{p_2}H_3^{p_3}H_4^{p_4}\\
&\hspace{1in}=\Bigg\{
\sum_{\begin{subarray}{c}\bfd\in\cP_3(d-1)\\ \p\in\cP_4^n(2n-4)\end{subarray}}
\!\!\!\!\!\!\!\!\llparenthesis\p\rrparenthesis
~~~+~~~\bigg(\sum_{s=1}^{s=4}\hb_s^{-1}\bigg)\sum_{d'=0}^{d'=2}
\sum_{\begin{subarray}{c}\bfd\in\cP_4(d-d')\\ \p\in\cP_4^n((3-d')n-3)\end{subarray}}
\Bigg\}
\prod_{s=1}^{s=4} \frac{(H_s\!+\!d_s\hb_s)^{p_s}}
{\hb_s\!\!\prod\limits_{r=1}^{d_s}(H_s\!+\!r\hb_s)^n}\,;
\end{split}\end{equation*}
both identities hold modulo $H_s^n$ and as power series in $\hb_s^{-1}$.
\end{mythm}

\noindent
Since the $d\!=\!1$ Gromov-Witten invariant counts lines in $\P^{n-1}$,
the $d\!=\!1$ case of the 4-pointed formula in Theorem~\ref{proj_thm} gives 
$$\blr{\si_{c_1}\si_{c_2}\si_{c_3}\si_{c_4},\bG(2,n)}
=\min\big\{c_s\!+\!1,n\!-\!1\!-\!c_s\!:\,s\!=\!1,2,3,4\big\}
\quad
\hbox{if}~c_s\!\in\!\bar\Z^+,\,\sum_{s=1}^{s=4}c_s=2n\!-\!4,$$
where $\si_c$ is the usual codimension~$c$ Schubert cycle on $\bG(2,n)$.
As pointed out to the author by A.~Buch, this identity can be confirmed by applying 
Pieri's rule \cite[p203]{GH} to $\si_{c_1}\si_{c_2}$ and~$\si_{c_3}\si_{c_4}$
and counting pairs of dual cycles in its outputs.
The $d\!=\!2$ case of the 4-pointed formula gives
$$\blr{H^{c_1},H^{c_2},H^{c_3},H^{c_4}}_{0,2}^{\P^{n-1}}=0.$$
This is indeed as expected, since every conic lies in a $\P^2$ \cite[p177]{GH}
and no~$\P^2$ meets general linear subspaces of $\P^{n-1}$ of total codimension~$3n$
(the space of planes meeting the constraints is the intersection of Schubert 
cycles in $\bG(3,n)$ of total codimension $3n\!-\!8$ and is thus empty).

\section{Main Theorem}
\label{mainthm_subs}

\noindent
In addition to the notation introduced at the beginning of Section~\ref{proj_subs},
for any $m,l\!\in\!\bar\Z^+$ we define
$$\llfloor{m}\rrfloor=\big\{s\!\in\!\bar\Z^+\!:\,s\!<\!m\big\},
\qquad \llfloor{m}\rrfloor_l=\big\{s\!\in\!\llfloor{m}\rrfloor\!:\,s\!\ge\!l\big\}.$$
We denote by $\cP_m([N])$ the set of unordered partitions $\bfS\!\equiv\!\{S_i\}_{i\in[m]}$ of $[N]$ 
into nonempty subsets~$S_i$ such that one of them is~$\{N\}$.\footnote{More precisely, 
$\cP_m([N])$ consists of unordered partitions with a choice of some ordering for each of the partitions.}
If $\p$ is an $N$-tuple of integers, $S\!\subset\![N]$, and $p'\!\in\!\Z$,
let $\p|_S$ and $\p p'$ denote the $S$-tuple consisting of the elements of~$\p$
indexed by~$S$ and the $(N\!+\!1)$-tuple obtained by adjoining~$p'$ to~$\p$
at the end, respectively, and~set
$$|\p|_S\equiv \big|\p|_S\big|\equiv\sum_{s\in S}p_s\,.$$
If $R$ is a ring and $\un{x}\!=\!(x_1,\ldots,x_N)$ is a tuple of variables,
let
$$R\big[\un{x}\big]=R[x_1,\ldots,x_N]$$
be the ring of polynomials in $x_1,\ldots,x_N$.
If $\Phi\!\in\!R[[q]]$ and $d\!\in\!\Z$, let $\lrbr{\Phi}_{q;d}\!\in\!R$
denote the coefficient of~$q^d$ ($\lrbr{\Phi}_{q;d}\!\equiv\!0$ if $d\!<\!0$).\\

\noindent
Let $\P^{n-1}_N\!=\!(\P^{n-1})^N$.   For each $s\!=\!1,\ldots,N$, we~set
$$H_s=\pi_s^*H\in H^2\big(\P^{n-1}_N\big),$$
where $\pi_s\!:\P^{n-1}_N\!\lra\!\P^{n-1}$ is the projection onto 
the $s$-th coordinate.
Since $\ov\M_{0,N}(\P^{n-1},d)$ is smooth, there is a well-defined 
cohomology push-forward 
$$\ev_*\equiv \big\{\ev_1\!\times\!\ldots\!\ev_N\big\}_*\!:
H^*\big(\ov\M_{0,N}(\P^{n-1},d)\big)\lra H^*\big(\P^{n-1}_N\big).$$
With $\un\hb\!=\!(\hb_1,\ldots,\hb_N)$, 
$\un\hb^{-1}\!=\!(\hb_1^{-1},\ldots,\hb_N^{-1})$, 
and $\un{H}\!=\!(H_1,\ldots,H_N)$, let 
\BE{Zdfn_e} Z\big(\un\hb,\un{H},Q\big)= \sum_{d=0}^{\i}Q^d
\ev_*\Bigg\{\frac{e(\V_d)}{\prod_{s=1}^{s=N}(\hb_s\!-\!\psi_s)}\Bigg\}
\in H^*(\P^{n-1}_N)\big[\un\hb^{-1}\big]\big[\big[Q\big]\big].\EE
By~\e_ref{genus0_e},
this power series encodes all genus 0 GW-invariants of $X_{\a}$
with constrains that arise from~$\P^{n-1}$.
If $\b\!\equiv\!(b_1,\ldots,b_N)\in\!\Z^N$, let
$$\un\hb^{-\b}=\prod_{s=1}^{s=N}\!(\hb_s^{-1})^{b_s}\,. $$

\subsection{An asymptotic expansion}
\label{HG_subs}

\noindent
The power series $F$ defined by~\e_ref{Fdfn_e} admits an asymptotic expansion $w\!\lra\!\i$
which plays a central role in this paper and which we now describe.\\

\noindent
Define
\begin{alignat}{2}
\label{Ldfn_e0}
&L(q)\in 1+q\Q[[q]] &\qquad&\hbox{by}\qquad L(q)^n-\a^{\a}qL(q)^{|\a|}=1\in \Q[[q]]\,,\\
&\chi_0,\chi_1,\ldots,\chi_{|\a|}\in \Q &\qquad&\hbox{by}\qquad 
\prod_{k=1}^{k=l}\prod_{r=1}^{r=a_k}\big(a_kD+r\big)\equiv
\a^{\a}\sum_{i=0}^{i=|\a|}\chi_{|\a|-i}D^i\in\Z[D]\,.\notag
\end{alignat}
In the two extremal cases, \e_ref{Ldfn_e0} gives
\BE{Ldfn_e2}
L(q)=\begin{cases} (1\!+\!q)^{1/n},&\hbox{if}~|\a|\!=\!0;\\
(1-\a^{\a}q)^{-1/n},&\hbox{if}~|\a|\!=\!n.\end{cases}\EE
Setting $\chi_i\!\equiv\!0$ if $i\!<\!0$ or $i\!>\!|\a|$, we find that
\BE{xi_e}\begin{split} 
\chi_0=1\,, \qquad \chi_1=\frac{|\a|\!+\!l}{2}
.
\end{split}\EE
For $m,j\!\in\!\Z$, we define $\H_{m,j}\in\Q(u)$ recursively by
\BE{Hdfn_e}\begin{split}
\H_{m,j}&\equiv0 \quad\hbox{unless}~~0\le j\le m,\qquad
\H_{0,0}\equiv1;\\
\H_{m,j}(u)&\equiv\H_{m-1,j}(u)+
\frac{u\!-\!1}{|\a|\!+\!\nua u}
\bigg(nu\frac{\tnd}{\tnd u}+m\!-\!j\bigg)\H_{m-1,j-1}(u)
\quad\hbox{if}~ m\ge1,~0\le j\le m. 
\end{split}\EE
In particular, for $m\!\ge\!0$
\BE{Hlow_e}\begin{split} 
\H_{m,0}(u)&=1, \qquad \H_{m,1}(u)=\binom{m}{2}\frac{u\!-\!1}{|\a|\!+\!\nua u}
.
\end{split}\EE
Finally, we define differential operators $\fL_1,\ldots,\fL_n$ on $\Q[[q]]$ by
\BE{fLkdfn_e}
\fL_k=\sum_{i=0}^k\Bigg[\binom{n}{i}L^n\H_{n-i,k-i}(L^n) 
-(L^n\!-\!1)\sum_{r=0}^{k-i}\binom{|\a|\!-\!r}{i}\chi_r\,\H_{|\a|-i-r,k-i-r}(L^n)\Bigg]D^i\,,
\EE
where $D=q\frac{\tnd}{\tnd q}$.
By \e_ref{Hlow_e}, \e_ref{xi_e} and \e_ref{Ldfn_e0}, the first of these operator is 
\BE{L12_e}\begin{split}  
\fL_1&=\big(|\a|\!+\!\nua L^n\big)\Bigg\{D
+\frac{L^n\!-\!1}{|\a|\!+\!\nua L^n}
\bigg(\frac{\nua nL^n}{2(|\a|\!+\!\nua L^n)}-\frac{l\!+\!1}{2}\bigg)\Bigg\}\\
&=\big(|\a|\!+\!\nua L^n\big)\Bigg\{
\Bigg(\bigg(\frac{n}{|\a|\!+\!\nua L^n}\bigg)^{1/2}L^{\frac{l+1}{2}}\Bigg)
D\Bigg(\bigg(\frac{n}{|\a|\!+\!\nua L^n}\bigg)^{1/2}L^{\frac{l+1}{2}}\Bigg)^{-1}\Bigg\}.
\end{split}\EE

\begin{prp}\label{Fexp_prp}
The power series~$F$ of~\e_ref{Fdfn_e} admits an asymptotic expansion
\BE{F0exp_e} F(w,q)\sim e^{\xi(q)w}
\sum_{b=0}^{\i}\Phi_b(q)w^{-b}
\qquad\hbox{as}\quad w\lra\i,\EE
with $\xi,\Phi_1,\ldots\!\in\!q\Q[[q]]$ and $\Phi_0\!\in\!1\!+\!q\Q[[q]]$ 
determined by the first-order ODEs  
\begin{alignat}{1}
1+\xi'(q)=L(q)\,,\qquad
\label{PhiODE} \fL_1\Phi_b+\frac{1}{L}\fL_2\Phi_{b-1}
+\ldots +\frac{1}{L^{n-1}}\fL_n\Phi_{b+1-n}&=0,
\end{alignat}
where $\Phi_b\equiv0$ for $s<0$.
\end{prp}

\noindent
From \e_ref{L12_e} and \e_ref{PhiODE},  we immediately find that 
\BE{F0exp_e2}  
\Phi_0(q)=\bigg(\frac{n}{|\a|+\nua L(q)^n}\bigg)^{1/2}L(q)^{(l+1)/2}.\EE
In the extremal cases, this reduces to
\BE{F0exp_e2b}
\Phi_0(q)= \begin{cases} 
L(q)^{-(n-1)/2}=(1\!+\!q)^{-(n-1)/2n},&\hbox{if}~|\a|\!=\!0;\\
L(q)^{(l+1)/2}=(1\!-\!\a^{\a}q)^{-(l+1)/2n},&\hbox{if}~|\a|\!=\!n.
\end{cases}\EE
Proposition~\ref{Fexp_prp} in the $|\a|\!=\!n$ case is proved in \cite[Section~4]{Po2}, 
building up on the $\a\!=\!(n)$ case contained in Lemma~1.3 and Theorems~1.1, 1.2, and~1.4 in~\cite{ZaZ}.
The remaining cases are addressed in Appendix~\ref{HGprp_pf}.

\subsection{One- and two-pointed formulas}
\label{pt1and2}

\noindent
By the dilaton relation \cite[p527]{MirSym} and \cite[Theorems 9.5,10.7,11.8]{Gi2},
the  generating function~\e_ref{Zdfn_e} with $N\!=\!1$ and
the degree~0 term defined to be $\lr{\a}H_1^l\hb_1$ is given~by
\BE{Z1pt_e} Z(\hb_1,H_1,Q)=\lr{\a}H_1^le^{-J(q_1)w_1}\hb_1\frac{F(w_1,q_1)}{I_0(q_1)},
\quad\hbox{where}\quad
w_1=\frac{H_1}{\hb_1}\,,~~q_1e^{\de_{0\nua}J(q_1)}=\frac{Q}{H_1^{\nua}}\,.\EE
The generating function~\e_ref{Zdfn_e} for $N\!=\!2$ is given in 
\cite[Section~2]{PoZ} in terms of certain transforms of~$F$, 
which we describe next.\\

\noindent
Define
\begin{alignat}{1}
\label{bDdfn_e}
\bD\!:\Q(w)\big[\big[q\big]\big]&\lra \Q(w)\big[\big[q\big]\big]
\qquad\hbox{by}\qquad
\bD H(w,q)\equiv \left\{1+\frac{q}{w}\frac{\tnd}{\tnd q}\right\}H(w,q);\\
\label{F0dfn_e}
F_0(w,q)&\equiv\sum_{d=0}^{\i}
q^dw^{\nua d}\frac{\prod\limits_{k=1}^{k=l}\prod\limits_{r=0}^{a_kd-1}(a_kw+r)}
{\prod\limits_{r=1}^d((w\!+\!r)^n-w^n)}\in\cP;\\
\label{Fpdfn_e0}
F_p&\equiv\bD^pF_0=\bM^pF_0\in\cP
\qquad\forall~p=1,2,\ldots,l.
\end{alignat}
In particular, $F_l=F$. 
For $\nua\!>\!0$, we also define $\nc^{(d)}_{p,s},\ntc^{(d)}_{l+p,l+s}\in\!\Q$
with $p,s,d\!\ge\!0$ by 
\BE{ncrec_e}\begin{split}
\sum_{d=0}^{\i}\sum_{s=0}^{\i}\nc^{(d)}_{p,s}w^sq^d&=
\sum_{d=0}^{\i}q^d\frac{(w\!+\!d)^p\!\!\prod\limits_{k=1}^l\prod\limits_{r=1}^{a_kd}(a_kw+r)}
{\prod\limits_{r=1}^d(w+r)^n}
=w^p\bD^pF\big(w,q/w^{\nua}\big),\\
\sum_{\begin{subarray}{c}d_1+d_2=d\\ d_1,d_2\ge0\end{subarray}}
\sum_{r=0}^{p-\nua d_1}\ntc^{(d_1)}_{l+p,l+r}\nc^{(d_2)}_{r,s}&=\de_{0,d}\de_{p,s}
\qquad \forall\,d,s\!\in\!\bar\Z^+,\, s\!\le\!p\!-\!\nua d.
\end{split}\EE
Since $\nc^{(0)}_{p,s}=\de_{p,s}$, the second equation in \e_ref{ncrec_e} 
expresses $\ntc^{(d)}_{l+p,l+s}$
with $s\!\le\!p\!-\!\nua d$ in terms of the numbers $\ntc^{(d_1)}_{l+p,l+r}$
with $d_1\!<\!d$; the numbers $\ntc^{(d)}_{l+p,l+s}$ with $s\!>\!p\!-\!\nua d$
will not be needed.
In particular, $\ntc^{(0)}_{p,s}\!=\!\de_{p,s}$ for all $p,s\!\ge\!l$.
For $p\!>\!l$, set 
\BE{Flpdfn_e} F_p(w,q)=\begin{cases}
\bM^pF(w,q),&\hbox{if}~\nu_{\a}\!=\!0;\\
\sum\limits_{d=0}^{\i}\sum\limits_{s=0}^{p-l-\nua d}
\frac{\ntc^{(d)}_{p,l+s}\,q^d}{w^{p-l-\nu_{\a}d-s}}\bD^sF(w,q),&
\hbox{if}~\nu_{\a}\!>\!0.
\end{cases}\EE
Thus, $F_p\!\in\!\cP$ for all $p\!\in\!\Z^+$ by \e_ref{ncrec_e} and
$F_p=\bD^{p-l}F$ unless $p\!\ge\!l\!+\!\nu_{\a}$.
By \cite[Theorem~3]{PoZ}, the  generating function~\e_ref{Zdfn_e} with $N\!=\!2$
and the degree~0 term defined to~be the image~of
$\frac{\lr{\a}H_1^lH_2^l}{\hb_1+\hb_2}\frac{H_1^{n-l}\!-H_2^{n-l}}{H_1-H_2}$
is given~by 
\BE{Z2pt_e} Z(\hb_1,\hb_2,H_1,H_2,Q)=
\frac{\lr{\a}}{\hb_1\!+\!\hb_2}\,e^{-J(q_1)w_1-J(q_2)w_2}
\sum_{\begin{subarray}{c}p_1+p_2=n-1+l\\ p_1,p_2\ge l\end{subarray}}
\prod_{s=1}^{s=2}H_s^{p_s}\frac{F_{p_s}(w_s,q_s)}{I_{p_s-l}(q)}\,,\EE
where
$$w_s=\frac{H_s}{\hb_s}\,,\qquad q_se^{\de_{0\nua}J(q_s)}=\frac{Q}{H_s^{\nua}}\,.$$

\begin{rmk} The mismatch in the indexing of $I_*$ and $F_*$ is unfortunate for
the purposes of this section. 
However,
the choice of the indexing for the former is intended to simplify 
the explicit formulas for the Calabi-Yau CIs in Section~\ref{CY_subs},
while the choice of the indexing for the latter is intended to simplify 
some of the formulas in the proof of Theorem~\ref{main_thm} in the rest of the paper.
\end{rmk}

\subsection{Multi-pointed formulas}
\label{multipt_subs}

\noindent
Similarly to \e_ref{Z2pt_e}, the generating function~\e_ref{Zdfn_e} with $N\!\ge\!3$
is a linear combination of the $N$-fold products
\BE{Depdfn_e}\De_{\p}(\un\hb,\un{H},Q)
\equiv\prod_{s=1}^{s=N}
\frac{H^{p_s}}{\hb_s}\frac{F_{p_s}(w_s,q_s)}{\prod\limits_{r=p_s-l}^{n-l-1}\!\!\!\!I_r(q_s)},
\qquad\hbox{where}\quad 
w_s=\frac{H_s}{\hb_s}\,,~~ q_se^{\de_{0\nua}J(q_s)}=\frac{Q}{H_s^{\nua}}\,,\EE
with $\p\!=\!(p_1,p_2,\ldots,p_N)\!\in\!\nset^N$, $p_s\!\ge\!l$,
and with coefficients that are polynomials in $\hb_1^{-1},\ldots,\hb_N^{-1}$
of total degree at most $N\!-\!3$.
These coefficients are described below inductively using the coefficients 
$\ntc^{(d)}_{p,s}$ defined above and the asymptotic expansion of $F(w,q)$ 
provided by Proposition~\ref{Fexp_prp}.\\

\noindent
For $r\!<\!0$, we set $I_r(q)\!=\!1$.
By Proposition~\ref{Fexp_prp}, \e_ref{bDdfn_e}-\e_ref{Fpdfn_e0}, and \e_ref{Flpdfn_e},
there are asymptotic expansions
\BE{Fpexp_e} 
\frac{F_p(w,p)}{\prod\limits_{r=p-l}^{n-l-1}\!\!\!\!I_r(q)} 
\sim e^{\xi(q)w}\frac{I_0(q)}{L(q)^{\de_{0\nua}n}}
\sum_{b=0}^{\i}\Phi_{p;b}(q)w^{-b}
\qquad\hbox{as}\quad w\lra\i,\EE
with   $\Phi_{p;0}\!\in\!1\!+\!q\Q[[q]]$ and
$\Phi_{p;1},\Phi_{p;2}\ldots\!\in\!q\Q[[q]]$ described~by
\BE{Psipr_e}\begin{split}
\hat\Phi_{p+1;b}&=L\hat\Phi_{p;b}+\hat\Phi_{p;b-1}'
-\bigg(\sum_{r=0}^{r=p}\frac{I_r'}{I_r}\bigg)\hat\Phi_{p;b-1}
\quad\forall\,p\!\in\!\Z, \qquad \hat\Phi_{0;b}=\Phi_b\,,\\
\Phi_{p;b}(q)&=\begin{cases}
\sum\limits_{d=0}^{\i}\sum\limits_{s=0}^{p-\nua d}
\ntc^{(d)}_{p,s}q^d\hat\Phi_{s-l;b-(p-\nu_{\a}d-s)}(q),
&\hbox{if}~\nua\!>\!0,\\
\hat\Phi_{p-l;b}(q),&\hbox{if}~\nua\!=\!0,
\end{cases}
\end{split}\EE
where $\hat\Phi_{p;b}\equiv0$ if $b\!<\!0$, 
$\ntc^{(d)}_{p,s}\!\equiv\!\de_{0,d}\de_{p,s}$ unless $p,s\!\ge\!l$, and
$'$ denotes $q\frac{\tnd}{\tnd q}$ as before.
In the Calabi-Yau case, $\nu_{\a}\!=\!0$, the recursion~\e_ref{Psipr_e} for the coefficients
$\Phi_{p;b}\!=\!\hat\Phi_{p-l;b}$ in the asymptotic expansion~\e_ref{Fpexp_e} is obtained using
the first two identities in the following lemma.\footnote{The last identity 
in Lemma~\ref{Iprp_lmm} follows
from the first two; it was used in Section~\ref{CY_subs}.}

\begin{lmm}[{\cite[Proposition~4.4]{Po2}}]\label{Iprp_lmm}
If $|\a|\!=\!n$, the power series $I_p$ defined by \e_ref{Fdfn_e} and~\e_ref{Icdfn_e} satisfy
\begin{alignat}{1}
\label{Iprop_e1} I_{n-l-p}&=I_p  \qquad\forall\,p=0,1,\ldots,n\!-\!l;\\
\label{Iprop_e2} I_0I_1\cdots I_{n-l}&=L^n;\\
\label{Iprop_e3} I_0^{n-l}I_1^{n-l-1}\cdots I_{n-l}^0&=L^{\frac{n(n-l)}{2}}.
\end{alignat}
\end{lmm}

\noindent
For example, by \e_ref{Psipr_e},
\begin{gather}\label{Psi0_e}
\hat\Phi_{p;0}=L^p\Phi_0, ~~~ 
\hat\Phi_{p;1}=L^p\big(\Phi_1+\bA_p^{(1)}\Phi_0\big); ~~~
\frac{\Phi_{p;0}(q)}{\Phi_0(q)}=L(q)^{p-l}
\begin{cases}
\sum\limits_{d=0}^{\i}\frac{\ntc_{p,p-\nua d}^{(d)}q^d}{L(q)^{\nua d}},&
\hbox{if}~\nua\!>\!0;\\
1,&\hbox{if}~\nua\!=\!0;
\end{cases}\\
\label{Psi1_e}
\frac{\Phi_{p;1}(q)}{\Phi_0(q)}=L(q)^{p-l}
\begin{cases}
\sum\limits_{d=0}^{\i}\frac{\ntc_{p,p-\nua d}^{(d)}q^d}{L(q)^{\nua d}}
\bigg(\!\frac{\Phi_1(q)}{\Phi_0(q)}+\bA_{p-l-\nua d}^{(1)}(q)\!\!\bigg)
+\sum\limits_{d=0}^{\i}\frac{\ntc_{p,p-\nua d-1}^{(d)}q^d}{L(q)^{\nua d+1}},&
\hbox{if}~\nua\!>\!0,\\
\frac{\Phi_1(q)}{\Phi_0(q)}+\bA_{p-l}^{(1)}(q),&\hbox{if}~\nua\!=\!0,
\end{cases}
\end{gather}
where
\BE{bA1dfn_e}\ntc_{p,s}^{(d)}\equiv0~~\hbox{if}~s\!+\!\nua d>p,\qquad
\bA_p^{(1)}=L^{-1}\bigg(p\frac{\Phi_0'}{\Phi_0}
+\frac{p(p\!-\!1)}{2}\frac{L'}{L}-
\sum_{r=0}^{r=p}(p\!-\!r)\frac{I_r'}{I_r}\bigg).\EE
In the two extremal cases, \e_ref{F0exp_e2b} gives
\BE{bAcases_e}
\bA_p^{(1)}=L^{-1}\begin{cases}
-\frac{(n-p)p}{2}\frac{L'}{L},&\hbox{if}~|\a|\!=\!0;\\
\frac{(p+l)p}{2}\frac{L'}{L}
-\sum\limits_{r=0}^{r=p}(p\!-\!r)\frac{I_r'}{I_r},&\hbox{if}~|\a|\!=\!n.
\end{cases}\EE\\

\noindent
If $m\!\in\!\bar\Z^+$, $d,t\!\in\!\Z$, and 
$\c\!\equiv\!(c_r)_{r\in\Z^+}\!\in\!(\bar\Z^+)^{\i}$, let
\BE{Phimcdfn_e}\begin{split}
\cS_m(d,t)&=\big\{(\p,\b)\!\in\!\nset^m\!\times\!\Z^m\!:\,
|\p|\!-\!|\b|=n\!-\!2+(m\!-\!1)(l\!+\!2)+\nua d+nt\big\}\,,\\
\Phi_{m,\c}&=\frac{\Phi_0^2}{I_0^2}
(-1)^{m+|\c|}(m\!+\!|\c|)!
\prod_{r=1}^{\i}
\frac{1}{c_r!}\bigg(\frac{\Phi_r}{(r\!+\!1)!\,\Phi_0}\bigg)^{c_r}\,.
\end{split}\EE
For any $p,p'\!\in\!\nset$ and $b,b',d,t\!\in\!\Z$, let 
\BE{coeff2_e}\nc_{(pp'),(bb')}^{(d,t)}=
\begin{cases}
(-1)^{b}\LRbr{\frac{L(q)^{\de_{0\nua}(1+t)n}}{I_0(q)^2}}_{q;d}\,,
&\hbox{if}~b\!\ge\!0,\,b\!+\!b'\!=\!-1,\,p\!+\!p'\!+\!nt\!=\!n\!-\!1\!+\!l;\\
0,&\hbox{otherwise}.
\end{cases}\EE
For any $N$-tuples $\p\!\in\!\nset^N$, $\b\!\in\!(\bar\Z^+)^N$
with $N\!\ge\!3$ and $d,t\!\in\!\bar\Z^+$, we inductively define
\begin{equation}\label{coeffdfn_e}\begin{split}
\nc_{\p,\b}^{(d,t)}&=\!\!\!
\sum_{\begin{subarray}{c}m,d',t'\in\Z\\ m\ge3\end{subarray}}
\sum_{\begin{subarray}{c}\bfS\in\cP_m([N])\\
\bfd\in\cP_m(d-d')\\ \bft\in\cP_m(t-t')\\   (\p',\b')\in\cS_m(d',t')\end{subarray}}
\sum_{\begin{subarray}{c}
\b''\in(\bar{\Z}^+)^m\\ \c\in(\bar\Z^+)^{\i}\\ |\b''|+\|\c\|=m-3\end{subarray}}
\left(\bigg(\prod_{i=1}^{i=m} \nc_{\p|_{S_i}p_i',\b|_{S_i}b_i'}^{(d_i,t_i)}\bigg)\right.\\
&\hspace{2.2in}
\times\left.\LRbr{\!\Phi_{m-3,\c}(q)\!\prod_{i=1}^{i=m}\!
\frac{I_0(q)^2\Phi_{p_i';b_i'+1+b_i''}(q)}{b_i''!\,L(q)^{\de_{0\nua}n}\Phi_0(q)}\!}_{q;d'}\right),
\end{split}\end{equation}
where $\Phi_{p;b}\!\equiv\!0$ if $b\!<\!0$ and
$\nc_{\p|_{S_i}p_i',\b|_{S_i}b_i'}^{(d_i,t_i)}\!\equiv\!0$ if $b_i'\!<\!0$ and $|S_i|\!\ge\!2$.
By induction, 
\BE{coeffnon0_e}
\nc_{\p,\b}^{(d,t)}\neq0 \quad\Lra\quad
|\b|\!\le\!N\!-\!3, \quad |\p|-|\b|+\nua d+nt=(N\!-\!1)(n\!-\!2)+2+l\,.\EE
Since $\Phi_{m-3,\c},\Phi_{p_i';b_i'+1+b_i''}\!\in\!q\Q[[q]]$ unless $\c\!=\!\0$
and $b_i'\!+\!1\!+\!b_i''\!=\!0$, 
\BE{coeffnon0_e2}
\nc_{\p,\b}^{(0,t)}=\de_{|\p|+nt,(N-1)(n-1)+l}\binom{N\!-\!3}{\b}.\EE

\begin{thmlet}\label{main_thm}
Suppose $n,N\!\in\!\Z^+$, with $N\!\ge\!3$, and $\a\!\in\!(\Z^+)^l$ is such that $\|\a\|\!\le\!n$.
The generating function \e_ref{Zdfn_e} for $N$-pointed genus~0 GW-invariants of
a complete intersection $X_{\a}\!\subset\!\P^{n-1}$ 
is given~by 
\BE{mainthm_e} Z\big(\un\hb,\un{H},Q\big)= \lr{\a}
e^{-\sum\limits_{s=1}^{s=N}\!\!J(q_s)w_s}\!\!\!
\sum_{\p\in\nset_l^N}\sum_{\b\in(\bar\Z^+)^N}
\sum_{d=0}^{\i}\nc_{\p,\b}^{(d,0)}q^d\un\hb^{-\b}\De_{\p}(\un\hb,\un{H},Q),\EE
where $w_s\!=\!H_s/\hb_s$, $q_se^{\de_{0\nua J(q_s)}}=Q/H_s^{\nua}$, and
$qe^{\de_{0\nua J(q)}}=Q$.
\end{thmlet}

\noindent
We show in Section~\ref{equiv_sec} that this theorem follows from 
Theorem~\ref{equiv_thm}.\\ 

\noindent
By \e_ref{genus0_e}, \e_ref{Zdfn_e}, \e_ref{mainthm_e}, \e_ref{coeffnon0_e2}, and
\e_ref{Depdfn_e} 
$$\blr{\tau_{b_1}(H^{c_1}),\ldots,\tau_{b_N}(H^{c_N})}_{0,0}^{X_{\a}}
=\de_{|\c|,n-1-l}\lr\a\binom{N\!-\!3}{\b}$$
whenever $b_i,c_i\ge0$, as expected.\\

\noindent
By~\e_ref{coeff2_e}, for each $p\!\in\!\nset$, there exists 
a unique pair $(\hat{p},t_p)\!\in\!\nset\!\times\!\Z$ such that 
$\nc_{(p\hat{p}),(bb')}^{(d,t_p)}\neq0$ at least for some $b,b',d\!\in\!\Z$:
\BE{pdual_e}(\hat{p},t_p)=\begin{cases}
(n\!-\!1\!+\!l\!-\!p,0),&\hbox{if}~p\!\ge\!l;\\
(l\!-\!1\!-\!p,1),&\hbox{if}~p\!<\!l.
\end{cases}\EE
For any $\p\!\in\!\nset^N$, let 
\BE{tpdfn_e}t_{\p}=\sum_{s=1}^{s=N}t_{p_s}=\big|\big\{s\!\in\![N]\!:\,p_s\!<\!l\big\}\big|.\EE
We note that
\BE{ntceq0_e}
\ntc_{\hat{p},\hat{p}-\nua d}^{(d)}\neq0 \qquad\Lra\qquad
p+\nua d+(n\!-\!l)t_p\le n\!-\!1.\EE\\

\noindent
If $N\!\ge\!3$, $\p\!\in\!\nset^N$, $\b\!\in\!(\bar\Z^+)^N$, $d\!\in\!\bar\Z^+$, and $t\!\in\!\Z$
satisfy the last property in~\e_ref{coeffnon0_e} and $|\b|\!=\!N\!-\!3$,
the only nonzero terms in~\e_ref{coeffdfn_e} arise from $(m,\c)\!=\!(N,\0)$,
$p_i'\!=\!\hat{p}_i$, $b_i'\!=\!-\!1\!-\!b_i$, and $b_i''\!=\!b_i$.
If in addition $\nua\!\neq\!0$, 
by the last statement in~\e_ref{Psi0_e}, \e_ref{F0exp_e2},  and Lemma~\ref{Lbinom_lmm}
\BE{coeff3Fano_e}\begin{split}
\nc_{\p,\b}^{(d,t)}&=\binom{N\!-\!3}{\b}\sum_{d'=0}^{d'=d}
\ntc_{\hat\p}^{(d-d')}
\!\!\LRbr{\frac{nL(q)^{\nua d'+n(1+t-t_{\p})}}{|\a|+\nua L(q)^n}}_{q;d'}\\
&=\binom{N\!-\!3}{\b}\sum_{d'=0}^{d'=d}
\!\big(\a^{\a}\big)^{d'}\binom{d'\!+\!t\!-\!t_{\p}}{d'}\ntc_{\hat\p}^{(d-d')}\,,
\end{split}\EE
with the binomial coefficients defined as in~\e_ref{binomdfn_e} and
$$\ntc_{\hat\p}^{(d)}\equiv 
\sum_{\bfd\in\cP_N(d)}\!\!\!\!\!\ntc_{\hat\p}^{(\bfd)}\,,\qquad
\ntc_{\hat\p}^{(\bfd)}\equiv \prod_{s=1}^{s=N}
\!\!\ntc_{\hat{p}_s,\hat{p}_s-\nua d_s}^{(d_s)}\,.$$
If $\nua\!=\!0$, the last property in \e_ref{coeffnon0_e} imposes no restriction on~$d$.
In this case, we find~that 
\BE{coeff3CY_e}
\sum_{d=0}^{\i}\nc_{\p,\b}^{(d,t)}q^d
=\binom{N\!-\!3}{\b}\frac{L(q)^{n(1+t)}}{I_0(q)^2}\,.\EE
In the $\nua\!=\!0$ case, the last property in \e_ref{coeffnon0_e} forces $t\!\ge\!0$ 
and $t_{\p}\!=\!0$ if $t\!=\!0$, whenever $|\b|\!=\!N\!-\!3$.
The proof of Theorem~\ref{main_thm} implies that the right-hand side of 
\e_ref{coeff3Fano_e} also vanishes if either $t\!<\!0$ or $t\!=\!0$ and $t_{\p}\!>\!0$.
By~\e_ref{ntceq0_e} and the last property in~\e_ref{coeffnon0_e},
$$(n\!-\!l)(d'\!+\!t\!+\!1\!-\!t_{\p})-(|\a|\!-\!l)d'+lt-1\ge0$$
whenever the $d'$-summand in \e_ref{coeff3Fano_e} is nonzero;
this implies that
$$ 1 \le t_{\p}\!-\!t \le d'$$
whenever the triple product in \e_ref{coeff3Fano_e} is nonzero and 
either $t\!<\!0$ or $t\!=\!0$ and $t_{\p}\!>\!0$.
The explicit expression on the right-hand side of~\e_ref{coeff3Fano_e} thus
provides a direct reason for the vanishing of~$\nc_{\p,\b}^{(d,t)}$
in these cases.\\



\noindent
If $N\!=\!3$, the only possibly nonzero coefficients in \e_ref{mainthm_e} 
are $\nc_{\p,\0}^{(d,0)}$; these are described by~\e_ref{coeff3Fano_e} and~\e_ref{coeff3CY_e}. 
If $N\!=\!4$, the only possibly nonzero coefficients in \e_ref{mainthm_e} 
are $\nc_{\p,\0}^{(d,0)}$ and 
$$\nc_{\p,1000}^{(d,0)}=\nc_{\p,0100}^{(d,0)}=\nc_{\p,0010}^{(d,0)}=\nc_{\p,0001}^{(d,0)}\,,$$
with $\p\!\in\!\nset^4$; 
the latter set of coefficients is given by~\e_ref{coeff3Fano_e} and~\e_ref{coeff3CY_e}
whenever $\p$ satisfies the last property in~\e_ref{coeffnon0_e}
with $N\!=\!4$, $|\b|\!=\!1$, and $t\!=\!0$.
We next give a formula for the former set of coefficients.
For $p,d\!\in\!\Z$, define 
\begin{gather}
\lrbr{p}_d,\lrbr{\hat{p}}_d,\tau_d(p),t_d(p)\in\Z \qquad\hbox{by}\notag\\
\label{pmoddfn_e}
0\le \lrbr{p}_d,\lrbr{\hat{p}}_d\le n\!-\!1,\quad 
\lrbr{p}_d+\nua d+n\tau_d(p)=p,\quad \lrbr{p}_d+\lrbr{\hat{p}}_d+nt_d(p)=n\!-\!1\!+\!l.
\end{gather}
If $\p,\bfd\!\in\!\Z^4$, let
$$\Si_2(\p,\bfd)=\big\{p_1\!+\!p_2\!+\!\nua(d_1\!+\!d_2),
p_1\!+\!p_3\!+\!\nua(d_1\!+\!d_3),p_2\!+\!p_3\!+\!\nua(d_2\!+\!d_3)\big\}.$$
If $\nua\!=\!0$, $\lrbr{p}_d$, $\lrbr{\hat{p}}_d$, and $\Si_2(\p,\bfd)$ do not 
depend on~$d$ or~$\bfd$, and so we omit $d$ and~$\bfd$ from the notation in this case.
In the $\nua\!=\!0$ case,
a direct computation from~\e_ref{coeffdfn_e}, \e_ref{coeff3CY_e}, \e_ref{coeff2_e},
\e_ref{Psi0_e}, and~\e_ref{Psi1_e} gives
\BE{coeff4CY_e}\begin{split}
\sum_{d=0}^{\i}\nc_{\p,\0}^{(d,0)}q^d&=\frac{L(q)^{n+1}}{I_0(q)^2}
\Bigg\{\sum_{p'-1\in\Si_2(\p)}\hspace{-.2in}
\bA_{\lrbr{\hat{p}'}-l}^{(1)}(q)
-\sum_{s=1}^{s=4}\bA_{\hat{p}_s-l}^{(1)}(q)\Bigg\},
\end{split}\EE
whenever  
$\p$ satisfies the last property in~\e_ref{coeffnon0_e}
with $N\!=\!4$, $|\b|\!=\!0$, and $t\!=\!0$.\\

\noindent
If $\nua\!\neq\!0$, $d,d',p\!\in\!\bar\Z^+$, and $t\!=\!0,1$, let 
\begin{equation*}\begin{split}
\ntc_{p,d'}^{(d,t)}
&\equiv\LRbr{\frac{nL(q)^{\nua d'+n(1-t)}}{|\a|\!+\!\nua L(q)^n}   
\Big(\ntc_{p,p-\nua d}^{(d)}L(q)\bA^{(1)}_{p-l-\nua d}(q)\!+\!
\ntc_{p,p-\nua d-1}^{(d)}\Big)}_{q;d'}\,\\
&=\ntc_{p,p-\nua d}^{(d)}
\LRbr{\frac{nL(q)^{\nua d'+n(1-t)+1}}{|\a|\!+\!\nua L(q)^n}\bA^{(1)}_{p-l-\nua d}(q)}_{q;d'}
+\binom{d'\!-\!t}{d'}\big(\a^{\a}\big)^{d'}\ntc_{p,p-\nua d-1}^{(d)}\,;
\end{split}\end{equation*}
the equality above holds by Lemma~\ref{Lbinom_lmm}.
On the other hand, by the second equation in~\e_ref{bA1dfn_e},
\e_ref{F0exp_e2},  \e_ref{Ldfn_e0}, and Corollary~\ref{Lbinom_crl}, 
\begin{equation*}\begin{split}
&\LRbr{\frac{nL(q)^{\nua d'+n(1-t)+1}}{|\a|\!+\!\nua L(q)^n}\bA^{(1)}_{p-l}(q)}_{q;d'}\\
&\hspace{1in}=\frac{p\!-\!l}{2}\bigg(\frac{\a^{\a}}{n}\bigg)^{\!\!d'}\!\Bigg(\!
d'|\a|^{d'}-(n\!-\!p)\!\!\!\!\!\!\!\!
\sum_{\begin{subarray}{c}d_1+d_2=d'-1\\ d_1,d_2\ge0\end{subarray}}
\!\!\!\!\!\!\!\!\!|\a|^{d_1}(n\!-\!\nua t)^{d_2}
-(d'\!-\!1\!+\!\de_{0,d'})tp|\a|^{d'-1}\!\Bigg)
\end{split}\end{equation*}
whenever $t\!=\!0,1$.
In particular, $\ntc_{p,0}^{(0,t)}\!=\!0$.
For $d,p\!\in\!\bar\Z^+$ such that $p\!\le\!2n\!-\!1$, let
$$\ntC_p^{(d)}=\sum_{\bfd\in\cP_4(d)}
\big(\a^{\a}\big)^{d_3}\binom{d_3\!+\!\tau_{d_2+d_3}(p)\!-\!t_{d_2+d_3}(p)}{d_3}
\ntc_{\lrbr{\hat{p}}_{d_2+d_3},\lrbr{\hat{p}}_{d_2+d_3}-\nua d_2}^{(d_2)}
\ntc_{\lrbr{p}_{d_2+d_3},d_4}^{(d_1,\tau_{d_2+d_3}(p))}\,.$$
Since $0\!\le\!p\!\le\!2n\!-\!1$,
$$\binom{d_3\!+\!\tau_{d_2+d_3}(p)\!-\!t_{d_2+d_3}(p)}{d_3}
\ntc_{\lrbr{\hat{p}}_{d_2+d_3},\lrbr{\hat{p}}_{d_2+d_3}-\nua d_2}^{(d_2)}\neq0
\qquad\Lra\qquad \tau_{d_2+d_3}(p)\in\{0,1\}$$
by \e_ref{ntceq0_e}, and so $\ntC_p^{(d)}$ is well-defined.
For example, $\ntC_p^{(0)}\!=\!0$.
If $\nua\!\neq\!0$, $t_{\p}\!=\!0$, and $\p$ satisfies the last property in~\e_ref{coeffnon0_e}
with $N\!=\!4$, $|\b|\!=\!0$, and $t\!=\!0$, then
\BE{coeff4Fano_e}\begin{split}
\nc_{\p,\0}^{(d,0)}&=\sum_{d'=0}^{d'=d}
\sum_{\bfd\in\cP_4(d-d')}\left(
\sum_{2n-2+l-p'\in\Si_2(\p,\bfd)}
\hspace{-.4in} \ntC_{p'}^{(d')}\ntc_{\hat\p}^{(\bfd)}
-\sum_{r=1}^{r=4}\bigg(\prod_{s\in[4]-r}
\!\!\!\!\!\ntc_{\hat{p}_s,\hat{p}_s-\nua d_s}^{(d_s)}\bigg)\ntc_{\hat{p}_r,d'}^{(d_r,0)}\right)\,.
\end{split}\EE
This is obtained by a direct computation from \e_ref{coeffdfn_e}, \e_ref{coeff3Fano_e},
\e_ref{coeff2_e}, \e_ref{Psi0_e}, and~\e_ref{Psi1_e}
except the vanishing of the coefficient of $\Phi_1(q)$ follows from
Corollary~\ref{sumsplit_crl}.
If $\ntc_{\hat\p}^{(\bfd)}\!\neq\!0$ in \e_ref{coeff4Fano_e}, then
$$l\le p_s\!+\!\nua d_s \le n\!-\!1 \qquad\forall~s\!\in\![4]$$
by the assumption that $t_{\p}\!=\!0$ and \e_ref{ntceq0_e}, and so
$$l\le p' \le 2n\!-\!2\!-\!l
\qquad\hbox{if}\quad 2n\!-\!2\!+\!l\!-\!p'\in\Si_2(\p,\bfd)\,;$$
thus, the right-hand side of \e_ref{coeff4Fano_e}  is well-defined.
In the case of a projective space, $\a\!=\!\eset$, the above formulas give
\begin{gather}
\ntc_{p,d'}^{(d,t)}=\begin{cases}
-\frac{p(n-p)}{2n}\,,&\hbox{if}~d\!=\!0,\,d'\!>\!0,\,t\!=\!0;\\
-\frac{p(n-p)}{2n}\,,&\hbox{if}~(d,d',t)\!=\!(0,1,1);\\
0,&\hbox{otherwise};\end{cases}  \qquad
\ntC_p^{(d)}=\begin{cases}
-\frac{\lrbr{p}_0(n-\lrbr{p}_0)}{2n}\,,&\hbox{if}~d\!>\!0;\\
0,&\hbox{if}~d\!=\!0;
\end{cases}\notag\\
\label{coeff4proj_e}
\nc_{\p,\0}^{(d,0)}=
\begin{cases}
0,&\hbox{if}~d=0,2;\\
\min\{p_s\!+\!1,n\!-\!1\!-\!p_s\},&\hbox{if}~d\!=\!1;
\end{cases}
\end{gather}
the last statement holds under the assumption that $|\p|+\!nd\!=\!3n\!-\!4$.\\

\noindent
The $N$-pointed formula of Theorem~\ref{main_thm} takes the simplest form in the two extremal
cases, $\nua\!=\!0$ (Calabi-Yau) and $\nua\!=\!n$ (projective space),
as $\ntc_{p,s}^{(d)}\!=\!\de_{0,d}\de_{p,s}$ in these two cases.
However, it is also straightforward to compute all the relevant coefficients 
in the intermediate cases.
For example, for a cubic threefold $X_3\!\subset\!\P^4$, the only non-trivial 
coefficients~$\ntc_{p,s}^{(d)}$ are
$$\ntc_{3,1}^{(1)}=\ntc_{4,1}^{(1)}=-6, \qquad \ntc_{4,2}^{(1)}=-21,$$
as computed in \cite[Section~2]{PoZ}.\footnote{In this paper, the subscripts
on $\ntc$ are shifted up by $l$ from \cite{PoZ}.}
From this, \e_ref{coeff3Fano_e}, and~\e_ref{coeff4Fano_e}, 
we find that the only nonzero coefficients 
in the $N\!=\!3,4$ cases of~\e_ref{mainthm_e} 
with $d\!\in\!\Z^+$ and $\b\!=\!\0$ are 
\begin{alignat*}{6} 
\nc_{133,\0}^{(1,0)}&=6, &\quad \nc_{223,\0}^{(1,0)}&=15,&\quad
\nc_{113,\0}^{(2,0)}&=36, &\quad \nc_{122,\0}^{(2,0)}&=126, &\quad
\nc_{111,\0}^{(3,0)}&=216,\\
\nc_{1333,\0}^{(1,0)}&=6, &\quad \nc_{2233,\0}^{(1,0)}&=15,&\quad
\nc_{1133,\0}^{(2,0)}&=72, &\quad \nc_{1223,\0}^{(2,0)}&=252, &\quad
\nc_{1113,\0}^{(3,0)}&=648,\\
&& && &&
\nc_{2222,\0}^{(2,0)}&=729, &\quad \nc_{1122,\0}^{(3,0)}&=2484, &\quad
\nc_{1111,\0}^{(4,0)}&=5184,
\end{alignat*}
up to the permutations of the first three subscripts.
From~\e_ref{mainthm_e}, we then find that 
\begin{alignat*}{2}
\lr{H^3,H,H}_{0,1}^{X_3}=\lr{H^3,H,H,H}_{0,1}^{X_3}&=18,&\quad  
\lr{H^2,H^2,H}_{0,1}^{X_3}=\lr{H^2,H^2,H,H}_{0,1}^{X_3}&=45, \\
\lr{H^3,H^3,H}_{0,2}^{X_3}=\frac12\lr{H^3,H^3,H,H}_{0,2}^{X_3}&=108, &\quad \lr{H^3,H^2,H^2}_{0,2}^{X_3}=\frac12\lr{H^3,H^2,H^2,H}_{0,2}^{X_3}&=378,\\ 
\lr{H^2,H^2,H^2,H^2}_{0,2}^{X_3}&=2187,&\quad
\lr{H^3,H^3,H^3}_{0,3}^{X_3}=\frac13\lr{H^3,H^3,H^3,H}_{0,3}^{X_3}&=648,\\
\lr{H^3,H^3,H^2,H^2}_{0,3}^{X_3}&=7452,&\quad
\lr{H^3,H^3,H^3,H^3}_{0,4}^{X_3}&=15552.
\end{alignat*}
These conclusions are consistent with the divisor relation.
The above invariants are enumerative at least for $d\!=\!1,2,3$.
The degree~1 and~2 numbers agree with the classical Schubert calculus computations
on~$G(2,5)$ and~$G(3,5)$, respectively.
The approach of~\cite{ES} can be used to test the two degree~3 numbers.\\

\noindent
Based on \e_ref{coeffdfn_e}, the coefficient $\nc_{\p,\b}^{(d,0)}$
in \e_ref{mainthm_e} with $\p\!\in\!\nset^N$ involves the power series
$\Phi_r$ of Proposition~\ref{Fexp_prp} with $r\!=\!0,1,\ldots,N\!-\!3\!-\!|\b|$ only.
By \e_ref{coeff4CY_e} and \e_ref{coeff4Fano_e}, only the power series~$\Phi_0$
enters in the $N\!=\!4$ case.
For $N\!=\!5$, the power series $\Phi_1$ and~$\Phi_2$ do enter in the final expression
for $\nc_{\p,\0}^{(d,0)}$.
However, at least for $\a\!=\!(n)$, i.e.~when $X_{\a}$ is a Calabi-Yau hypersurface,
$\Phi_2$ cancels with $\Phi_1^2/\Phi_0$ (these two power series are equal in this case).\\

\subsection{Alternative description of the structure constants}
\label{graph_subs}

\noindent
We now describe the constants $\nc_{\p,\b}^{(d,0)}$ defined above as sums over $N$-marked 
trivalent trees.\footnote{The constants $\nc_{\p,\b}^{(d,t)}$ with $t\!>\!0$ can be described
in the same way as well, but are not needed in this approach.}
It is fairly straightforward to see that the two descriptions are equivalent;
this also follows from the two variations of the main localization computation
in Section~\ref{equivpf_sec}.\\

\noindent
A \sf{graph} consists of a set~$\Ver$ of \sf{vertices} and 
a collection $\Edg$ of \sf{edges}, i.e.~of two-element subsets of~$\Ver$.
In Figure~\ref{trivalent_fig}, the vertices are represented by dots,
while each edge $\{v_1,v_2\}$ is shown as the line segment between $v_1$ and~$v_2$.
For such a graph~$\Ga$ and $v\!\in\!\Ver$, let
$$\tnE_v(\Ga)=\big\{e\!\in\!\Edg\!:\,v\!\in\!e\big\}$$
be the set of edges leaving~$v$.
A graph $(\Ver,\Edg)$ is a \sf{tree} if it contains no~loops, 
i.e.~the set $\Edg$ contains no $m$-element subset of the~form
$$\big\{\{v_1,v_2\},\{v_2,v_3\},\ldots,\{v_m,v_1\}\big\},
\qquad v_1,\ldots,v_m\!\in\!\Ver,~ m\!\ge\!3;$$
all four graphs in Figure~\ref{trivalent_fig} are trees.
An \sf{$N$-marked graph} is a tuple $\Ga\!=\!(\Ver,\Edg;\eta)$,
where $(\Ver,\Edg)$ is a graph and $\eta\!:[N]\!\lra\!\Ver$ is a map;
in Figure~\ref{trivalent_fig},
the elements of the~set $[N]\!=\![4]$ are shown in bold face and 
are linked by line segments to their images under~$\eta$.
An $n$-marked graph $\Ga\!=\!(\Ver,\Edg;\eta)$ is called \sf{trivalent} if 
$$m_v\equiv\val_{\Ga}(v)-3\equiv \big|\tnE_{\Ga}(v)\big|+\big|\eta^{-1}(v)\big|-3\ge0$$
for every vertex $v\!\in\!\Ver$. 
There is a unique trivalent 3-marked tree; 
the four trivalent 4-marked trees are shown in Figure~\ref{trivalent_fig}.
For any $N$-marked tree,
\BE{mvsum_e}\sum_{v\in\Ver}\!\!m_v+|\Edg|=N-3.\EE\\

\begin{figure}
\begin{pspicture}(6,-1.1)(10,2)
\psset{unit=.3cm}
\pscircle*(25,0){.3}
\psline[linewidth=.04](25,0)(23,2)\rput(22.5,2){\smsize{$\bf 2$}}
\psline[linewidth=.04](25,0)(23,-2)\rput(22.5,-2){\smsize{$\bf 1$}}
\psline[linewidth=.04](25,0)(27,2)\rput(27.5,2){\smsize{$\bf 3$}}
\psline[linewidth=.04](25,0)(27,-2)\rput(27.5,-2){\smsize{$\bf 4$}}
\pscircle*(40,0){.3}
\psline[linewidth=.04](40,0)(36,0)
\pscircle*(36,0){.3}
\psline[linewidth=.04](36,0)(34,2)\rput(33.5,2){\smsize{$\bf 2$}}
\psline[linewidth=.04](36,0)(34,-2)\rput(33.5,-2){\smsize{$\bf 1$}}
\psline[linewidth=.04](40,0)(42,2)\rput(42.5,2){\smsize{$\bf 3$}}
\psline[linewidth=.04](40,0)(42,-2)\rput(42.5,-2){\smsize{$\bf 4$}}
\pscircle*(55,0){.3}
\psline[linewidth=.04](55,0)(51,0)
\pscircle*(51,0){.3}
\psline[linewidth=.04](51,0)(49,2)\rput(48.5,2){\smsize{$\bf 3$}}
\psline[linewidth=.04](51,0)(49,-2)\rput(48.5,-2){\smsize{$\bf 1$}}
\psline[linewidth=.04](55,0)(57,2)\rput(57.5,2){\smsize{$\bf 2$}}
\psline[linewidth=.04](55,0)(57,-2)\rput(57.5,-2){\smsize{$\bf 4$}}
\pscircle*(70,0){.3}
\psline[linewidth=.04](70,0)(66,0)
\pscircle*(66,0){.3}
\psline[linewidth=.04](66,0)(64,2)\rput(63.5,2){\smsize{$\bf 4$}}
\psline[linewidth=.04](66,0)(64,-2)\rput(63.5,-2){\smsize{$\bf 1$}}
\psline[linewidth=.04](70,0)(72,2)\rput(72.5,2){\smsize{$\bf 3$}}
\psline[linewidth=.04](70,0)(72,-2)\rput(72.5,-2){\smsize{$\bf 2$}}
\end{pspicture}
\caption{The trivalent 4-marked trees}
\label{trivalent_fig}
\end{figure}
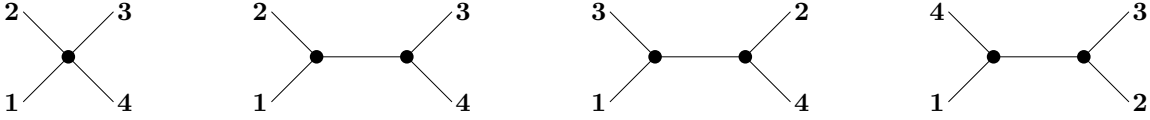

\noindent
We will call a partial ordering $\prec$ on a set $\Ver$ \sf{linear} if
for any pair of distinct incomparable elements $v_1,v_2\!\in\!\Ver$
there exists a third element $v\!\in\!\Ver$ such that $v\!\prec\!v_1,v_2$.
A finite linearly ordered set~$\Ver$ has a unique minimal element $v_0\!\in\!\Ver$.
For each trivalent $N$-marked tree $\Ga\!=\!(\Ver,\Edg;\eta)$, 
we fix a partial ordering $\prec$ on~$\Ver$ so that if $v\!\prec\!v'$,
then there exist
$$v_1,\ldots,v_m\!\in\!\Ver  \quad\hbox{s.t.}\quad
v_{i-1}\!\prec\!v_i,~~
\{v_{i-1},v_i\}\in\Edg~~\forall~i\!\in\![m\!+\!1],
~~\hbox{where}~~v_0\!\equiv\!v,~v_{m+1}\!\equiv\!v'\,.~
\footnote{Such a partial ordering is determined by the minimal vertex $v_0$,
which could be taken to be $\eta(N)$, for example.}$$
For every edge $e\!\in\!\Edg$, let $v_e^-,v_e^+\!\in\!\Ver$ be the elements of $e\!\subset\!\Ver$ 
with $v^-_e\!\prec\!v^+_e$.
For each $v\!\in\!\Ver$, let
$$\tnE_{\Ga}^-(v)=\big\{e\!\in\!\Edg\!:\,v_e^-\!=\!v\big\}$$
be the set of edges descending to $v$.
If $v\!\neq\!v_0$, let $e_v\!\in\!\Edg$ be the unique edge descending from~$v$.\\

\noindent
Let $(\p,\b,d)\!\in\!\nset_l^N\!\times\!(\bar\Z^+)^N\!\times\!\bar\Z^+$
be a tuple satisfying the two properties on the right-hand side of~\e_ref{coeffnon0_e} with $t\!=\!0$,
$\Ga\!=\!(\Ver,\Edg;\eta)$ be a trivalent $N$-marked tree, and
$$\bfd\!\equiv\!(d_v)_{v\in\Ver}\in\cP_{\Ga}(d)\equiv\cP_{\Ver}(d)$$ 
be a partition of $d$ into nonnegative integers.
We denote~by
$$\cS_{\Ga}(\p,\b,\bfd)\subset\nset^{\Edg}\times(\bar\Z^+)^{\Edg}\times\Z^{\Ver}$$
the subset of triples $(\p',\b',\bft)$ such that
\BE{sumcond_e}\sum_{s\in\eta^{-1}(v)}\!\!\!\!\!\big(\hat{p}_s\!+\!b_s\big)
+\sum_{e\in\tnE_{\Ga}^-(v)}\!\!\!\!\big(\hat{p}_e'\!-\!1\!-\!b_e'\big)
+(p_{e_v}'\!+\!b_{e_v}'\big)
=n\!-\!3+(m_v\!+\!2)(l\!+\!1)+\nua d_v+nt_v\EE
for all $v\!\in\!\Ver$, where $\hat{p}$ is as in~\e_ref{pdual_e}
and we set $p_{e_v}'\!+\!b_{e_v}\!\equiv\!0$ if $v\!=\!v_0$.
Each choice of~$\b'$ determines~$\p'$ and~$\bft$ uniquely by solving \e_ref{sumcond_e} 
for $p_v$ and $t_v$ starting with maximal elements of~$\Ver$ and moving down;
the equation for $v\!=\!v_0$ will then be automatically solvable for~$t_v$
because of the last property in~\e_ref{coeffnon0_e}. 
Furthermore, for every $(\p',\b',\bft)\in\cS_{\Ga}(\p,\b,\bfd)$
$$t_{\p'}+\sum_{v\in\Ver}\!\!t_v=0\,,$$
with $t_{\p'}$ as in~\e_ref{tpdfn_e}.\\

\noindent
If $(\p,\b)\!\in\nset_l^N\!\times\!(\bar\Z^+)^N$ and $d\!\in\!\bar\Z^+$ satisfy 
the last property in~\e_ref{coeffnon0_e} with $t\!=\!0$, set
\BE{ncCdfn2_e}\begin{split}
&\nc_{\p,\b}^{(d,0)} =\sum_{\Ga}\!\!\!\!\!
\sum_{\begin{subarray}{c}\bfd\in\cP_{\Ga}(d)\\ 
(\p',\b',\bft)\in\cS_{\Ga}(\p,\b,\bfd)\end{subarray}}
\!\!\!\!\!\!\!\!\!\!\!\!\!\!(-1)^{|\b|+|\b'|}\!\!\!\!\!\!\!\!\!\!\!\!\!\!\!\!
\sum_{\begin{subarray}{c} \b''\in(\bar{\Z}^+)^N;\b^-,\b^+\in(\bar{\Z}^+)^{\Edg}\\ 
(\c_v)_{v\in\Ver}\in((\bar\Z^+)^{\i})^{\Ver}\\ 
|\b''|_{\eta^{-1}(v)}+|\b^-|_{\tnE_{\Ga}^-(v)}+b_{e_v}^++\|\c_v\|=m_v\end{subarray}}
\!\!\!\!\!\!\!\!
\prod_{v\in\Ver}\!\Bigg\llbracket \Phi_{m_v,\c_v}(q)\!\\
&\hspace{.6in}\times\!\!\!\!
\prod_{s\in\eta^{-1}(v)}\!\!\!\!
\frac{\Phi_{\hat{p}_s;b_s''-b_s}(q)}{b_s''!\,\Phi_0(q)}
\times\!\!\!\!\!\!
\prod_{e\in\tnE_{\Ga}^-(v)}\!\!\!\!\!
\frac{L(q)^{\de_{0\nua}nt_{p_e'}}\Phi_{\hat{p}_e';b_e^-+1+b_e'}(q)}{b_e^-!\,\Phi_0(q)}
\times \!
\frac{I_0(q)^2\Phi_{p_{e_v}';b_{e_v}^+-b_{e_v}'}(q)}{b_{e_v}^+!\,L(q)^{\de_{0\nua}n}\Phi_0(q)}
\!\Bigg\rrbracket_{q;d_v},
\end{split}\EE
where $b_{e_{v_0}}^+\!\equiv\!0$,
the last fraction is defined to be 1 for $v\!=\!v_0$, 
and the outer sum is taken over all trivalent $N$-marked trees 
$\Ga\!=\!(\Ver,\Edg;\eta)$.
For example, the contribution of the one-vertex $N$-marked tree is 
\begin{equation*}\begin{split}
&(-1)^{|\b|}
\sum_{\begin{subarray}{c} \b''\in(\bar{\Z}^+)^N,\c\in(\bar\Z^+)^{\i}\\ 
|\b''|+\|\c\|=N-3\end{subarray}}
\Bigg\llbracket \Phi_{N-3,\c}(q)
\prod_{s=1}^{s=N}\!\frac{\Phi_{\hat{p}_s;b_s''-b_s}(q)}{b_s''!\,\Phi_0(q)}
\Bigg\rrbracket_{q;d}.
\end{split}\end{equation*}
If $|\b|\!=\!N\!-\!3$, this gives \e_ref{coeff3Fano_e} and \e_ref{coeff3CY_e}
with $t,t_{\p}\!=\!0$.\\

\noindent
For a nonzero summand in \e_ref{ncCdfn2_e}, 
$$b_s\le b_s'' ~~\forall~s\!\in\![N] \qquad\hbox{and}\qquad
|\b''|\le N\!-\!3-|\Edg|;$$
the latter inequality follows from \e_ref{mvsum_e}.
This implies the bound on~$\b$ in~\e_ref{coeffnon0_e}.
If $d\!\in\!\bar\Z^+$  and $(\p,\b)\!\in\nset_l^N\!\times\!(\bar\Z^+)^N$ do not satisfy
the last condition in~\e_ref{coeffnon0_e} with $t\!=\!0$, set $\nc_{\p,\b}^{(d,0)}\!=\!0$.\\

\noindent
In the Calabi-Yau case, $\nua\!=\!0$, the collection 
$\cS_{\Ga}(\p,\b)\!\equiv\!\cS_{\Ga}(\p,\b,\bfd)$ does not depend on~$\bfd$.
In the projective case, $\nua\!=\!n$, the collection of pairs $(\p',\b')$ 
does not depend on~$\bfd$.
As~$t_v$ in~\e_ref{sumcond_e} is determined by~$\b'$, 
we abbreviate the elements of $\cS_{\Ga}(\p,\b)$ as~$(\p',\b')$ in either case.
In these extremal cases,  \e_ref{Phimcdfn_e} and~\e_ref{F0exp_e2b} 
reduce~\e_ref{ncCdfn2_e}~to
\begin{equation*}\begin{split}
\nc_{\p,\b}^{(d,0)}&=\sum_{\Ga}
\sum_{(\p',\b')\in\cS_{\Ga}(\p,\b)}
\!\!\!\!\!\!\!\!\!\!\!(-1)^{|\b|+|\b'|}\!\!\!\!\!\!\!\!\!\!\!\!\!\!\!\!
\sum_{\begin{subarray}{c} \b''\in(\bar{\Z}^+)^N;\b^-,\b^+\in(\bar{\Z}^+)^{\Edg}\\ 
(\c_v)_{v\in\Ver}\in((\bar\Z^+)^{\i})^{\Ver}\\ 
|\b''|_{\eta^{-1}(v)}+|\b^-|_{\E_{\Ga}^-(v)}+b_{e_v}^++\|\c_v\|=m_v\end{subarray}}
\!\!\!\!\!\!\!\Bigg\llbracket L(q)^{|\a|t_{\p'}}
\Phi_{\Ga,(\c_v)_{v\in\Ver}}(q)\\
&\hspace{2.2in}\times
\prod_{s=1}^{s=N}\frac{\Phi_{\hat{p}_s;b_s''-b_s}(q)}{b_s''!\,\Phi_0(q)}
\times\prod_{e\in\Edg}\!\!\!
\frac{\Phi_{\hat{p}_e';b_e^-+1+b_e'}(q)\Phi_{p_e';b_e^+-b_e'}(q)}{b_e^-!b_e^+!\,\Phi_0(q)^2}
\Bigg\rrbracket_{q;d},
\end{split}\end{equation*}
where 
\begin{equation*}\begin{split}
\Phi_{\Ga,(\c_v)_{v\in\Ver}}
&=\frac{L^{|\a|-(n-1-l)|\Ver|}}{I_0^2}
\!\prod_{v\in\Ver}\!\!\!\Bigg(\!\!  
(-1)^{m_v+|\c_v|}(m_v\!+\!|\c_v|)!
\prod_{r=1}^{\i}
\frac{1}{c_{v;r}!}\bigg(\!\!\frac{\Phi_r}{(r\!+\!1)!\,\Phi_0}\!\!\bigg)^{\!c_{v;r}}
\Bigg).
\end{split}\end{equation*}\\

\noindent
The coefficients $\nc_{\p,\b}^{(d,0)}$ must be invariant under the permutations of
$[N]$ (same permutations in the components of~$\p$ and~$\b$).
For $N\!\ge\!4$, this is not apparent from either of the above two descriptions of these coefficients,
even in the extremal cases;
thus, this is a consequence of the proof of Theorem~\ref{main_thm} below.
In the case of~\e_ref{pt4_e}, this invariance can be seen directly using Lemma~\ref{Iprp_lmm},
as indicated in Section~\ref{CY_subs}.

\section{Equivariant GW-invariants}
\label{equiv_sec}

\noindent
In this section we first review the relevant aspects of equivariant cohomology;
a more detailed discussion can be found in \cite[Section~1.1]{bcov1}.
We then state an equivariant version of Theorem~\ref{main_thm} and use it
to obtain Theorem~\ref{main_thm}.\\

\noindent
We denote by $\T$ the $n$-torus $(\C^*)^n$.
Its group cohomology is the polynomial algebra on $n$~generators:
$$H_{\T}^*\equiv H^*(B\T;\Q)=\Q[\al]\equiv\Q[\al_1,\ldots,\al_n],$$
where $\al\!=\!(\al_1,\ldots,\al_n)$ and $\al_i\!=\!\pi_i^*c_1(\ga^*)$ if
$$\pi_i\!: B\T\lra B\C^*\!=\!\P^{\i} \qquad\hbox{and}\qquad
\ga\lra\P^{\i}$$
are the projection onto the $i$-th component 
and the tautological line bundle, respectively.
Let
$$\H_{\T}^*=\Q_{\al}\equiv \Q(\al_1,\ldots,\al_n)
\qquad\hbox{and}\qquad
\cI\subset\Q[\al_1,\ldots,\al_n]\subset\H_{\T}^*$$
be the field of fractions of $H_{\T}^*$ and
the ideal in $\Q[\un\al]$ generated by the elementary symmetric polynomials
$\si_1,\si_2,\ldots,\si_{n-1}$ in $\al_1,\al_2,\ldots,\al_n$, respectively.
Let
$$\hat\si_r=(-1)^{r-1}\si_r\in\Q_{\al}\quad r=0,1,2,\ldots,\qquad
D_{\al}=\prod_{j\neq k}(\al_j-\al_k).$$ \\

\noindent
If $\T$ is acting on a topological space $M$, let
$$H_{\T}^*(M)\equiv H^*(BM;\Q), \qquad\hbox{where}\qquad BM=E\T\!\times_{\T}\!M,$$
be the \sf{equivariant cohomology} of $M$.
The projection map $BM\!\lra\!B\T$ induces an action of $H_{\T}^*$ on $H_{\T}^*(M)$.
We define
$$\H_{\T}^*(M)=H_{\T}^*(M)\otimes_{H_{\T}^*}\H_{\T}^*.$$
If the $\T$-action on $M$ lifts to an action on a (complex) 
vector bundle $V\!\lra\!M$, let
$$\E(V)\equiv e(BV)\in H_{\T}^*(M)\subset \H_{\T}^*(M)$$
denote the \sf{equivariant euler class of} $V$.\\

\noindent
Throughout the paper we work with the standard action of $\T$ on $\P^{n-1}$:
$$\big(e^{\I\th_1},\ldots,e^{\I\th_n}\big)\cdot [z_1,\ldots,z_n] 
=\big[e^{\I\th_1}z_1,\ldots,e^{\I\th_n}z_n\big];$$
it has $n$~fixed points:
$$P_1=[1,0,\ldots,0], \qquad P_2=[0,1,0,\ldots,0], 
\quad\ldots\quad P_n=[0,\ldots,0,1].$$
The $\T$-equivariant cohomology of $\P^{n-1}_N$ with respect 
to the induced diagonal $\T$-action on~$\P^{n-1}_N$ is given~by 
\begin{equation}\label{pncoh_e}
H_{\T}^*(\P^{n-1}_N) = \Q\big[\al,\un\x\big]\Big/
\big\{(\x_s\!-\!\al_1)\ldots(\x_s\!-\!\al_n)\!:s\!=\!1,\ldots,N\big\},
\end{equation}
where $\un\x\!=\!(\x_1,\ldots,\x_n)$
and $\x_s\!=\!\pi_s^*\x$ if $\pi_s\!:\P^{n-1}_N\!\lra\!\P^{n-1}$ is
the projection onto the $s$-th component and 
$\x\!\in\!H_{\T}^*(\P^{n-1})$ is the equivariant hyperplane class.
For each $\p\!\in\!\nset^N$, let
$$\un\x^{\p}=\prod_{i=1}^{i=N}\!\x_s^{p_s}\in H_{\T}^*(\P^{n-1}_N)\,;$$
these elements form a basis for $H_{\T}^*(\P^{n-1}_N)$ as a module over $H_{\T}^*\!=\!\Q[\al]$.\\

\noindent
The action of $\T$ on $\P^{n-1}$ naturally lifts to
the tautological line bundle~$\ga$, the vector bundle
$$\L\equiv \bigoplus_{k=1}^{k=l}\ga^{*\otimes a_k}
= \bigoplus_{k=1}^{k=l}\O_{\P^{n-1}}(a_k)\lra\P^{n-1},$$
and the tangent bundle~$T\P^{n-1}$ so~that
\begin{equation}\label{tangrestr_e}
\E(\L)\big|_{P_i}=\lr{\a}\al_i^l, \quad
\E\big(T\P^{n-1}\big)\big|_{P_i}=
\prod_{\begin{subarray}{c}1\le k\le n\\k\neq i\end{subarray}}\!\!(\al_i\!-\!\al_k)
\qquad\forall~i=1,2,\ldots,n.
\end{equation}
Via composition of maps, the action of $\T$ on $\P^{n-1}$ and $\L$
induces actions on $\ov\M_{0,N}(\P^{n-1},d)$ and 
$$\V_d=\ov\M_{0,N}(\L,d)\lra \ov\M_{0,N}(\P^{n-1},d)$$
so that the evaluation maps
$$\ev\!\equiv\!\ev_1\!\times\!\ldots\!\times\!\ev_N\!: 
\ov\M_{0,N}(\P^{n-1},d)\lra \P^{n-1}_N\,, \quad
\wt\ev_s\!:\V_d\lra\ev_s^*\L,~~\wt\ev_s\big([\cC,f;\ti{f}]\big)=\big[\ti{f}(x_s(\cC))],$$
where $x_s(\cC)$ is the $s$-th marked point of the curve~$\cC$, are $\T$-equivariant. 
In particular, $\V_d$ has a well-defined equivariant euler class
$$\E(\V_d)\in H_{\T}^*\big(\ov\M_{0,N}(\P^{n-1},d)\big).$$
Since the bundle homomorphisms $\wt\ev_s$ are surjective, their kernels 
are again equivariant vector bundles.
Let 
$$\V_d''=\ker\wt\ev_2\lra\ov\M_{0,2}(\P^{n-1},d).$$\\

\noindent
With $\un\hb$ and $\un\hb^{-1}$ as in~\e_ref{Zdfn_e} and $\un\x$ 
as in~\e_ref{pncoh_e}, let
\BE{cZdfn_e} \cZ\big(\un\hb,\un\x,Q\big)= \sum_{d=0}^{\i}Q^d
\ev_*\bigg\{\frac{\E(\V_d)}{\prod_{s=1}^{s=N}(\hb_s\!-\!\psi_s)}\bigg\}
\in H_{\T}^*(\P^{n-1}_N)\big[\big[\un\hb^{-1},Q\big]\big],\EE
where $\ev\!:\ov\M_{0,N}(\P^{n-1},d)\!\lra\!\P^{n-1}_N$;
for $N\!=\!1,2$, we define the coefficient of $Q^0$ to~be
$$ \lr\a\x_1^l \qquad\hbox{and}\qquad
-\frac{\lr\a\x_1^l}{\hb_1\!+\!\hb_2}
\sum_{\begin{subarray}{c}p_1+p_2+r=n-1\\ p_1,p_2,r\ge0\end{subarray}}
\!\!\!\!\!\!\!\!\!\!\!\hat\si_r\x_1^{p_1}\x_2^{p_2}\,,$$
respectively.
For each $p\!\in\!\nset$, let
\BE{cZpdfn_e}\cZ_p(\hb,\x,Q)=\x^p
+\sum_{d=1}^{\i}Q^d\ev_{1*}\bigg\{\frac{\E(\V_d'')\ev_2^*\x^p}{\hb-\psi}\bigg\}
\in H_{\T}^*(\P^{n-1})\big[\big[\hb^{-1},Q\big]\big],\EE
where $\ev_1,\ev_2\!:\ov\M_{0,2}(\P^{n-1},d)\!\lra\!\P^{n-1}$.
Similarly to \e_ref{Depdfn_e}, let
$$\cZ_{\p}(\un\hb,\un\x,Q)
\equiv\prod_{s=1}^{s=N}
\frac{1}{\hb_s}\frac{\cZ_{p_s}(\hb,\x_s,Q)}{\prod\limits_{r=p_s-l+1}^{n-l-1}\!\!\!\!I_r(q_s)}\,.$$

\begin{thmlet}\label{equiv_thm}
Suppose $n,N\!\in\!\Z^+$, with $N\!\ge\!3$, and $\a\!\in\!(\Z^+)^l$ is such that $\|\a\|\!\le\!n$.
The generating function~\e_ref{cZdfn_e} for equivariant $N$-pointed genus~0 GW-invariants of
a complete intersection $X_{\a}\!\subset\!\P^{n-1}$ 
is given~by 
\BE{equivthm_e}\cZ\big(\un\hb,\un\x,Q\big)= \lr{\a}
\sum_{\p\in\nset^N}\sum_{\b\in(\bar\Z^+)^N}
\sum_{d=0}^{\i}\cC_{\p,\b}^{(d)}q^d\un\hb^{-\b}\cZ_{\p}(\un\hb,\un\x,Q)\EE
for some $\cC_{\p,\b}^{(d)}\!\in\!\Q[\al]$ such that 
\BE{equivthm_e2}\cC_{\p,\b}^{(d)}-\sum_{t=0}^{\i} \nc_{\p,\b}^{(d,t)}\hat\si_n^t\in\cI,\EE
where $\nc_{\p,\b}^{(d,t)}\!\in\!\Q$ are the numbers defined above Theorem~\ref{main_thm}.
\end{thmlet}

\noindent
Setting $\al\!=\!0$ in Theorem~\ref{equiv_thm} and using \cite[Theorem~3]{PoZ}, we obtain 
\BE{Zform_e} Z\big(\un\hb,\un{H},Q\big)= \lr{\a}
e^{-\sum\limits_{s=1}^{s=N}\!\!J(q_s)w_s}\!\!\!
\sum_{\p\in\nset^N}\sum_{\b\in(\bar\Z^+)^N}
\sum_{d=0}^{\i}\nc_{\p,\b}^{(d,0)}q^d\un\hb^{-\b}\De_{\p}(\un\hb,\un{H},Q).\EE
This implies Theorem~\ref{main_thm} provided $\nc_{\p,\b}^{(d,0)}\!=\!0$ 
if $p_s\!<\!l$ for some~$s\!\in\![N]$; this is shown in the next paragraph.\\

\noindent
Suppose instead $\nc_{\p,\b}^{(d,0)}\!=\!0$ for some triple $(\p,\b,d)$ with $p_1\!<\!l$.
Choose $(\p,\b,d)$ minimizing~$p_1$, as well as minimizing~$d$ for the smallest possible~$p_1$.
We show that 
\BE{GWc0_e}
\blr{\tau_{b_1}H^{n-1-p_1},\ldots,\tau_{b_N}H^{n-1-p_N}}_{0,d}^{X_{\a}}
=\lr\a\nc_{\p,\b}^{(d,0)}\,.\EE
By \e_ref{genus0_e} and \e_ref{Zdfn_e}, this GW-invariant is the coefficient of 
$Q^d\prod\limits_{s=1}^{s=N}\hb_s^{-(b_s+1)}H_s^{p_s}$
of the right-hand side of~\e_ref{Zform_e}.
Suppose a triple $(\p',\b',d')$, with $\nc_{\p',\b'}^{(d',0)}\!\neq\!0$,
contributes to this coefficient.
Since the lowest power of~$H$ in the coefficient of a product of powers of $q$ and~$\hb^{-1}$
in $H^pF_p(w,q)$ is $\min(p,l)$, $p_1'\!=\!p_1$ by the minimality of~$p_1$
and thus $d'\!=\!d$ by the minimality of~$d$. 
Since the coefficient of~$q^0$ in $H^pF_p(w,q)$ is~$H^p$, $p_s'\!=\!p_s$ for all $s\!\in\![N]$
and thus $b_s'\!=\!b_s$ for all $s\!\in\![N]$; this gives~\e_ref{GWc0_e}.
Since $H^{n-1-p_1}|_{X_{\a}}\!=\!0$ for $p_1\!<\!l$, 
we conclude that $\nc_{\p,\b}^{(d,0)}\!=\!0$.\\

\noindent
The proof of Theorem~\ref{equiv_thm} below provides an algorithm for computing the
structure coefficients~$\cC_{\p,\b}^{(d)}$ completely.
On the other hand, they may be irrelevant in many applications.
For example, the one- and two-point equivariant generating functions~\e_ref{cZdfn_e}
play a key in the localization computation of the genus~1 GW-invariants 
of Calabi-Yau complete intersections in~\cite{bcov1} and in~\cite{Po2},
but the structure coefficients lying in~$\cI$ are ignored.
Similarly, the equivariant generating functions with $N\!\le\!g$ and
the structure coefficients lying in~$\cI$ dropped should play a key role
in computing genus $g\!\ge\!2$ GW-invariants of complete intersections.

\section{Proof of Theorem~\ref{main_thm}}
\label{equivpf_sec}

\subsection{Localization Setup}
\label{outline_subs}

\noindent
If $\T$ acts smoothly on a smooth compact oriented manifold $M$, 
there is a well-defined integration-along-the-fiber homomorphism
$$\int_M\!: H_{\T}^*(M)\lra H_{\T}^*$$
for the fiber bundle $BM\!\lra\!B\T$.
The classical localization theorem of~\cite{ABo} relates it to 
integration along the fixed locus of the $\T$-action.
The latter is a union of smooth compact orientable manifolds~$F$
and $\T$ acts on the normal bundle $\N F$ of each~$F$.
Once an orientation of $F$ is chosen, there is a well-defined 
integration-along-the-fiber homomorphism
$$\int_F\!: H_{\T}^*(F)\lra H_{\T}^*.$$
The localization theorem states that 
\begin{equation}\label{ABothm_e}
\int_M\psi = \sum_F\int_F\frac{\psi|_F}{\E(\N F)} \in \H^*_{\T}
\qquad\forall~\psi\in H_{\T}^*(M),
\end{equation}
where the sum is taken over all components $F$ of the fixed locus of $\T$. 
Part of the statement of~\e_ref{ABothm_e} is that $\E(\N F)$
is invertible in~$\H_{\T}^*(F)$.\\

\noindent
The standard $\T$-action on $\P^{n-1}_N$ has $nN$~fixed points:
$$P_{i_1\ldots i_N}\equiv 
P_{i_1}\!\times\!\ldots\!\times\!P_{i_N}.$$
The restriction maps on the equivariant cohomology induced by the inclusions
$P_{i_1\ldots i_N}\lra \P^{n-1}_N$ are the homomorphisms
\begin{equation}\label{restrmap_e}
H_{\T}^*(\P_N^{n-1})\lra \Q[\al_1,\ldots,\al_n], \qquad 
\x_s\lra\al_{i_s},~s\!=\!1,\ldots,N.
\end{equation}
By~\e_ref{pncoh_e} and \e_ref{restrmap_e},
$$\eta=0 \in H_{\T}^*(\P_N^{n-1}) \quad\Llra\quad
\eta|_{P_{i_1\ldots i_N}}=0\in H_{\T}^*
~~~\forall\, i_s=1,2,\ldots,n,~s\!=\!1,\ldots,N,$$
i.e.~an element of $H_{\T}^*(\P_N^{n-1})$ is determined by its restrictions
to the $nN$ $\T$-fixed points.
For each $i\!=\!1,2,\ldots,n$, the equivariant Poincare dual of $P_i$ 
in $\P^{n-1}$ is given~by
\BE{phidfn_e}
\phi_i= \prod_{k\neq i}(\x\!-\!\al_k) \in H_{\T}^*(\P^{n-1}).\,\footnotemark\EE
Thus,\footnotetext{In other words, if $\eta\!\in\!H_{\T}^*(\P^{n-1})$, then
$$\eta|_{P_i}\equiv\int_{P_i}\eta|_{P_i}=
\int_{\P^{n-1}}\eta\phi_i.$$}
by the defining property of the cohomology pushforward \cite[(1.11)]{bcov1},
the power series $\cZ(\un\hb,\un\x,Q)$ in~\e_ref{cZdfn_e} is  
completely determined by the $nN$ power series
\BE{cZval_e}\cZ\big(\un\hb,\al_{i_1,\ldots,i_N},Q\big)
=\sum_{d=0}^{\i}Q^d\!\!
\int_{\ov\M_{0,N}(\P^{n-1},d)}\E(\V_d)\prod_{s=1}^{s=N}\!\!
\bigg(\frac{\ev_s^*\phi_{i_s}}{\hb_s\!-\!\psi_s}\bigg),\EE
where $\al_{i_1\ldots i_N}\!\equiv\!(\al_{i_1},\ldots,\al_{i_N})$.\\

\noindent
As described in detail in \cite[Section~27.3]{MirSym},
the fixed loci~$\cZ_{\Ga}$ of the $\T$-action on $\ov\M_{0,N}(\P^{n-1},d)$ 
are indexed by \sf{$N$-marked decorated trees}~$\Ga$.
An \sf{$N$-marked decorated tree} is a tuple
\begin{equation}\label{decortgraphdfn_e}
\Ga = \big(\Ver,\Edg;\mu,\d,\eta\big),
\end{equation}
where $(\Ver,\Edg)$ is a tree and
$$\mu\!:\Ver\lra[n]\equiv\!\!\{1,\ldots,n\}, \qquad
\d\!:\Edg\lra\Z^+, \quad\hbox{and}\quad \eta\!:[N]\lra\Ver$$
are maps such that
\begin{equation}\label{decorgraphcond_e}
\mu(v_1)\neq\mu(v_2)  \qquad\hbox{if}\quad \{v_1,v_2\}\in\Edg.
\end{equation}
In the first diagram of Figure~\ref{graphcore_fig}, 
the value of the map $\mu$ on each vertex
is indicated by the number next to the vertex.
Similarly, the value of the map $\d$ on each edge is indicated by 
the number next to the edge.
By~\e_ref{decorgraphcond_e}, no two consecutive vertex labels are the same.
Let
$$|\Ga|=\sum_{e\in\Edg}\d(e).$$
For each $e\!=\!\{v,v'\}\!\in\!\tnE_v(\Ga)$, 
let $\mu_v(e)\!=\!\mu(v')\!\in\![n]$.\\

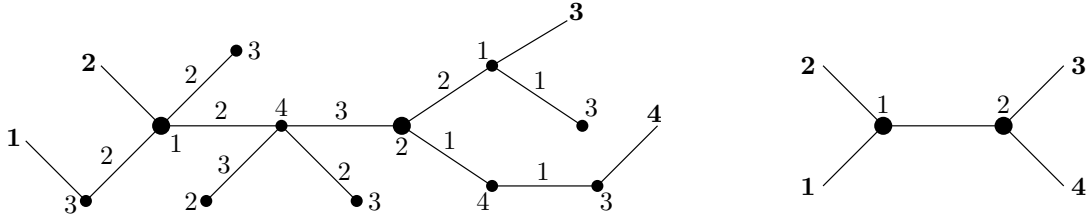
\begin{figure}
\begin{pspicture}(6,-1.1)(10,2)
\psset{unit=.4cm}
\pscircle*(30,0){.3}\rput(30,-.7){\smsize{$2$}}
\psline[linewidth=.04](30,0)(26,0)\pscircle*(26,0){.2}
\rput(28,.5){\smsize{$3$}}\rput(26,.6){\smsize{$4$}}
\psline[linewidth=.04](26,0)(23.5,-2.5)\pscircle*(23.5,-2.5){.2}
\rput(24.1,-1.3){\smsize{$3$}}\rput(23,-2.5){\smsize{$2$}}
\psline[linewidth=.04](26,0)(28.5,-2.5)\pscircle*(28.5,-2.5){.2}
\rput(28.1,-1.5){\smsize{$2$}}\rput(29.1,-2.5){\smsize{$3$}}
\psline[linewidth=.04](26,0)(22,0)\pscircle*(22,0){.3}
\rput(24,.5){\smsize{$2$}}\rput(22.5,-.6){\smsize{$1$}}
\psline[linewidth=.04](22,0)(24.5,2.5)\pscircle*(24.5,2.5){.2}
\rput(23,1.7){\smsize{$2$}}\rput(25.1,2.5){\smsize{$3$}}
\psline[linewidth=.04](22,0)(20,2)\rput(19.6,2.1){\smsize{$\bf 2$}}
\psline[linewidth=.04](22,0)(19.5,-2.5)\pscircle*(19.5,-2.5){.2}
\rput(20.2,-1){\smsize{$2$}}\rput(19,-2.6){\smsize{$3$}}
\psline[linewidth=.04](19.5,-2.5)(17.5,-.5)\rput(17.1,-.4){\smsize{$\bf 1$}}
\psline[linewidth=.04](30,0)(33,2)\pscircle*(33,2){.2}
\rput(31.4,1.5){\smsize{$2$}}\rput(32.7,2.5){\smsize{$1$}}
\psline[linewidth=.04](30,0)(33,-2)\pscircle*(33,-2){.2}
\rput(31.6,-.5){\smsize{$1$}}\rput(32.7,-2.5){\smsize{$4$}}
\psline[linewidth=.04](35.5,3.5)(33,2)\rput(35.8,3.8){\smsize{$\bf 3$}}
\psline[linewidth=.04](36,0)(33,2)\pscircle*(36,0){.2}
\rput(34.6,1.5){\smsize{$1$}}\rput(36.3,.6){\smsize{$3$}}
\psline[linewidth=.04](36.5,-2)(33,-2)\pscircle*(36.5,-2){.2}
\rput(34.7,-1.5){\smsize{$1$}}\rput(36.8,-2.5){\smsize{$3$}}
\psline[linewidth=.04](38.5,0)(36.5,-2)\rput(38.4,.4){\smsize{$\bf 4$}}
\pscircle*(50,0){.3}\rput(50,.7){\smsize{$2$}}
\psline[linewidth=.04](50,0)(46,0)
\pscircle*(46,0){.3}\rput(46,.7){\smsize{$1$}}
\psline[linewidth=.04](46,0)(44,2)\rput(43.5,2){\smsize{$\bf 2$}}
\psline[linewidth=.04](46,0)(44,-2)\rput(43.5,-2){\smsize{$\bf 1$}}
\psline[linewidth=.04](50,0)(52,2)\rput(52.5,2){\smsize{$\bf 3$}}
\psline[linewidth=.04](50,0)(52,-2)\rput(52.5,-2){\smsize{$\bf 4$}}
\end{pspicture}
\caption{A decorated tree, with special vertices indicated by larger dots, 
and its decorated core}
\label{graphcore_fig}
\end{figure}

\noindent
If $\Ga$ is a decorated tree  as in~\e_ref{decortgraphdfn_e} and $v\!\in\!\Ver$, let
$$\val_{\Ga}(v)= \big|\tnE_{\Ga}(v)\big|+\big|\eta^{-1}(v)\big|$$
be the \sf{valence} of $v$ in $\Ga$.
If in addition $N\!\ge\!3$, \sf{the core of~$\Ga$} is the tuple 
$\bar\Ga\!\equiv\!(\ov\Ver,\ov\Edg;\bar\mu,\bar\eta)$ such~that
\begin{enumerate}[label=(R\arabic*),leftmargin=*]
\item $(\ov\Ver,\ov\Edg)$ is a tree,
$\ov\Ver\!=\!\{v\!\in\!\Ver\!:\,\val_{\Ga}(v)\!\ge\!3\big\}$ and
$\bar\mu\!=\!\mu|_{\ov\Ver}$;
\item $\{v,v'\}\!\in\!\ov\Edg$ if and only if
$v,v'\!\in\!\ov\Ver$, $v\!\neq\!v'$, and for some $m\!\ge\!0$ there exist distinct
$$v_1,\ldots,v_m\!\in\!\Ver\!-\!\ov\Ver  \quad\hbox{s.t.}\quad
\{v_{i-1},v_i\}\in\Edg~~\forall~i\!\in\![m\!+\!1],
~~\hbox{where}~~v_0\!\equiv\!v,~v_{m+1}\!\equiv\!v'\,;$$
\item if $s\!\in\!\eta^{-1}(\ov\Ver)\subset[N]$, $\bar\eta(s)\!=\!\eta(s)$;
if  $s\!\in\!\eta^{-1}(\Ver\!-\!\ov\Ver)$, there exist distinct elements
$$v_1,\ldots,v_m\!\in\!\Ver\!-\!\ov\Ver
\quad\hbox{s.t.}\quad \{v_{i-1},v_i\}\in\Edg~~\forall~i\!\in\![m\!+\!1],
~~\hbox{where}~~v_0\!\equiv\!\bar\eta(s),~v_{m+1}\!=\!\eta(s)\,.$$
\end{enumerate}
The core of a graph with $N\!\ge\!3$ is obtained by repeatedly collapsing all vertices 
with valence less than~3 onto their neighbors, until no such vertices are left;
see Figure~\ref{graphcore_fig}.
We will call 
the vertices $\ov\Ver$ of the core $\bar\Ga$ the \sf{special vertices of~$\Ga$}.\\

\noindent
The localization formula~\e_ref{ABothm_e} reduces the restriction of \e_ref{cZdfn_e}
to each fixed point $P_{i_1\ldots i_N}\!\in\!\P^{n-1}_N$ to a sum over decorated trees.
This sum can be computed by breaking each such tree~$\Ga$ at its special vertices 
into~\sf{strands}, with each of the strands keeping a copy of the special vertex, 
with its label, which will have a new marked point attached; see Figure~\ref{strands_fig}.
There are three types of strands:
\begin{enumerate}[label=(S\arabic*),leftmargin=*]
\item one-marked strands;
\item\label{Z2conn_item} strands with two new marked points;
\item\label{Z2main_item} 
strands with one new marked points and one of the original $N$ marked points.\\
\end{enumerate}

\begin{figure}
\begin{pspicture}(6,-1.1)(10,2)
\psset{unit=.4cm}
\pscircle*(23,2){.2}\rput(23.2,2.6){\smsize{$1$}}
\psline[linewidth=.04](23,2)(19,2)\pscircle*(19,2){.2}
\rput(21,2.5){\smsize{$2$}}\rput(18.8,2.6){\smsize{$3$}}
\psline[linewidth=.04](23,2)(25,.5)\rput(25.6,.6){\smsize{$\bf e_1$}}
\pscircle*(23,-2){.2}\rput(23.2,-2.6){\smsize{$1$}}
\psline[linewidth=.04](23,-2)(19,-2)\pscircle*(19,-2){.2}
\rput(21,-2.5){\smsize{$2$}}\rput(18.8,-2.6){\smsize{$3$}}
\psline[linewidth=.04](19,-2)(17,0)\rput(16.5,0){\smsize{$\bf 1$}}
\psline[linewidth=.04](23,-2)(25,-.5)\rput(25.5,-.6){\smsize{$\bf 1'$}}
\pscircle*(38,0){.2}\rput(38,-.7){\smsize{$2$}}
\psline[linewidth=.04](38,0)(34,0)\pscircle*(34,0){.2}
\rput(36,.5){\smsize{$3$}}\rput(34,.6){\smsize{$4$}}
\psline[linewidth=.04](34,0)(31.5,-2.5)\pscircle*(31.5,-2.5){.2}
\rput(32.1,-1.3){\smsize{$3$}}\rput(31,-2.5){\smsize{$2$}}
\psline[linewidth=.04](34,0)(36.5,-2.5)\pscircle*(36.5,-2.5){.2}
\rput(36.1,-1.5){\smsize{$2$}}\rput(37.1,-2.5){\smsize{$3$}}
\psline[linewidth=.04](34,0)(30,0)\pscircle*(30,0){.2}
\rput(32,.5){\smsize{$2$}}\rput(30.5,-.6){\smsize{$1$}}
\psline[linewidth=.04](38,0)(40,1.5)\rput(40.7,1.8){\smsize{$\bf e^-$}}
\psline[linewidth=.04](30,0)(28,1.5)\rput(27.5,1.6){\smsize{$\bf e^+$}}
\pscircle*(44,2){.2}\rput(44,2.7){\smsize{$2$}}
\psline[linewidth=.04](44,2)(48,2)\pscircle*(48,2){.2}
\rput(46,2.5){\smsize{$2$}}\rput(47.7,2.5){\smsize{$1$}}
\psline[linewidth=.04](50.5,3.5)(48,2)\rput(50.8,3.8){\smsize{$\bf 3$}}
\psline[linewidth=.04](51,0)(48,2)\pscircle*(51,0){.2}
\rput(49.6,1.5){\smsize{$1$}}\rput(51.3,.6){\smsize{$3$}}
\psline[linewidth=.04](44,2)(42,.5)\rput(41.6,.5){\smsize{$\bf 3'$}}
\pscircle*(44,-2){.2}\rput(44,-2.7){\smsize{$2$}}
\psline[linewidth=.04](44,-2)(48,-2)\pscircle*(48,-2){.2}
\rput(46,-1.5){\smsize{$1$}}\rput(47.7,-2.5){\smsize{$4$}}
\psline[linewidth=.04](51.5,-2)(48,-2)\pscircle*(51.5,-2){.2}
\rput(49.7,-1.5){\smsize{$1$}}\rput(51.8,-2.5){\smsize{$3$}}
\psline[linewidth=.04](53.5,0)(51.5,-2)\rput(53.4,.4){\smsize{$\bf 4$}}
\psline[linewidth=.04](44,-2)(42,-.5)\rput(41.6,-.5){\smsize{$\bf 4'$}}
\end{pspicture}
\caption{The strands of the graph in the first diagram in Figure~\ref{graphcore_fig}.}
\label{strands_fig}
\end{figure}
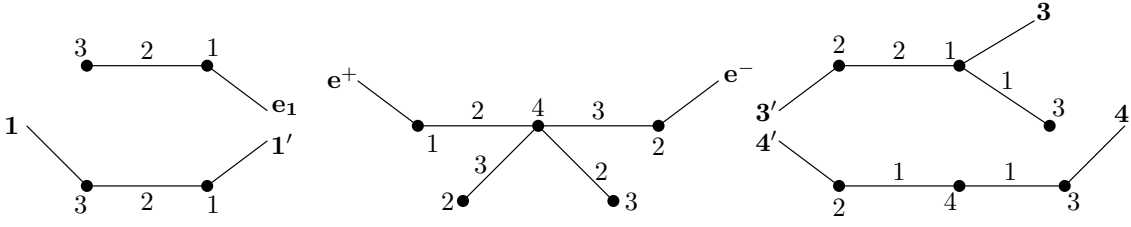

\noindent
By~\e_ref{ABothm_e}, each one-pointed strand at a special vertex 
$v\!\in\!\ov\Ver\!\subset\!\Ver$ contributes to
\BE{cZprdfn_e}\cZ'^*\big(\hb,\al_j,Q\big)
\equiv\sum_{d=1}^{\i}Q^d\!\!
\int_{\ov\M_{0,1}(\P^{n-1},d)}\E(\V_d')
\frac{\ev_1^*\phi_j}{\hb\!-\!\psi_1},\EE
where $j\!=\!\mu(v)\!\in\![n]$ is the label of the vertex $v$ of $\Ga$ and 
$$\V_d'\lra \ov\M_{0,1}(\P^{n-1},d)$$
is the kernel of the surjective vector bundle homomorphism
$\wt\ev_1\!:\V_d\!\lra\!\ev_1^*\L$.
By the dilaton relation \cite[p527]{MirSym},
$$\wt\cZ^*\big(\hb,\al_j,Q\big)\equiv
\sum_{d=1}^{\i}Q^d\!\!
\int_{\ov\M_{0,2}(\P^{n-1},d)}\E(\V_d')
\bigg(\frac{\ev_1^*\phi_j}{\hb\!-\!\psi_1}\bigg)
=\hb^{-1}\cZ'^*\big(\hb,\al_j,Q\big).$$
Each of the two-pointed strands contributes to 
$$\cZ^*\big(\hb_1,\hb_2,\al_{j_1},\al_{j_2},Q\big)
\equiv\sum_{d=1}^{\i}Q^d\!\!
\int_{\ov\M_{0,2}(\P^{n-1},d)}\E(\V_d)
\frac{\ev_1^*\phi_{j_1}}{\hb_1\!-\!\psi_1}\frac{\ev_2^*\phi_{j_2}}{\hb_2\!-\!\psi_2},$$
where $j_1,j_2\!\in\![n]$ are the labels of the vertices to which the marked points
are attached.
Thus, the power series $\cZ(\un\hb,\un\x,Q)$ in~\e_ref{cZdfn_e} is determined
by the previously computed power series for
one- and two-pointed GW-invariants.\\

\noindent
While the number of one-marked strands at each node can be arbitrary large,
as indicated in~\cite[Sections~2.1,2.2]{bcov1} 
it is possible to sum over all possibilities 
for these strands at each special vertex; see Corollary~\ref{pt1_crl} below.
On the other hand, the number of special vertices, 
the number of  two-pointed strands of type~\ref{Z2conn_item},
and the number of two-pointed strands of type~\ref{Z2main_item},
are bounded (by $N\!-\!2$, $N\!-\!3$, and $N$, respectively).
Using the Residue Theorem for~$S^2$, one can then sum up over all 
possibilities of the markings for each of the distinguished nodes.
Thus, the approach of breaking trees at special nodes reduces~\e_ref{cZdfn_e} 
to a finite sum, with one summand for each trivalent $N$-marked tree.\\

\noindent
The description of the structure constants $\nc_{\p,\b}^{(d,t)}$ in Section~\ref{graph_subs}
is obtained by breaking the trees at all special vertices.
On the other hand, the description in Section~\ref{multipt_subs} is obtained by breaking
at the special vertex~$\bar\eta(N)$ only.
In addition to the strands~(S1), we would then obtain strands with marked points indexed
by the sets $S_i\!\sqcup\!\{0\}$, for a partition $\{S_i\}_{i\in[m]}$ of $[N]$ so that 
one of the sets~$S_i$ is~$\{N\}$.
With either approach, the main step is summing over all possibilities for 
the strands~(S1), as done in Corollary~\ref{pt1_crl}.

\subsection{Notation and preliminaries}
\label{prelim_subs}

\noindent
If $f\!=\!f(\hb)$ is a rational function in $\hb$ and $\hb_0\!\in\!S^2$,
let
$$\Res{\hb=\hb_0}\big\{f(\hb)\big\} = \frac{1}{2\pi\I}\oint f(\hb)\tnd\hb\,,$$
where the integral is taken over a positively oriented loop around $\hb\!=\!\hb_0$
containing no other singular points of~$f$, 
denote the residue of $f(\hb)\tnd\hb$ at $\hb\!=\!\hb_0$.
With this definition,
$$\Res{\hb=\i}\big\{f(\hb)\big\}=-\Res{w=0}\big\{w^{-2}f(w^{-1})\big\}.$$
If $f$ involves variables other than $\hb$, 
$\Res{\hb=\hb_0}\big\{f(\hb)\big\}$ will be a function  of such variables.
If $f$ is a power series in $q$ with coefficients that are rational
functions in~$\hb$ and possibly other variables, denote by 
$\Res{\hb=\hb_0}\big\{f(\hb)\big\}$ the  power series in~$q$ obtained by 
replacing each of the coefficients by its residue at $\hb\!=\!\hb_0$.
If $\hb_1,\ldots,\hb_k$ is a collection of distinct points in~$S^2$, let
$$\Res{\hb=\hb_1,\ldots,\hb_k}\big\{f(\hb)\big\} =\sum_{i=1}^{i=k} \Res{\hb=\hb_i}\big\{f(\hb)\big\}$$
be the sum of the residues at the specified values of $\hb$.\\

\noindent
We denote by
$$\Q_{\al}'\equiv \Q\big[\al,\si_n^{-1},D_{\al}^{-1}\big]
\subset\Q_{\al}$$
the subring of rational functions in $\al_1,\ldots,\al_n$
with denominators that are products of~$\si_n$ and~$D_{\al}$.
Let 
$$\Q_{\al;\hb,\x}'\equiv 
\Q_{\al'}[\hb,\x^{\pm1}]_{\left\langle(\x+r\hb)^n-\x^n,
\prod\limits_{k=1}^{k=n}(\x-\al_k+r\hb)-\prod\limits_{k=1}^{k=n}(\x-\al_k)
\big|r\in\Z^+\right\rangle}\subset\Q_{\al}(\hb,\x)$$
be the subring of rational functions in $\al_1,\ldots,\al_n$, $\hb$, and $\x$ 
with numerators that are polynomials in $\al_1,\ldots,\al_n$, $\hb$, and $\x$
and with denominators that are products~of 
$$\si_n\,,\quad D_{\al}\,,\quad \x\,, \quad (\x\!+\!r\hb)^n\!-\!\x^n\,, \quad 
\prod\limits_{k=1}^{k=n}(\x\!-\!\al_k\!+\!r\hb)-\prod\limits_{k=1}^{k=n}(\x\!-\!\al_k)\,,
\qquad\hbox{with}~~r\in\Z^+.$$
If $R$ is one of the rings $\Q_{\al}'$, $\Q_{\al}'[\x^{\pm1}]$, or $\Q_{\al;\hb,\x}'$
and $f_1$ and $f_2$ are elements of~$R$ or~$R[[Q]]$, we will write   $f_1\!\sim\!f_2$
if $f_1\!\sim\!f_2$ lies in $\cI\cdot R$ or  $\cI\cdot R[[Q]]$, respectively.
By the next lemma, certain operations on these rings respect these equivalence relations.

\begin{lmm}\label{ressum_lmm}
\begin{enumerate}[label=(\arabic*),leftmargin=*]
\item If $f\!\in\!\Q_{\al;\hb,\x}'$, there exists 
$g\!\in\!\Q_{\al}'[\x^{\pm1}]$ such that
$$\Res{\hb=0}\big\{f(\hb,\x\!=\!\al_j)\big\}=g(\x\!=\!\al_j)
\qquad\forall\,j\!\in\![n].$$
\item If $g\!\in\!\Q_{\al}'[\x^{\pm1}]$,
$$\Res{\x=0,\i}\left\{\frac{g(\x)}{\prod\limits_{k=1}^{k=n}(\x-\al_k)}\right\}
\in\Q_{\al}'\,.$$
\item For every $p\!\in\!\Z$, 
$$-\Res{\x=0,\i}\left\{\frac{\x^p}{\prod\limits_{k=1}^{k=n}(\x-\al_k)}\right\}
\sim\begin{cases} \hat\si_n^t,&\hbox{if}~p\!=\!n\!-\!1+nt~\hbox{with}~t\in\Z;\\
0,&\hbox{if}~p\!+\!1\not\in n\Z.
\end{cases}$$
\end{enumerate}
\end{lmm}

\begin{proof}
If $f\!\in\!\Q_{\al;\hb,\x}'$, then
\begin{gather}
\Res{\hb=0}\big\{f(\hb,\x\!=\!\al_j)\big\}=
\bigg(\Res{\hb=0}\big\{f(\hb,\x)\big\}\bigg)\bigg|_{\x=\al_j}\,, \qquad
\Res{\hb=0}\big\{f(\hb,\x)\big\}\in \Q_{\al}'\big[\x^{\pm1},\si_{n-1}(\x)^{-1}\big],\notag\\
\hbox{where}\qquad \si_{n-1}(\x)=\sum_{i=1}^{i=n}\prod_{k\neq i}(\x-\al_k)\,.
\label{sinmn1dfn_e}
\end{gather}
The first claim of this lemma thus follows from the observation that 
$$\frac{1}{\si_{n-1}(\x)}\bigg|_{\x=\al_j}=\frac{1}{D_{\al}^2}
\Bigg(\sum_{i=1}^{i=n}\bigg(\prod_{\begin{subarray}{c}i'\neq i\\ k\neq i'\end{subarray}}
(\al_{i'}\!-\!\al_k)^2\bigg)\bigg(\prod_{k\neq i}(\x\!-\!\al_k)\bigg)\Bigg)
\Bigg|_{\x=\al_j}
\qquad\forall\,j\in[n].$$
The second claim is immediate from the third.
The third claim of this lemma follows from the power series expansions
$$-\frac{1}{\x^n-\hat\si_n}=\sum_{r=0}^{\i}\hat\si_n^{-r-1}\x^{nr}\,,
\qquad \frac{1}{1-\hat\si_n w^n} =\sum_{r=0}^{\i}\hat\si_n^rw^{nr}\,.$$
around $\x\!=\!0$ and $w\!=\!0$, respectively.
\end{proof}

\noindent
We will also use  the Residue Theorem on $S^2$: 
$$\sum_{\x_0\in S^2}\Res{\x=\x_0}\big\{f(\x)\big\}=0$$
for every  rational function $f\!=\!f(\x)$ on $S^2\!\supset\!\C$.

\subsection{Equivariant one- and two-pointed formulas}
\label{equiv1and2pt_subs}

\noindent
The most fundamental generating function for GW-invariants in the mirror symmetry computations
following~\cite{Gi2} is
\BE{wtcZdfn_e}\begin{split}\wt\cZ(\hb,\x,Q) &\equiv1+\wt\cZ^*(\hb,\x,Q)\\
&\equiv 1+\sum_{d=1}^{\i}Q^d\ev_{1*}\bigg\{\frac{\E(\V_d')}{\hb-\psi_1}\bigg\}
\in H_{\T}^*(\P^{n-1})\big[\big[\hb^{-1},Q\big]\big],
\end{split}\EE
where $\ev_1\!:\ov\M_{0,2}(\P^{n-1},d)\!\lra\!\P^{n-1}$ and 
$\V_d'\!\lra\!\ov\M_{0,2}(\P^{n-1},d)$ is the kernel of the surjective 
vector bundle homomorphism $\wt\ev_1\!:\V_d\!\lra\!\ev_1^*\L$.
By~\cite{Gi2}, $\wt\cZ(\hb,\al_j,Q)\!\in\!\Q_{\al}(\hb)$ for $j\!\in\![n]$.
Thus, we can define
\begin{alignat*}{1}
\ze(\al_j,Q)&=\Res{\hb=0}\big\{\ln\big(1+\wt\cZ^*(\hb,\al_j,Q)\big)\big\}
\in Q\cdot\Q_{\al}\big[\big[Q\big]\big],\\
\wt\cZ_{m,B}(\al_j,Q)&=\!\!\sum_{m'=0}^{\i}\!\frac{(m'\!+\!m)!}{m'!} 
\!\!\!\!\!\sum_{{\b\in\cP_{m'}(m-B+m')}}
\!\!\!\Bigg(\prod_{k=1}^{k=m'}\frac{(-1)^{b_k}}{b_k!}
\Res{\hb=0}\Big\{\hb^{-b_k}\wt\cZ^*(\hb,\al_j,Q)\Big\}\!\Bigg)
\in \Q_{\al}\big[\big[Q\big]\big],\notag
\end{alignat*}
for $m,B\!\in\!\bar\Z^+$.
Since the power series $\wt\cZ^*(\hb,\x,Q)$ has no $Q$-constant term,
the above sum is finite in each $Q$-degree.
It is shown Section~\ref{maincomp_subs} that the power series $\wt\cZ_{m,B}(\x,Q)$ 
describe the contributions of the strands~(S1) at a vertex~$v$ 
of the core of a tree with $m_v\!=\!m$
(with~$m_v$ computed with respect to the core).
Let 
\BE{cZ2dfn_e}\begin{split}
\cZ^*(\hb_1,\hb_2,\x_1,\x_2,Q)&=
\sum_{d=1}^{\i}Q^d\ev_*\bigg\{\frac{\E(\V)}{(\hb_1\!-\!\psi_1)(\hb_2\!-\!\psi_2)}\bigg\}
\in H_{\T}^*(\P_2^{n-1})\big[\big[\hb_1^{-1},\hb_2^{-1},Q\big]\big]\,,\\
\cZ(\hb_1,\hb_2,\x_1,\x_2,Q)&=-
\frac{\lr\a\x_1^l}{\hb_1\!+\!\hb_2}
\sum_{\begin{subarray}{c}p_1+p_2+r=n-1\\ p_1,p_2,r\ge0 \end{subarray}}
\!\!\!\!\!\!\!\!\!\hat\si_r\x_1^{p_1}\x_2^{p_2}+\cZ^*(\hb_1,\hb_2,\x_1,\x_2,Q)\,,
\end{split}\EE
where $\ev\!:\ov\M_{0,2}(\P^{n-1},d)\!\lra\!\P_2^{n-1}$ is 
the total evaluation map.

\begin{prp}\label{res_prp}
The power series~\e_ref{wtcZdfn_e} and~\e_ref{cZpdfn_e} admit expansions 
\begin{alignat}{1}
\label{cZexp_e}
\wt\cZ(\hb,\al_j,Q)&=e^{\ze(\al_j,Q)/\hb}\sum_{b=0}^{\i}\Psi_b(\al_j,Q)\hb^b\,,\\
\label{cZpexp_e}
\frac{\cZ_p(\hb,\al_j,Q)}{\prod\limits_{r=p-l+1}^{n-l-1}\!\!\!\!I_r(q)}
&=e^{\ze(\al_j,Q)/\hb}\sum_{b=0}^{\i}\Psi_{p;b}(\al_j,Q)\hb^b\,,
\end{alignat}
for some $\ze,\Psi_b,\Psi_{p;b}\!\in\!\Q_{\al}'[\x^{\pm1}][[Q]]$ such that 
\BE{PsiPhi_e}
\Psi_b(\x,Q)\sim\frac{\Phi_b(\q)}{I_0(\q)}\x^{-b}\,,\qquad
\Psi_{p;b}(\x,Q)\sim \frac{I_0(\q)\Phi_{p;b}(\q)}{L(\q)^{\de_{0\nua}n}}\x^{p-b}\,,
\EE
where $\q e^{\de_{0\nua}J(\q)}=Q/\x^{\nua}$.
\end{prp}

\begin{proof}
The existence of the expansion~\e_ref{cZexp_e} follows from 
Lemmas~2.2 and~2.3 in~\cite{bcov1},
but a direct argument is provided below and in Appendix~\ref{HGprp_pf}. 
Let
$$\cY(\hb,\x,q)=
\sum_{d=0}^{\i}q^d\frac{\prod\limits_{k=1}^{k=l}\prod\limits_{r=1}^{r=a_kd}(a_k\x+r\hb)}
{\prod\limits_{r=1}^{r=d}\left(\prod\limits_{k=1}^{k=n}\!\!(\x\!-\!\al_k\!+\!r\hb)
-\prod\limits_{k=1}^{k=n}\!\!(\x\!-\!\al_k)\right)}
\in \big(\Q_{\al;\hb,\x}'\!\cap\!\Q_{\al}[\x][[\hb^{-1}]]\big)\big[\big[Q\big]\big].$$
By \cite[Section~29.1]{MirSym},
\BE{MirSym_e0}
\wt\cZ(\hb,\x,Q)
=e^{-J(q)\frac{\x^{\de_{0\nua}}}{\hb}+f(q)\frac{\si_1}{\hb}}\frac{\cY(\hb,\x,q)}{I_0(q)}\EE
for some $f\!\in\!q\Q[[q]]$ (which is 0 unless $\nua\!=\!0$), where $qe^{\de_{0\nua}J(q)}=Q$.
Since 
$$\cY(\hb,\x,q)=\bigg\{1+\frac{\hb}{\x}q\frac{\tnd}{\tnd q}\bigg\}^l\cY_0(\hb,\x,q),$$
with $\cY_0(\hb,\x,q)$ given by \e_ref{cY0dfn_e},
Lemma~\ref{cY0_lmm} implies that $\cY(\hb,\x,q)$ admits an expansion of the form
\BE{cYexp_e}\cY(\hb,\x,q)=e^{\xi(\x,q)/\hb}\sum_{b=0}^{\i}\Phi_b(\x,q)\hb^b\EE
with $\xi(\x,q),\Phi_0(\x,q),\Phi_1(\x,q),\ldots\in\Q_{\al}(\x)[[q]]$.
Since 
\begin{gather*}
\xi(\x,q)=\Res{\hb=0}\big\{\ln\cY(\hb,\x,q)\big\}, \qquad
\Phi_b(\x,q)=\Res{\hb=0}\Big\{\hb^{-b-1}e^{-\xi(\x,q)/\hb}\cY(\hb,\x,q)\Big\},\\
\hbox{and}\qquad
\cY(\hb,\x,q)-F(w,\q)\in q\cdot \cI\Q_{\al;\hb,\x}'\,,
\end{gather*}
where $w\!=\!\x/\hb$, Proposition~\ref{Fexp_prp} and 
the first statement of Lemma~\ref{ressum_lmm} imply that there exist
$$\ti\xi(\x,q),\ti\Phi_0(\x,q),\ti\Phi_1(\x,q),\ldots\in\Q_{\al}'[\x^{\pm1}]\big[\big[q\big]\big]$$
 such that 
\BE{ticY_e}\begin{split}
\cY(\hb,\al_j,q)&=e^{\ti\xi(\al_j,q)/\hb}\sum_{b=0}^{\i}\ti\Phi_b(\al_j,q)\hb^b\, 
\quad \forall\,j\in[n],\\
\ti\xi(\x,q)&\sim\xi(\q)\x,\qquad 
\ti\Phi_b(\x,q)\sim \Phi_b(\q)\x^{-b}\quad\forall\,b\in\Z^+.
\end{split}\EE
By \e_ref{MirSym_e0} and \e_ref{ticY_e}, \e_ref{cZexp_e} and the first statement in~\e_ref{PsiPhi_e}
hold with 
$$ \ze(\x,Q)=\ti\xi(\x,q)-J(\q)\x+f(q)\si_1,\qquad
\Psi_b(\x,Q)=\frac{\ti\Phi_b(\x,\q)}{I_0(q)}=\frac{\ti\Phi_b(\x,\q)}{I_0(\q)}\,.$$
The existence of the expansion~\e_ref{cZpexp_e} follows from the existence of 
the expansion~\e_ref{cZexp_e} 
and the description of $\cZ_p(\hb,\x,Q)$ as a linear combination of the derivatives
of  $\wt\cZ(\hb,\x,Q)$ in \cite[Theorem~4]{PoZ}.
By \cite[Theorem~4]{PoZ},
$$\cZ_p(\hb,\x,Q)\sim e^{-J(\q)w}\x^p\frac{F_p(w,\q)}{I_{p-l}(\q)}\,.$$
Along with the first statement in Lemma~\ref{ressum_lmm} and \e_ref{Fpexp_e}, 
this gives  the second claim in~\e_ref{PsiPhi_e}.
\end{proof}

\begin{crl}\label{pt1_crl}
For all $m\!\in\!\bar\Z^+$ and $\c\!\in\!(\bar\Z^+)^{\i}$, there exists
$\Psi_{m,\c}\!\in\!\Q_{\al}'[\x^{\pm1}]\big[\big[Q\big]\big]$
such that 
\BE{cZmB_e}
\wt\cZ_{m,B}(\al_j,Q)=
\!\!\!\!\!\sum_{\c\in(\bar\Z^+)^{\i}}\!\!\!
\Bigg(\!\!(-1)^{m-\|\c\|}\binom{B}{m\!-\!\|\c\|}\ze(\al_j,Q)^{B-(m-\|\c\|)}
\Psi_{m,\c}(\al_j,Q)\!\!\Bigg)\EE
for all $B\!\in\!\bar\Z^+$ and $j\!\in\![n]$ and
\BE{PsimcPhi_e}
\Psi_{m,\c}(\x,Q)\sim
\bigg(\frac{I_0(\q)}{\Phi_0(\q)}\bigg)^{m+3}\Phi_{m,\c}(\q) \x^{-\|\c\|}\,,\EE
where $\q e^{\de_{0\nua}J(\q)}=Q/\x^{\nua}$.
\end{crl}

\begin{proof}
By Lemma~\ref{res_lmm} and \e_ref{cZexp_e}, \e_ref{cZmB_e} holds with 
\BE{Psimcdfn_e}\Psi_{m,\c}(\x,Q)=(-1)^{m+|\c|}(m\!+\!|\c|)!
\frac{1}{\Psi_0(\x,Q)^{m+1}}
\prod_{r=1}^{\i}
\frac{1}{c_r!}\bigg(\frac{1}{(r\!+\!1)!}\frac{\Psi_r(\x,Q)}{\Psi_0(\x,Q)}\bigg)^{c_r}\,.\EE
Along with the first statement in~\e_ref{PsiPhi_e} and \e_ref{Phimcdfn_e}, 
this implies \e_ref{PsimcPhi_e}.\end{proof}

\begin{lmm}\label{pt2_lmm}
There exists a collection 
$\{\cC_{p_-p_+}\}_{p_{\pm}\in\nset} \subset \Q[\al]\big[\big[Q\big]\big]$
such that 
\BE{pt2Res_e}\begin{split}
&\frac{1}{\lr\a}\Res{\hb_+=0}
\left\{\frac{1}{\hb_+^{1+b_+}}e^{-\frac{\ze(\al_{j_+},Q)}{\hb_+}}
\cZ(\hb_-,\hb_+,\al_{j_-},\al_{j_+},Q)\right\}\\
&\hspace{.5in}
=\sum_{b_-=0}^{b_-=b_+}\left( \frac{(-1)^{b_-}}{\hb_-^{b_-}}
\!\!\!\!\sum_{p_+,p_-\in\nset}\!\!\!\!\!\!\!
\cC_{p_-p_+}(Q)\Psi_{p_+;b_+-b_-}(\al_{j_+},Q)
\frac{\cZ_{p_-}(\hb_-,\al_{j_-},Q)}{\hb_-\!\!\prod\limits_{r=p_--l+1}^{n-l-1}\!\!\!\!\!I_r(q)}\right)
\end{split}\EE
for all $b_+\!\in\!\bar\Z^+$ and $j_-,j_+\!\in\![n]$ and 
\BE{p2cC_e} \cC_{p_-p_+}(Q)\sim
\begin{cases}
\frac{L(q)^{\de_{0\nua}(1+t)n}}{I_0(q)^2}\hat\si_n^t,
&\textnormal{if}~
p_-\!+\!p_+\!+\!nt\!=\!n\!-\!1\!+\!l,\,t\!=\!0,1,\\
0,&\textnormal{otherwise},\end{cases}\EE
where $q e^{\de_{0\nua}J(q)}=Q$.
\end{lmm}

\begin{proof}
By \cite[Theorem~4]{PoZ},
\BE{cZ2form_e}\begin{split}
&\cZ(\hb_-,\hb_+,\x_-,\x_+,Q)\\
&\qquad=\frac{\lr\a}{\hb_-\!+\!\hb_+} \Bigg\{
-\!\!\!\!\!\!\sum_{\begin{subarray}{c} p_-+p_++r=n-1+l\\ 
p_-,p_+\in\nset,r\in\bar\Z^+\\  p_-,p_+\ge l \end{subarray}}
~~~~
+\!\!\!\!\!\!\sum_{\begin{subarray}{c}
p_-+p_++r=n-1+l\\ p_-,p_+\in\nset,r\in\bar\Z^+\\ p_-,p_+<l \end{subarray}}
\Bigg\} \hat\si_r\cZ_{p_-}(\hb_-,\x_-,Q)\cZ_{p_+}(\hb_+,\x_+,Q).
\end{split}\EE
Combining this identity with \e_ref{cZpexp_e}, 
we find that \e_ref{pt2Res_e} holds with 
\BE{pt2cC_e} 
\cC_{p_-p_+}(Q)=
\bigg(\prod_{r=p_+-l+1}^{n-l-1}\!\!\!\!\!\!\!\!I_r(q)\bigg)
\bigg(\prod_{r=p_--l+1}^{n-l-1}\!\!\!\!\!\!\!\!I_r(q)\bigg)
\hat\si_{n-1+l-p_--p_+}
\cdot
\begin{cases} 1,&\hbox{if}~p_-,p_+\!<\!l;\\
-1,&\hbox{if}~p_-,p_+\!\ge\!l;\\
0,&\hbox{otherwise}.\end{cases}\EE
Along with the first two statements in Lemma~\ref{Iprp_lmm}, this implies \e_ref{p2cC_e}.
\end{proof}

\subsection{Main localization computation}
\label{maincomp_subs}

\noindent
We now prove Theorem~\ref{equiv_thm}, with each of the two definitions 
of the structure constants~$\nc_{\p,\b}^{(d,t)}$, by summing up the contributions
of the $\T$-fixed loci~$\cZ_{\Ga}$ of $\ov\M_{0,N}(\P^{n-1},d)$, with $d\!\in\!\bar\Z^+$.
As outlined in Section~\ref{outline_subs}, this will be done by breaking each~$\Ga$
(and correspondingly each fixed locus~$\cZ_{\Ga}$) at either one special vertex, $v\!=\!\bar\mu(N)$,
or at every special vertex of~$\Ga$.\\

\noindent
Let $\Ga$ be a decorated tree with $N$ marked points as in~\e_ref{decortgraphdfn_e}.
Let $\bar\Ga\equiv(\ov\Ver,\ov\Edg;\bar\mu,\bar\eta)$ be the core
of~$\Ga$ as in Section~\ref{outline_subs} and $v\!=\!\bar\eta(N)$.
Similarly to Figure~\ref{strands_fig}, we break $\Ga$ at the vertex
$v\in\ov\Ver\subset\Ver$ into strands~$\Ga_e$ indexed by the set~$\tnE_v(\Ga)$ 
of the edges with vertex~$v$ in~$\Ga$;
each strand~$\Ga_e$ keeps a copy of the vertex~$v$ and gains an extra marked point, 
which will be labeled~$e$, attached at~$v$.
For each $e\!\in\!\tnE_v(\Ga)$, denote by $S_e\!\subset\![N]$ the subset of 
the original marked points carried by the strand~$\Ga_e$.
Let 
\begin{equation*}\begin{split}
\tnE_v^*(\Ga)&=\big\{e\!\in\!\tnE_v(\Ga)\!:\,S_e\!\neq\!\eset\big\}\sqcup\eta^{-1}(v),
\quad
\tnE_v'(\Ga)=\big\{e\!\in\!\tnE_v(\Ga)\!:\,S_e\!=\!\eset\big\},\\
\ov\tnE_v(\Ga)&=\tnE_v^*(\Ga)\cup\tnE_v'(\Ga)\subset\tnE_v(\Ga)\sqcup[N].
\end{split}\end{equation*}
Thus, $|\tnE_v'(\Ga)|\!\ge\!0$, $|\tnE_v^*(\Ga)|\!\ge\!3$ 
(because $\bar\Ga$ is a trivalent tree), and 
$\{S_e\}_{e\in\tnE_v^*(\Ga)}\in\cP_{\tnE_v^*(\Ga)}([N])$,
where $S_e\!\equiv\!\{e\}$ if $e\!\in\!\eta^{-1}(v)$.\\

\noindent
The fixed locus $\cZ_{\Ga}$ corresponding to~$\Ga$, 
the restriction of $\E(\V)$ to~$\cZ_{\Ga}$, and 
the euler class of the normal bundle of~$\cZ_{\Ga}$ are given~by
\BE{Zreg_e5}\begin{split}
&\cZ_{\Ga}=\ov\cM_{0,\ov\tnE_v(\Ga)}\times\prod_{e\in\tnE_v(\Ga)}\!\!\!\!\cZ_{\Ga_e},
\qquad 
\frac{\E(\V)}{\E(\L_{\mu(v)})}=\prod_{e\in\tnE_v(\Ga)}\!\!\frac{\pi_e^*\E(\V)}{\E(\L_{\mu(v)})}\,,\\
&\frac{\E(T_{\mu(v)}\P^{n-1})}{\E(\N\cZ_{\Ga})}=
\prod_{e\in\tnE_v(\Ga)}\!\!\frac{\E(T_{\mu(v)}\P^{n-1})}
{\E(\N\cZ_{\Ga_e})\,(\hb_e'\!-\!\pi_e^*\psi_e)},
\end{split}\EE 
where $\ov\cM_{0,\ov\tnE_v(\Ga)}\approx\ov\cM_{0,|\tnE_v(\Ga)|+|\eta^{-1}(v)|}$ 
is the moduli space of 
stable rational $\ov\tnE_v(\Ga)$-marked curves,
$$\hb_e' \equiv c_1(L_e') \in H^*\big(\ov\cM_{0,\ov\tnE_v(\Ga)}\big)$$ 
is the first chern class of the universal tangent line bundle for the marked point 
corresponding to the edge~$e$, and
$$\pi_e\!:\cZ_{\Ga}\lra\cZ_{\Ga_e}
\subset \bigcup_{d_e=1}^{\i} \ov\M_{0,S_e\sqcup\{e\}}(\P^{n-1},d_e)$$ 
is the projection map.
By \cite[Section~27.2]{MirSym},
$$\psi_e|_{\cZ_{\Ga_e}}=\frac{\al_{\mu_v(e)}\!-\!\al_{\mu(v)}}{\d(e)}\,.$$
Thus, by \cite[Exercise~25.2.8]{MirSym},
\BE{cMint_e}\begin{split}
&\int_{\ov\cM_{0,\ov\tnE_v(\Ga)}}\!
\Bigg\{\!\!\Bigg(\prod_{e\in\tnE_v(\Ga)}\!\frac{1}{\hb_e'\!-\!\pi_e^*\psi_e}\Bigg)
\!\!\Bigg(\prod_{e\in\eta^{-1}(v)}\!\!\frac{1}{\hb_e\!-\!\psi_e}\Bigg)\!\!\Bigg\}\\
&\hspace{.5in}
=(-1)^{|\tnE_v(\Ga)|}\!\!\!\sum_{\b\in(\bar\Z^+)^{\ov\tnE_v(\Ga)}}\! 
\int_{\ov\cM_{0,\ov\tnE_v(\Ga)}}\!\!\Bigg\{\!\!
\Bigg(\prod_{e\in\tnE_v(\Ga)}\!\!\!\!\!\psi_e^{-b_e-1}\hb_e'^{\,b_e}\!\!\Bigg)\!\!
\Bigg(\prod_{e\in\eta^{-1}(v)}\!\!\!\!\!\!\hb_e^{-b_e-1}\psi_e^{b_e}\!\!\Bigg)
\!\!\Bigg\}\\
&\hspace{.5in}
=\!\!\!
\sum_{\b\in(\bar\Z^+)^{\ov\tnE_v(\Ga)}}\!\!
\Bigg\{\!\!\binom{|\ov\tnE_v(\Ga)|\!-\!3}{\b}
\!\Bigg(\prod_{e\in\tnE_v(\Ga)}\!\!\! 
\bigg(\frac{\al_{\mu(v)}\!-\!\al_{\mu_v(e)}}{\d(e)}\bigg)^{\!\!-b_e-1}\!\Bigg)\!\!
\Bigg(\prod_{e\in\eta^{-1}(v)}\!\!\!\!\!\!\hb_e^{-b_e-1}\!\!\Bigg)
\!\Bigg\}.
\end{split}\EE
Combining this with \e_ref{Zreg_e5}, \e_ref{tangrestr_e}, and \e_ref{phidfn_e}, we obtain
\BE{Zreg_e6}\begin{split}
&\frac{\prod\limits_{k\neq\mu(v)}\!\!\!\!(\al_{\mu(v)}\!-\!\al_k)}{\lr\a\al_{\mu(v)}^l}
\int_{\cZ_{\Ga}}\!\!
\frac{\E(\V)}{\E(\N\cZ_{\Ga})}
\prod_{s=1}^{s=N}\!\!\bigg(\frac{\ev_s^*\phi_{i_s}}{\hb_s\!-\!\psi_s}\bigg)\\
&\hspace{.3in}
=\!\!\!\sum_{\b\in(\bar\Z^+)^{\ov\tnE_v(\Ga)}}
\!\!\!\Bigg\{\!\!\binom{|\ov\tnE_v(\Ga)|\!-\!3}{\b}
\prod_{s\in\eta^{-1}(v)}\!\!\!
\bigg(\hb_s^{-b_s-1}\!\!\prod_{k\neq i_s}\!(\al_{\mu(v)}\!-\!\al_k)\!\!\bigg)\\
&\hspace{1.3in}
\times
\!\!\prod_{e\in\tnE_v(\Ga)}\!\!\! \Bigg(\!\!\!
\bigg(\frac{\al_{\mu(v)}\!-\!\al_{\mu_v(e)}}{\d(e)}\bigg)^{\!\!-b_e-1}
\!\!\!\int_{\cZ_{\Ga_e}}\!\!\!
\frac{\E(\V)\ev_e^*\phi_{\mu(v)}}{\lr\a\al_{\mu(v)}^l\E(\N\cZ_{\Ga_e})}
\prod_{s\in S_e}\!\!\bigg(\frac{\ev_s^*\phi_{i_s}}{\hb_s\!-\!\psi_s}\bigg)
\!\!\!\Bigg)\!\!\Bigg\}\,.
\end{split}\EE
The equality holds after dividing the right-hand side
by the order of the appropriate group of symmetries; see~\cite[Section 27.3]{MirSym}.
This group is taken into account in the next paragraph.\\

\noindent
We now sum up \e_ref{Zreg_e6} over all possibilities for $\Ga$.
If $e\!\in\!\tnE_v'(\Ga)$,
$$\frac{\E(\V)}{\lr\a \x^l}=\E(\V'),$$
with $\V'\!=\!\V'_{|\Ga_e|}$ as in~\e_ref{cZprdfn_e}.
Thus, in this case, by \cite[Section~2.2]{bcov1}
\BE{Zreg_e8}\begin{split}
&\sum_{\Ga_e}Q^{|\Ga_e|}
\bigg(\frac{\al_{\mu_v(e)}\!-\!\al_{\mu(v)}}{\d(e)}\bigg)^{\!\!-b_e-1}
\!\!\!\int_{\cZ_{\Ga_e}}\!\!\!
\frac{\E(\V)\ev_e^*\phi_{\mu(v)}}{\lr\a\al_{\mu(v)}^l\E(\N\cZ_{\Ga_e})}
\prod_{s\in S_e}\!\!\bigg(\frac{\ev_s^*\phi_{i_s}}{\hb_s\!-\!\psi_s}\bigg)\\
&\hspace{.3in}
=\sum_{\Ga_e}Q^{|\Ga_e|}
\bigg(\frac{\al_{\mu_v(e)}\!-\!\al_{\mu(v)}}{\d(e)}\bigg)^{\!\!-b_e-1}
\!\!\!\int_{\cZ_{\Ga_e}}\!\!\!
\frac{\E(\V')\ev_e^*\phi_{\mu(v)}}{\E(\N\cZ_{\Ga_e})}
=-\Res{\hb_e=0}\Big\{\hb_e^{-b_e}\wt\cZ^*(\hb_e,\al_{\mu(v)},Q)\Big\}\,,
\end{split}\EE
where the sum is taken over all possibilities for the strand $\Ga_e$, 
leaving the vertex~$v$, with $\mu(v)$ fixed.
By a similar reasoning, if $e\!\in\!\tnE_v^*(\Ga)$,
\BE{Zreg_e8b}\begin{split}
&\sum_{\Ga_e}Q^{|\Ga_e|}
\bigg(\frac{\al_{\mu_v(e)}\!-\!\al_{\mu(v)}}{\d(e)}\bigg)^{\!\!-b_e-1}
\!\!\!\int_{\cZ_{\Ga_e}}\!\!\!
\frac{\E(\V)\ev_e^*\phi_{\mu(v)}}{\lr\a\al_{\mu(v)}^l\E(\N\cZ_{\Ga_e})}
\prod_{s\in S_e}\!\!\bigg(\frac{\ev_s^*\phi_{i_s}}{\hb_s\!-\!\psi_s}\bigg)\\
&\hspace{.3in}
=-\frac{1}{\lr\a\al_{\mu(v)}^l}\Res{\hb_e=0}\Big\{\hb_e^{-b_e-1}
\cZ^*\big((\hb_s)_{s\in S_e},\hb_e,(\x_s\!=\!\al_{i_s})_{s\in S_e},\x_e\!=\!\al_{\mu(v)},Q\big)\Big\},
\end{split}\EE
where the sum is taken over all possibilities for the strand $\Ga_e$,  
leaving the vertex~$v$, with $\mu(v)$ fixed, $|\Ga_e|\!>\!0$,
and carrying the marked points $S_e\!\subset\![N]$,
and $\cZ^*$ is the positive-degree part of the power series~\e_ref{cZval_e} 
with $[N]$ replaced by $S_e\!\sqcup\!\{e\}$ if $|S_e|\!\ge\!2$
(for $|S_e|\!=\!1$, $\cZ^*$ is defined in~\e_ref{cZ2dfn_e}).\footnote{By the proof of 
\cite[Chapter 30, (3.21)]{MirSym}, LHS of~\e_ref{Zreg_e8} summed over 
$\Ga_e$ with $\d(e)\!=\!d$ and $\mu_v(e)\!=\!i$ fixed is the residue of 
$\hb^{-b}\wt\cZ^*(\hb,\al_{\mu(v)},Q)$ at $\hb\!=\!(\al_i\!-\!\al_{\mu(v)})/d$; 
see also \cite[Section~2.2]{bcov1}.
Since $\wt\cZ^*(\hb,\al_{\mu(v)},Q)$ vanishes to second order at $\hb\!=\!\i$,
$\hb^{-b}\wt\cZ^*(\hb,\al_{\mu(v)},Q)\tnd\hb$ has no residue at $\hb\!=\!\i$ 
for all $b\!\in\!\bar\Z^+$.
Since $\wt\cZ^*(\hb,\al_{\mu(v)},Q)\tnd\hb$ has poles only at $\hb\!=\!(\al_i\!-\!\al_{\mu(v)})/d$
with $i\!\in\![n]\!-\!\mu(v)$ and $d\!\in\!\Z^+$, and at~$\hb\!=\!0$,
\e_ref{Zreg_e8} follows from the Residue Theorem on~$S^2$.
By \e_ref{cZ2form_e}, the same reasoning applies to $\hb^{-1}\cZ^*(\hb_s,\hb,\al_{i_s},\al_{\mu(v)},Q)$,
giving the $|S_e|\!=\!1$ case of~\e_ref{Zreg_e8b}.
Since $|\tnE_v^*(\Ga)|\!\ge\!3$, $|S_e\!\sqcup\!\{e\}|\!<\!N$; 
by Theorem~\ref{equiv_thm} and induction on~$N$, the same reasoning is applicable to
\e_ref{Zreg_e8b} for $|S_e|\!\ge\!2$ as well.\label{ressum_ftnt}}
Finally, if $s\!\in\!\eta^{-1}(v)$, 
\BE{cZ2deg0_e}\begin{split}
\hb_s^{-b_s-1}\!\!\prod_{k\neq i_s}\!(\al_{\mu(v)}\!-\!\al_k)
=\frac{(-1)^{b_s}}{\lr\a\al_{\mu(v)}^l}\Res{\hb_e=0}\bigg\{\hb_e^{-b_s-1}
\!\LRbr{\cZ\big(\hb_s,\hb_e,\al_{i_s},\al_{\mu(v)},Q\big)}_{Q;0}\bigg\}\,.
\end{split}\EE
This corresponds to the strand $\Ga_e$ in \e_ref{Zreg_e8b} with $|\Ga_e|\!=\!0$ 
whenever $S_e\!=\!\{s\}$ is a single-element set.
On the other hand, if $|S_e|\!\ge\!2$, 
$$\Res{\hb_e=0}\Big\{\hb_e^{-b_e-1}
\!\LRbr{\cZ\big((\hb_s)_{s\in S_e},\hb_e,(\x_s\!=\!\al_{i_s})_{s\in S_e},
\x_e\!=\!\al_{\mu(v)},Q\big)}_{Q;0}\Big\}=0.$$\\

\noindent
Putting this all together, 
taking into account the group of symmetries (permutations of the one-marked strands), 
and summing over all possibilities for $m'\!\equiv\!|\tnE_v'(\Ga)|$,
while keeping  
$$m\equiv|\tnE_v^*(\Ga)|\ge3, \qquad 
\{S_i\}_{i\in[m]}\equiv\{S_e\}_{e\in\tnE_v^*(\Ga)}\in\cP_m([N]), 
\quad\hbox{and}\quad j\equiv \eta(v)\in[n]\,$$ 
fixed, we find that 
\BE{cZsum_e1}\begin{split}
&\frac{\prod\limits_{k\neq j}\!(\al_j\!-\!\al_k)}{\lr\a\al_j^l}
\sum_{\Ga}Q^{|\Ga|}\!\!\int_{\cZ_{\Ga}}\!\!
\frac{\E(\V)}{\E(\N\cZ_{\Ga})}
\prod_{s=1}^{s=N}\!\!\bigg(\frac{\ev_s^*\phi_{i_s}}{\hb_s\!-\!\psi_s}\bigg)
=\sum_{\b\in(\bar\Z^+)^m}\!\!\Bigg\{\wt\cZ_{m-3,\|\b\|}(\al_j,Q)\\
&\hspace{1.1in}\times
\prod_{i=1}^{i=m}\!\bigg(\frac{1}{\lr\a\al_j^l}
\frac{(-1)^{b_i}}{b_i!}\Res{\hb_i'=0}\Big\{\hb_i'^{\,-b_i-1}
\cZ\big((\hb_s)_{s\in S_i},\hb_i',(\al_{i_s})_{s\in S_i},\al_j,Q\big)\!\Big\}\!\!\bigg)
\!\!\Bigg\}\,.
\end{split}\EE
By \e_ref{cZmB_e} and the first statement of Lemma~\ref{comb_l0}, 
the right-hand side of this expression reduces~to
\begin{equation*}\begin{split}
&\sum_{\b\in(\bar\Z^+)^m}\!\!\!\!\sum_{\begin{subarray}{c}\b''\in(\bar\Z^+)^m\\ 
\c\in(\bar\Z^+)^{\i}\\ |\b''|+\|\c\|=m-3\end{subarray}}\!\!\!\!\!\!\!
\Bigg\{\! \Psi_{m-3,\c}(\al_j,Q)\\
&\hspace{.5in} \times\prod_{i=1}^{i=m}\!\!\Bigg(
\frac{1}{\lr\a\al_j^l}\frac{1}{b_i!}\binom{b_i}{b_i''}
\Res{\hb_i'=0}\bigg\{\hb_i'^{\,-b_i''-1}
\bigg(\!\!-\frac{\ze(\al_j,Q)}{\hb_i'}\bigg)^{b_i-b_i''}\!\!
\cZ\big((\hb_s)_{s\in S_i},\hb_i',(\al_{i_s})_{s\in S_i},\al_j,Q\big)\!\bigg\}
\!\!\Bigg)\!\!\Bigg\}\\
&=\!\!\!\!\!\sum_{\begin{subarray}{c}\b''\in(\bar\Z^+)^m\\ 
\c\in(\bar\Z^+)^{\i}\\ |\b''|+\|\c\|=m-3\end{subarray}}\!\!\!\!\!\!\!
\Bigg\{\! \Psi_{m-3,\c}(\al_j,Q)\!
\prod_{i=1}^{i=m}\!\!\Bigg(
\frac{1}{\lr\a\al_j^l}\frac{1}{b_i''!}
\Res{\hb=0}\bigg\{\frac{e^{-\frac{\ze(\al_j,Q)}{\hb}}}{\hb^{\,b_i''+1}}
\cZ\big((\hb_s)_{s\in S_i},\hb,(\al_{i_s})_{s\in S_i},\al_j,Q\big)\!\bigg\}
\!\!\Bigg)\!\!\Bigg\}.
\end{split}\end{equation*}
Since $m\!\ge\!3$, $|S_i|\!\le\!N\!-\!2$ for every $i\!\in\![m]$.
Thus, each of the power series~$\cZ^*$ appearing in the last expression above is
described either by \e_ref{cZ2form_e} or Theorem~\ref{equiv_thm} with~$N$ replaced by
$|S_i|\!+\!1\!<\!N$ (which we can assume to hold by induction).
By  the last expression for the left-hand side of~\e_ref{cZsum_e1},
Lemma~\ref{pt2_lmm}, \e_ref{equivthm_e} with $N$ replaced by $|S_i|\!+\!1\!<\!N$
whenever $|S_i|\!\ge\!2$, and~\e_ref{cZpexp_e}, 
the sum on the left hand-side side of~\e_ref{cZsum_e1} equals
\begin{equation*}\begin{split}
\lr\a \!\!\!\!\!\!\!
\sum_{\begin{subarray}{c}\p\in\nset^N\\ \b\in(\bar\Z^+)^N\end{subarray}}
\!\!\!\Bigg\{\! 
\frac{\un\hb^{-\b}\!\cZ_{\p}(\un\hb,\al_{i_1\ldots i_N},Q)}
{\al_j^{l(m-1)}\!\!\prod\limits_{k\neq j}\!(\al_j\!-\!\al_k)}
\!\!
\sum_{\begin{subarray}{c}\bfd\in(\bar\Z^+)^m\\ \p'\in\nset^m\\  \b'\in\Z^m\end{subarray}}
\!\!\!\sum_{\begin{subarray}{c}\b''\in(\bar\Z^+)^m\\ 
\c\in(\bar\Z^+)^{\i}\\ |\b''|+\|\c\|=m-3\end{subarray}}\!\!\!\!\!\!\!\!\!\!\!\!
 q^{|\bfd|}\Psi_{m-3,\c}(\al_j,Q)
\prod_{i=1}^{i=m}
\frac{\cC_{\p|_{S_i}p_i',\b|_{S_i}b_i'}^{(d_i)}\Psi_{p_i';b_i'+1+b_i''}(\al_j,Q)}{b_i''!}\Bigg\}&
\end{split}\end{equation*}
with $\Psi_{p;b}\!\equiv\!0$ if $b\!<\!0$.
In the two-pointed case (for $|S_i|\!=\!1$), the above structure constants are given~by 
\BE{cC2base_e}\sum_{d=0}^{\i}q^d\cC_{pp',bb'}^{(d)}
=\de_{b+b',-1}(-1)^b\cC_{pp'}(Q),\EE
with $\cC_{pp'}$ as in \e_ref{pt2cC_e}.
Summing over all 
$$j\in[n]\,, \qquad \bfS\!\equiv\!\{S_i\}_{\in[m]}\in\cP_m([N])\,, \quad\hbox{and}\quad m\ge3\,$$
and using the Residue Theorem on $S^2$,
we obtain a recursion for the coefficients~$\cC_{\p,\b}^{(d)}$ in Theorem~\ref{equiv_thm}:
\BE{cCrec_e}\begin{split}
\cC_{\p,\b}^{(d)}&=-
\sum_{\begin{subarray}{c}m,d'\in\bar\Z^+\\ m\ge3\end{subarray}}
\sum_{\begin{subarray}{c} \bfS\in\cP_m([N])\\ 
\bfd\in\cP_m(d-d')\\ (\p',\b')\in\nset^m\times\Z^m\end{subarray}}
\sum_{\begin{subarray}{c}\b''\in(\bar\Z^+)^m\\ 
\c\in(\bar\Z^+)^{\i}\\ |\b''|+\|\c\|=m-3\end{subarray}}\!\!\!\!\!\!\\
&\hspace{1in}
\Res{\x=0,\i}\!\! 
\LRbr{ \frac{\Psi_{m-3,\c}(\x,Q)}{\x^{l(m-1)}\prod\limits_{k=1}^{k=n}\!(\x\!-\!\al_k)}
\!\prod_{i=1}^{i=m}
\frac{\cC_{\p|_{S_i}p_i',\b|_{S_i}b_i'}^{(d_i)}\Psi_{p_i';b_i'+1+b_i''}(\x,Q)}{b_i''!} 
\!}_{q;d'}\,.
\end{split}\EE\\

\noindent
By~\e_ref{equivthm_e} and~\e_ref{cZpdfn_e}, if $\b\!\in\!(\bar\Z^+)^N$ and $d\!\in\bar\Z^+$,
the coefficient of 
$$q^d \prod_{s=1}^{s=N}\!\big((\hb_s^{-1})^{b_s+1}\big)$$
in the power series $\cZ(\un\hb,\un\x,Q)$ is
\BE{cZcoeff_e}\begin{split}
\LRbr{\cZ(\un\hb,\un\x,Q)}_{\un\hb^{-1},q;\b+\mathbf1,d}
&=\sum_{\p\in\nset^N}\!\!\!\cC_{\p,\b}^{(d)}\un\x_s^{\p}\\
&\quad
+\!\!\!\sum_{\begin{subarray}{c}d'\in\dset\\ \bfd\in\cP_N(d-d')\end{subarray}}
\sum_{\p\in\nset^N}
\sum_{\begin{subarray}{c} \b'\in(\bar\Z^+)^N\\ b_s'\le b_s\end{subarray}} 
\!\!\cC_{\p,\b'}^{(d')}
\!\!\prod_{s=1}^{s=N}\!\! \LRbr{\cZ_{(p_s)}(\hb_s,\x_s,Q) }_{\hb_s^{-1},q;b_s-b_s',d_s}\,,
\end{split}\EE
where $\LRbr{\cZ_{(p)}(\hb,\x,Q) }_{\hb^{-1},q;b,d'}$
is the coefficient of $q^{d'}(\hb^{-1})^b$ in
$$\cZ_{(p)}(\hb,\x,Q)\equiv\frac{\cZ_p(\hb,\x,Q)}{\prod\limits_{r=p_s-l+1}^{n-l-1}\!\!\!\!I_r(q_s)}
\in H_{\T}^*\big(\P^{n-1}\big)\big[\big[\hb^{-1},Q\big]\big]
=H_{\T}^*\big(\P^{n-1}\big)\big[\big[\hb^{-1},Q\big]\big].$$
Since $H_{\T}^*(\P^{n-1})$ and $H_{\T}^*(\P_N^{n-1})$ are free modules over $\Q[\al]$
with bases $\{\x^p\}_{p\in\nset}$ and $\{\un\x^{\p}\}_{\p\in\nset^N}$, respectively,
and
$$\LRbr{\cZ_{(p)}(\hb,\x,Q) }_{\hb^{-1},q;b,d'}\in H_{\T}^*(\P^{n-1}),\qquad
\LRbr{\cZ(\un\hb,\un\x,Q)}_{\un\hb^{-1},q;\b+\mathbf1,d}\in H_{\T}^*(\P_N^{n-1})$$
by \e_ref{cZpdfn_e} and \e_ref{cZdfn_e}, 
\e_ref{cZcoeff_e} and induction on~$d$ imply that 
$\cC_{\p,\b}^{(d)}\in\Q[\al]$ as claimed in Theorem~\ref{equiv_thm}.\\

\noindent 
We now confirm \e_ref{equivthm_e2} by induction on $N$.
For $N\!=\!2$, \e_ref{equivthm_e2} holds by \e_ref{cC2base_e}, \e_ref{p2cC_e}, and \e_ref{coeff2_e}.
On the other hand, by \e_ref{cCrec_e}, \e_ref{PsimcPhi_e}, 
the second statement in~\e_ref{PsiPhi_e}, 
the inductive assumption~\e_ref{equivthm_e2},  
and the last two statements in Lemma~\ref{ressum_lmm}, 
\begin{equation*}\begin{split}
&\cC_{\p,\b}^{(d)}\sim-\!\!\!
\sum_{\begin{subarray}{c}m,d'\in\bar\Z^+\\ m\ge3\end{subarray}}
\sum_{\begin{subarray}{c} \bfS\in\cP_m([N])\\ 
\bfd\in\cP_m(d-d')\\ \bft\in(\bar\Z^+)^m\\
(\p',\b')\in\nset^m\times\Z^m\end{subarray}}
\!\sum_{\begin{subarray}{c}\b''\in(\bar\Z^+)^m\\ 
\c\in(\Z^+)^{\i}\\ |\b''|+\|\c\|=m-3\end{subarray}}\!\!\!\!\hat\si_n^{|\bft|}\\
&\hspace{.1in}
\Res{\x=0,\i}\!\! \LRbr{ 
\frac{\x^{|\p'|-|\b'|-(l+2)(m-1)+1}}{\prod\limits_{k=1}^{k=n}\!(\x\!-\!\al_k)}
\Phi_{m-3,\c}(\q)
\!\prod_{i=1}^{i=m}\!\!\Bigg(\!\! \nc_{\p|_{S_i}p_i',\b|_{S_i}b_i'}^{(d_i,t_i)}
\frac{}{}
\frac{I_0(\q)^2\Phi_{p_i';b_i'+1+b_i''}(\q)}{b_i''!\,L(\q)^{\de_{0\nua}n}\Phi_0(\q)} 
\!\!\Bigg)\!}_{q;d'}\,.
\end{split}\end{equation*}
Since $\q\!=\!q/\x^{\nua}$, by the last statement of Lemma~\ref{ressum_lmm}
the negative of the expression on the last line is equivalent~to 
$$\LRbr{\Phi_{m-3,\c}(q)
\!\prod_{i=1}^{i=m}\!\!\Bigg(\!\! \nc_{\p|_{S_i}p_i',\b|_{S_i}b_i'}^{(d_i,t_i)}
\frac{I_0(q)^2\Phi_{p_i';b_i'+1+b_i''}(q)}{b_i''!\,L(q)^{\de_{0\nua}n}\Phi_0(q)} 
\!\!\Bigg)\!}_{q;d'}\hat\si_n^{t'}$$
with $t'\!\in\!\Z$ defined by 
$$|\p'|-|\b'|-(l+2)(m-1)+1-\nua d'=n-1+nt' \qquad\Llra\qquad 
(\p',\b')\in \cS_m(d',t');$$
if such an integer $t'$ does not exist, the above residue is equivalent to~0.
Since $\cC_{\p,\b}^{(d)}\!\in\!\Q[\al]$ by the previous paragraph, we conclude that 
\begin{equation*}\begin{split}
&\cC_{\p,\b}^{(d)}\sim \sum_{t=0}^{\i}\hat\si_n^t
\sum_{\begin{subarray}{c}m,d',t'\in\Z\\ m\ge3\end{subarray}}
\sum_{\begin{subarray}{c} \bfS\in\cP_m([N])\\ 
\bfd\in\cP_m(d-d')\\ \bft\in\cP_m(t-t')\\
(\p',\b')\in\cS_m(d',t')\end{subarray}}
\!\sum_{\begin{subarray}{c}\b''\in(\bar\Z^+)^m\\ 
\c\in(\Z^+)^{\i}\\ |\b''|+\|\c\|=m-3\end{subarray}}\!\!\!\!
\left(\bigg(\prod_{i=1}^{i=m}\nc_{\p|_{S_i}p_i',\b|_{S_i}b_i'}^{(d_i,t_i)}\bigg)\right.\\
&\hspace{2in}\times
\left.\LRbr{\Phi_{m-3,\c}(q)
\!\prod_{i=1}^{i=m}
\frac{I_0(q)^2\Phi_{p_i';b_i'+1+b_i''}(q)}{b_i''!\,L(q)^{\de_{0\nua}n}\Phi_0(q)}\!\!}_{q;d'}\right).
\end{split}\end{equation*}
Comparing this expression with \e_ref{coeffdfn_e}, 
we conclude that \e_ref{equivthm_e2} holds.\footnote{As can be seen by induction on $n$, $\cI\Q_{\al}'\cap\Q[\al]=\cI$.
Since $\cC_{\p,\b}^{(d)}$ is a symmetric function in $\al_1,\ldots,\al_n$,
it is even sufficient to check that the symmetric polynomials in $\cI\Q_{\al}'\cap\Q[\al]$
are contained in~$\cI$; this is immediate from the algebraic independence of 
the elementary symmetric functions.}\\

\noindent
We next show that \e_ref{equivthm_e2} holds with the coefficients $\nc_{\p,\b}^{(d,t)}$
as defined in~\e_ref{ncCdfn2_e}. 
Let $\Ga$ be an $N$-marked decorated tree and $\bar\Ga$ its core as before, 
with a partial ordering~$\prec$ as in Section~\ref{graph_subs}.
This time, we  break $\Ga$ and $\bar\cZ_{\Ga}$ at all vertices 
$\ov\Ver\!\subset\!\Ver$ of~$\bar\Ga$, adding a marked point to each of the strands;
see Figure~\ref{strands_fig}.
There are now three types of strands, (S1)-(S3), described in Section~\ref{outline_subs}.
Each strand of type~(S3) carries one of the original marked points $s\!\in\![N]$
and an added marked point~$s'$, which we associate with the element of $\tnE_v(\Ga)$
that leaves~$v$ in the direction of~$\eta(s)$.
These strands are thus naturally indexed by the complement of the subset $\eta^{-1}(\ov\Ver)\!\subset\![N]$
of the marked points attached to a vertex of the core in~$\Ga$.
Each strand of type~(S2) runs between vertices in $\ov\Ver\!\subset\!\Ver$ in~$\Ga$
that are joined by an edge $e\!=\!\{v_e^-,v_e^+\}$ in~$\bar\Ga$, with $v_e^-\!\prec\!v_e^+$.
It carries two added marked points, which we label~$e^-$ and~$e^+$, attached 
to the vertices $v_e^-$ and $v_e^+$, respectively, in the strand~$\Ga_e$.
We associate the marked point $e^-$ (resp.~$e^+$) with the element of 
$\tnE_{v_e^-}(\Ga)$ (resp.~$\tnE_{v_e^+}(\Ga)$)
that leaves~$v_e^-$ (resp.~$v_e^+$) in the directions of~$v_e^+$ (resp.~$v_e^-$).
Similarly to the first approach, for each $v\!\in\!\ov\Ver$, denote by
$\tnE_v'(\Ga)\!\subset\!\tnE_v(\Ga)$ the set of one-marked edges at~$v$ and set
\begin{equation*}\begin{split}
\bar\tnE_v(\Ga)&=\tnE_v'(\Ga)\cup\eta^{-1}(v) \cup \tnE_v(\bar\Ga)
\subset \tnE_v(\Ga)\sqcup[N], \qquad
\bar\tnE(\Ga)=\bigsqcup_{v\in\ov\Ver}\!\!\!\bar\tnE_v(\Ga)\,.
\end{split}\end{equation*}
As before, this set indexes the marked points on the contracted component.\\

\noindent
The analogues of the decompositions~\e_ref{Zreg_e5} in this case are 
\begin{equation*}\begin{split}
\cZ_{\Ga}&=\prod_{v\in\ov\Ver}\!\!\bigg(\ov\cM_{0,\ov\tnE_v(\Ga)}
\times\!\prod_{e\in\tnE_v'(\Ga)}\hspace{-.12in}\cZ_{\Ga_e}\bigg)
\times\prod_{e\in\ov\Edg}\!\!\!\cZ_{\Ga_e}\,,\\
\frac{\E(\V)}{\prod\limits_{v\in\ov\Ver}\E(\L_{\bar\mu(v)})}&=
\prod_{v\in\ov\Ver}\prod_{e\in\tnE_v'(\Ga)}\!\!\frac{\pi_e^*\E(\V)}{\E(\L_{\bar\mu(v)})}
\times\prod_{e\in\ov\Edg}\!\frac{\pi_e^*\E(\V)}{\E(\L_{\bar\mu(v_e^-)})\E(\L_{\bar\mu(v_e^+)})}\,,\\
\frac{\prod\limits_{v\in\ov\Ver}\E(T_{\bar\mu(v)}\P^{n-1})}{\E(\N\cZ_{\Ga})}&=
\prod_{v\in\ov\Ver}\!\!\Bigg( \prod_{e\in\tnE_v(\Ga)}\!\!\!\!
 \frac{\E(T_{\bar\mu(v)}\P^{n-1})}{\hb_e'\!-\!\pi_e^*\psi_e}
\times \!\!\prod_{e\in\tnE_v'(\Ga)}\!\!
\frac{1}{\E(\N\cZ_{\Ga_e})}\Bigg)
\times\prod_{e\in\ov\Edg}\frac{1}{\E(\N\cZ_{\Ga_e})},
\end{split}\end{equation*} 
For each $v\!\in\!\ov\Ver$, \e_ref{cMint_e} still applies.
The analogue of~\e_ref{Zreg_e6}, but weighted by the automorphism group, is~then
\BE{decomp2_e2}\begin{split}
&\left(\prod_{v\in\ov\Ver}
\frac{\prod\limits_{k\neq\bar\mu(v)}\!\!\!\!(\al_{\bar\mu(v)}\!-\!\al_k)}{\lr\a\al_{\bar\mu(v)}^l}\right)
\int_{\cZ_{\Ga}}\!\!
\frac{\E(\V)}{\E(\N\cZ_{\Ga})}
\prod_{s=1}^{s=N}\!\!\bigg(\frac{\ev_s^*\phi_{i_s}}{\hb_s\!-\!\psi_s}\bigg)\\
&\hspace{.1in}
=\hspace{-.1in}\sum_{\begin{subarray}{c}\b\in(\bar\Z^+)^{\ov\tnE(\Ga)}\\
|\b|_{\bar\tnE_v(\Ga)}=|\ov\tnE_v(\Ga)|-3\end{subarray}} 
\hspace{-.27in}\left\{\! \prod_{v\in\ov\Ver}\!\left(\!\!
\frac{(|\ov\tnE_v(\Ga)|\!-\!3)!}{|\tnE_v'(\Ga)|!}
\!\!\!\prod_{e\in\tnE_v'(\Ga)}\!\!\! \Bigg(\!\!\frac{1}{b_e!}\!
\bigg(\frac{\al_{\bar\mu(v)}\!-\!\al_{\mu_v(e)}}{\d(e)}\bigg)^{\!\!-b_e-1}
\!\!\!\!\int_{\cZ_{\Ga_e}}\!\!\!
\frac{\E(\V)\ev_e^*\phi_{\bar\mu(v)}}{\lr\a\al_{\bar\mu(v)}^l\E(\N\cZ_{\Ga_e})}\!\Bigg) \right.\right.\\
&\hspace{1.4in}\left.\times
\!\!\prod_{s\in\bar\eta^{-1}(v)}\!\!\!\Bigg(\!\!\frac{1}{b_{s'}!}
\bigg(\frac{\al_{\bar\mu(v)}\!-\!\al_{\mu_v(s')}}{\d(s')}\bigg)^{\!\!-b_{s'}-1}
\!\!\!\!\int_{\cZ_{\Ga_{s'}}}\!\!\!\frac{\E(\V)\ev_{s'}^*\phi_{\bar\mu(v)}\ev_s^*\phi_{i_s}}
{\lr\a\al_{\bar\mu(v)}^l\E(\N\cZ_{\Ga_{s'}})(\hb_s\!-\!\psi_s)}\!\!\Bigg)\!\!\right)\\
&\hspace{.5in}\times\left.
\!\!\prod_{e\in\ov\Edg}\!\! \Bigg(\!\frac{1}{b_{e^-}!b_{e^+}!}\!\!\prod_{*=-,+}\!\!\!
\bigg(\frac{\al_{\bar\mu(v_e^*)}\!-\!\al_{\mu_{v_e^*}(e^*)}}{\d(e^*)}\bigg)^{\!\!-b_{e^*}-1}
\!\!\!\times\!\!\int_{\cZ_{\Ga_e}}\!\!\!
\frac{\E(\V)\ev_{e^-}^*\phi_{\bar\mu(v_e^-)}\ev_{e^+}^*\phi_{\bar\mu(v_e^+)}}
{\lr\a^2\al_{\bar\mu(v_e^-)}^l\al_{\bar\mu(v_e^+)}^l\E(\N\cZ_{\Ga_e})}
\!\!\Bigg)\!\!\right\},
\end{split}\EE
where
$$\frac{1}{b_{s'}!}\bigg(\frac{\al_{\bar\mu(v)}\!-\!\al_{\mu_v(s')}}{\d(s')}\bigg)^{\!\!-b_{s'}-1}
\!\!\!\!\int_{\cZ_{\Ga_{s'}}}\!\!\!
\frac{\E(\V)\ev_{s'}^*\phi_{\bar\mu(v)}
\ev_s^*\phi_{i_s}}{\lr\a\al_{\bar\mu(v)}^l\E(\N\cZ_{\Ga_{s'}})(\hb_s\!-\!\psi_s)}
\equiv\frac{1}{b_s!}
\bigg(\hb_s^{-b_s-1}\!\!\prod_{k\neq i_s}\!(\al_{\bar\mu(v)}\!-\!\al_k)\!\!\bigg)$$
if $s\!\in\!\eta^{-1}(v)$.\\

\noindent
For each $v\!\in\!\ov\Ver$, \e_ref{Zreg_e8} still reduces the summation 
of~the factor on the second line in~\e_ref{decomp2_e2} over all possibilities
for $\Ga_e$ with $e\!\in\!\tnE_v'(\Ga)$ and for $m_v'\!\equiv\!|\tnE_v'(\Ga)|$
to $\wt\cZ_{m_v,\|\b_v\|}(\al_{j_v},Q)$, where
$$m_v\equiv m_v(\bar\Ga)=\big|\bar\eta^{-1}(v)\big|+\big|\tnE_v(\bar\Ga)\big|-3, \qquad
\b_v=\b|_{\bar\eta^{-1}(v)\cup\tnE_v(\bar\Ga)}\,, \qquad j_v=\bar\mu(v).$$
For each $s\!\in\!\bar\eta^{-1}(v)$, \e_ref{Zreg_e8b} and~\e_ref{cZ2deg0_e} 
with $v\!=\!\bar\eta(s)$ and $S_e\!=\!\{s\}$ still compute the sum of the factors on
the third line in~\e_ref{decomp2_e2} over all possibilities for $\Ga_{s'}$ of positive 
and zero degree, respectively.
By a similar reasoning (see Footnote~\ref{ressum_ftnt}), for each $e\!\in\!\ov\Edg$
\begin{equation*}\begin{split}
&\sum_{\Ga_e}\Bigg( 
\bigg(\frac{\al_{\mu_{v_e^-}(e^-)}\!-\!\al_{j_{v_e^-}}}{\d(e^-)}\bigg)^{\!\!-b_e^--1}
\bigg(\frac{\al_{\mu_{v_e^+}(e^+)}\!-\!\al_{j_{v_e^+}}}{\d(e^+)}\bigg)^{\!\!-b_e^+-1}
\!\!\!\int_{\cZ_{\Ga_e}}\!\!\!
\frac{\E(\V)\ev_{e^-}^*\phi_{j_{v_e^-}}\ev_{e^+}^*\phi_{j_{v_e^+}}}
{\E(\N\cZ_{\Ga_e})}
\!\!\Bigg)\\
&\hspace{2.2in}
=\Res{\hb_-=0}\Big\{\Res{\hb_+=0}\Big\{\hb_-^{-b_e^--1}\hb_+^{-b_e^+-1}
\cZ^*\big(\hb_-,\hb_+,\al_{j_{v_e^-}},\al_{j_{v_e^+}},Q\big)\Big\}\Big\},
\end{split}\end{equation*}
where the sum is taken over all possibilities for the strand $\Ga_e$ between the vertices
$v_{e^-}$ and $v_{e^+}$ in~$\Ga$ with $\mu(v_e^-)\!=\!j_{v_e^-}$ and $\mu(v_e^+)\!=\!j_{v_e^+}$
fixed.
Since 
$$\Res{\hb_-=0}\bigg\{\Res{\hb_+=0}\bigg\{\hb_-^{-b_e^--1}\hb_+^{-b_e^+-1}
\frac{\lr\a \al_{j_e^-}^l}{\hb_-\!+\!\hb_+}\!\!
\sum_{\begin{subarray}{c}p_-+p_++r=n-1\\ p_-,p_+,r\ge0 \end{subarray}}
\!\!\!\!\!\!\!\!\!\hat\si_r\al_{j_{v_e^-}}^{p_-}\al_{j_{v_e^+}}^{p_+}\bigg\}\bigg\}=0
\qquad\forall\,b_e^-,b_e^+\in\bar\Z^+,$$
we can replace $\cZ^*$ in the previous expression by $\cZ$.\\ 

\noindent
Putting this all together, we obtain a replacement for \e_ref{cZsum_e1}, 
involving products over $v\!\in\!\ov\Ver$ and $e\!\in\!\ov\Edg$,
which \e_ref{cZmB_e} and the first statement of Lemma~\ref{comb_l0}, 
reduce~to
\begin{equation*}\begin{split}
&\left(\prod_{v\in\ov\Ver}
\frac{\prod\limits_{k\neq j_v}\!\!\!(\al_{j_v}\!-\!\al_k)}{\lr\a\al_{j_v}^l}\right)
\sum_{\Ga}Q^{|\Ga|}\!\!\!\int_{\cZ_{\Ga}}\!\!
\frac{\E(\V)}{\E(\N\cZ_{\Ga})}
\prod_{s=1}^{s=N}\!\!\bigg(\frac{\ev_s^*\phi_{i_s}}{\hb_s\!-\!\psi_s}\bigg)\\
&=\sum_{\begin{subarray}{c} \b''\in(\bar{\Z}^+)^N\\ \b^-,\b^+\in(\bar{\Z}^+)^{\ov\Edg}\\ 
(\c_v)_{v\in\ov\Ver}\in((\bar\Z^+)^{\i})^{\ov\Ver}\\ 
|\b''|_{\eta^{-1}(v)}+|\b^-|_{\tnE_{\bar\Ga}^-(v)}+b_{e_v}^++\|\c_v\|=m_v\end{subarray}}
\hspace{-.75in}
\left\{\! \prod_{v\in\ov\Ver}\!\!\!\!\Psi_{m_v,\c_v}(\al_{j_v},Q)\times
\prod_{s=1}^{s=N}\!\!\bigg(\!
\frac{1}{\lr\a\al_{j_s}^l}\frac{1}{b_s''!}
\Res{\hb=0}\bigg\{\frac{e^{-\frac{\ze(\al_{j_s},Q)}{\hb}}}{\hb^{\,b_s''+1}}
\!\cZ\big(\hb_s,\hb,\al_{i_s},\al_{j_s},Q\big)\!\bigg\}
\!\!\bigg)\!\!\right.\\
&\hspace{.3in}
\times\left.\prod_{e\in\ov\Edg}\!\!\Bigg(\!\frac{1}{\lr\a^2\al_{j_{v_e^-}}^l\al_{j_{v_e^+}}^l}
\frac{1}{b_e^-!b_e^+!}
\Res{\hb_-=0}\bigg\{\Res{\hb_+=0}\bigg\{
\frac{e^{-\frac{\ze(\al_{j_{v_e^-}},Q)}{\hb_-}-\frac{\ze(\al_{j_{v_e^+}},Q)}{\hb_+}}}
{\hb_-^{b_e^-+1}\hb_+^{b_e^++1}}
\!\cZ\big(\hb_-,\hb_+,\al_{j_{v_e^-}},\al_{j_{v_e^+}},Q\big)\!\bigg\}\!\!\bigg\}
\!\!\Bigg)\!\!\right\}\,,
\end{split}\end{equation*}
where $b_{e_{v_0}}\!\equiv\!0$ for the minimal element $v_0\!\in\!\ov\Ver$,
$j_s\!=\!j_{\bar\mu(\bar\eta(s))}$, and
the sum is taken over all possibilities for~$\Ga$ with the core 
$\bar\Ga\!=\!(\ov\Ver,\ov\Edg;\bar\mu,\bar\eta)$ fixed.
Using Lemma~\ref{pt2_lmm} and~\e_ref{cZpexp_e} to compute the residues,
we find that the sum on the left-hand side of the above expression equals
\begin{equation*}\begin{split}
&\lr\a\!\!\!\!\!\sum_{\begin{subarray}{c}\p\in\nset^N\\ \b\in(\bar\Z^+)^N\end{subarray}}
\!\!\Bigg\{\un\hb^{-\b}\!\cZ_{\p}(\un\hb,\al_{i_1\ldots i_N},Q) 
\!\!\!\sum_{\begin{subarray}{c}\ti\p\in\nset^N\\ \p',\ti\p'\in\nset^{\ov\Edg}\\ 
\b'\in(\bar\Z^+)^{\ov\Edg}\end{subarray}} \hspace{-.23in}
(-1)^{|\b|+|\b'|} \hspace{-.5in}
\sum_{\begin{subarray}{c} \b''\in(\bar{\Z}^+)^N\\ \b^-,\b^+\in(\bar{\Z}^+)^{\ov\Edg}\\ 
(\c_v)_{v\in\ov\Ver}\in((\bar\Z^+)^{\i})^{\ov\Ver}\\ 
|\b''|_{\eta^{-1}(v)}+|\b^-|_{\tnE_{\bar\Ga}^-(v)}+b_{e_v}^++\|\c_v\|=m_v\end{subarray}}
\hspace{-.75in}
\prod_{v\in\ov\Ver} \frac{\Psi_{m_v,\c_v}(\al_{j_v},Q)}
{\al_{j_v}^{l(m_v+2)}\!\!\!
\prod\limits_{k\neq j_v}\!\!(\al_{j_v}\!-\!\al_k)}\\
&\hspace{.2in}\times
\prod_{s=1}^{s=N}\frac{\cC_{p_s\ti{p}_s}(Q)\Psi_{\ti{p}_s;b_s''-b_s}(\al_{j_s},Q)}{b_s''!}
\times\!\left.\prod_{e\in\ov\Edg}
\!\!\!\frac{\cC_{p_e'\ti{p}_e'}(Q)\Psi_{p_e';b_e^+-b_e'}(\al_{j_{v_e^+}},Q)
\Psi_{\ti{p}_e';b_e^-+1+b_e'}(\al_{j_{v_e^-}},Q)}{b_e^-!b_e^+!}\right\}.
\end{split}\end{equation*}
For each $v\!\in\!\ov\Ver$, we now sum up the product of the corresponding factors above
over all possibilities for $j_v\!\in\![n]$
(which also determines $j_s$ and  $j_{v_e^{\pm}}$ whenever $\eta(s)\!=\!v$ and 
$v_e^{\pm}\!=\!v$).
Using the Residue Theorem on $S^2$, we now obtain an explicit formula
for the coefficients~$\cC_{\p,\b}^{(d)}$ in Theorem~\ref{equiv_thm}:
\BE{cCformula_e}\begin{split}
&\cC_{\p,\b}^{(d)}=
\sum_{\Ga}\sum_{\bfd\in\cP_{\Ga}(d)}\!\!
\sum_{\begin{subarray}{c}\ti\p\in\nset^N\\ \p',\ti\p'\in\nset^{\Edg}\\ 
\b'\in(\bar\Z^+)^{\Edg}\end{subarray}} \hspace{-.23in}
(-1)^{|\b|+|\b'|} \hspace{-.5in}
\sum_{\begin{subarray}{c} \b''\in(\bar{\Z}^+)^N\\ \b^-,\b^+\in(\bar{\Z}^+)^{\Edg}\\ 
(\c_v)_{v\in\Ver}\in((\bar\Z^+)^{\i})^{\Ver}\\ 
|\b''|_{\eta^{-1}(v)}+|\b^-|_{\tnE_{\Ga}^-(v)}+b_{e_v}^++\|\c_v\|=m_v\end{subarray}}
\hspace{-.6in}
\left.\prod_{v\in\Ver} (-1)\!\!\!\Res{\x=0,\i}\! 
\right\llbracket\frac{\Psi_{m_v,\c_v}(\x,Q)}
{\x^{l(m_v+2)}\!\!\!\prod\limits_{k=1}^{k=n}\!\!(\x\!-\!\al_k)}\\
&\hspace{0.1in}\times\!\!\!\!\!
\prod_{s\in\eta^{-1}(v)}\!\!\!\!\!\!
\frac{\cC_{p_s\ti{p}_s}(Q)\Psi_{\ti{p}_s;b_s''-b_s}(\x,Q)}{b_s''!}
\times\!\!\!\!\!\!\left.\prod_{e\in\tnE_{\Ga}^-(v)}
\!\!\!\!\!\!\frac{\cC_{p_e'\ti{p}_e'}(Q)\Psi_{\ti{p}_e';b_e^-+1+b_e'}(\x,Q)}{b_e^-!}
\times\frac{\Psi_{p_{e_v}';b_{e_v}^+-b_{e_v}'}(\x,Q)}{b_{e_v}^+!}
\right\rrbracket_{q;d_v}
\end{split}\EE
where the outer sum is taken over all $N$-marked trivalent trees $\Ga\!\equiv\!(\Ver,\Edg;\eta)$ and 
$$\frac{\Psi_{p_{e_{v_0}}';b_{e_{v_0}}^+-b_{e_{v_0}}'}(\x,Q)}{b_{e_{v_0}}^+!}\equiv1$$
for the minimal element $v_0\!\in\!\Ver$.
Using \e_ref{PsimcPhi_e}, \e_ref{p2cC_e}, the second statement in~\e_ref{PsiPhi_e},   
and the last two statements in Lemma~\ref{ressum_lmm} as before, we conclude that  
\begin{equation*}\begin{split}
\cC_{\p,\b}^{(d)}&\sim
\sum_{\Ga}\!\!\!
\sum_{\begin{subarray}{c}\bfd\in\cP_{\Ga}(d)\\ \p'\in\nset^{\Edg},\, 
\b'\in(\bar\Z^+)^{\Edg}\end{subarray}} \hspace{-.3in}
(-1)^{|\b|+|\b'|}\hat\si_n^{t_{\p}+t_{\p'}+|\bft|} \hspace{-.5in}
\sum_{\begin{subarray}{c} \b''\in(\bar{\Z}^+)^N,\, \b^-,\b^+\in(\bar{\Z}^+)^{\Edg}\\ 
(\c_v)_{v\in\Ver}\in((\bar\Z^+)^{\i})^{\Ver}\\ 
|\b''|_{\eta^{-1}(v)}+|\b^-|_{\tnE_{\Ga}^-(v)}+b_{e_v}^++\|\c_v\|=m_v\end{subarray}}
\hspace{-.25in}
\left.\prod_{v\in\Ver}\! 
\right\llbracket \!\Phi_{m_v,\c_v}(q)\\
&\hspace{0in}\times\!\!\!\!\!
\prod_{s\in\eta^{-1}(v)}\!\!\!\!\!\! \frac{L(q)^{\de_{0\nua}nt_{p_s}}\Phi_{\hat{p}_s;b_s''-b_s}(q)}{b_s''!\,\Phi_0(q)}
\times\!\!\!\left.\prod_{e\in\tnE_{\Ga}^-(v)}
\!\!\!\!\!\!\frac{L(q)^{\de_{0\nua}nt_{p_e'}}\Phi_{\hat{p}_e';b_e^-+1+b_e'}(q)}{b_e^-!\,\Phi_0(q)}
\times\frac{I_0(q)^2\Phi_{p_{e_v}';b_{e_v}^+-b_{e_v}'}(q)}{b_{e_v}^+!\,L(q)^{\de_{0\nua}n}\Phi_0(q)}
\right\rrbracket_{q;d_v}
\end{split}\end{equation*}
with the last fraction above set to 1 for $v\!=\!v_0$ and
$\bft\!\in\!(\bar\Z^+)^{\Ver}$ defined by \e_ref{sumcond_e};
if an integer $t_v$ satisfying~\e_ref{sumcond_e} does not exist for some $v\!\in\!\Ver$, 
the corresponding summand above is defined to be~0.
This confirms~\e_ref{equivthm_e2} with $\nc_{\p,\b}^{(d,0)}$ as defined in Section~\ref{graph_subs}
(and describes $\nc_{\p,\b}^{(d,t)}$ with $t\!\in\!\Z^+$ as~well).

\begin{rmk}\label{equivcoeff_rmk}
The recursion~\e_ref{cCrec_e} and separately the closed formula~\e_ref{cCformula_e}
compute the coefficients~$\cC_{\p,\b}^{(d)}$ in~\e_ref{equivthm_e} 
and thus provide a straightforward algorithm for computing the equivariant $N$-pointed generating 
function~\e_ref{cZdfn_e}.
Following the proof of the first statement in Lemma~\ref{ressum_lmm}, 
the power series~$\Psi_{m,\c}(\x,Q)$ and~$\Psi_{p;b}(\x,Q)$ can be computed directly from 
the power series~$\Phi_b(\x,q)$ appearing in~\e_ref{cYexp_e}.
The latter can be computed similarly to the power series~$\Phi_b(q)$ 
appearing in Proposition~\ref{Fexp_prp};
see Appendix~\ref{HGprp_pf}.
For example, we first find that the power series~$\xi$ appearing in~\e_ref{cYexp_e} is described~by
$$\xi\in \x q\cdot\Q[\al,\x,\si_{n-1}(\x)^{-1}]\big[\big[q\big]\big], \qquad
\x+\xi'(\x,q)=L(\x,q)\,,$$
where $'$ denotes $q\frac{\tnd}{\tnd q}$ as before $L(\x,q)$ is defined by 
$$L(\x,q)\in \x+\x^{|\a|}q\cdot\Q[\al,\x,\si_{n-1}(\x)^{-1}]\big[\big[\x^{|\a|-1}q\big]\big], 
\quad
\si_n\big(L(\x,q)\big)-q\,\a^{\a}L(\x,q)^{|\a|}=\si_n(\x),$$
with $\si_n(\cdot)$ defined analogously to~\e_ref{sinmn1dfn_e};
setting $\al\!=\!0$ and $\x\!=\!1$ above gives \e_ref{Ldfn_e0}.
We then find~that 
$$\Phi_0(\x,q)=\left(\frac{\x\cdot\si_{n-1}(\x)}
{L(\x,q)\,\si_{n-1}(L(\x,q))-|\a|(\si_n(L(\x,q))-\si_n(\x))}\right)^{1/2}
\bigg(\frac{L(\x,q)}{\x}\bigg)^{(l+1)/2}\,;$$
setting $\al\!=\!0$ and $\x\!=\!1$ above gives \e_ref{F0exp_e2}.
This suffices for the $N\!=\!3$ case of~\e_ref{equivthm_e}.
\end{rmk}

\section{Proof of Theorem~\ref{GWbound_thm}}
\label{GWbound_app}

\noindent
In this section we prove the bound of Theorem~\ref{GWbound_thm} for $d\!\in\!\Z^+$
by considering four separate cases:
$|\a|\!>\!n$ and $|\a|\!\le\!n$ with $N\!=\!1,2,3+$.
The first case is fairly straightforward, since there are only finitely many nonzero 
GW-invariants modulo the string, dilaton, and divisor relations \cite[p527]{MirSym}.
In the $|\a|\!\le\!n$ cases, we use explicit mirror formulas.
For $N\!=\!1,2$, \e_ref{Z1pt_e} and \e_ref{Z2pt_e} reduce Theorem~\ref{GWbound_thm}
to extracting the coefficients of~$w^bq^d$ from the power series~$F(w,q)$
and~$F_p(w,q)$ defined in~\e_ref{Fdfn_e} and~\e_ref{Flpdfn_e};
Corollary~\ref{Fpest_crl} below presents them in a convenient form.
For $N\!\ge\!3$, the coefficients~$\c_{\p,\b}^{(d,0)}$ in Theorem~\ref{main_thm}
must also be suitable bounded.
This is done by Proposition~\ref{coeffbound_prp}; its proof constitutes most of this section.\\

\noindent
We begin by considering the $|\a|\!>\!n$ case. 
Let
$$d_{\max}=\big\{d\!\in\!\Z\!:\,(|\a|\!-\!n)d\le n\!-\!4\!-\!l\big\}.$$
If $d\!>\!d_{\max}$, the virtual dimension of $\ov\M_{0,0}(X_{\a},d)$ is negative,
and so all genus~~0 degree~$d$ GW-invariants vanish.
Thus, we can assume that $d_{\max}\!\in\!\Z^+$.
Let $C\!\in\!\R^+$ be such that 
$$\big|\blr{b_1!\,\tau_{b_1}H^{c_1},\ldots,b_N!\,\tau_{b_N}H^{c_N}}_{0,d}^{X_{\a}}\big|
\le C$$
whenever $b_s\!+\!c_s\!\ge\!2$ for all $s$ or $N\!\le\!d_{\max}$;
the number of nonzero invariants of this form is finite.
Let $b_{\max}$ be the largest of the sums $b_1\!+\!\ldots\!+\!b_N$ for nonzero 
invariants of this form.
It then follows by induction via the dilaton, string, and divisor relations that 
\begin{equation*}\begin{split}
&\big|\blr{b_1!\,\tau_{b_1}H^{c_1},\ldots,b_N!\,\tau_{b_N}H^{c_N},
\underset{k_1}{\underbrace{\tau_0H^1,\ldots,\tau_0H^1}},
\underset{k_2}{\underbrace{\tau_0H^0,\ldots,\tau_0H^0}},
\underset{k_3}{\underbrace{\tau_1H^0,\ldots,\tau_1H^0}}}_{0,d}^{X_{\a}}\big|\\
&\hspace{2.2in}
\le C\big(b_{\max}\!+\!d_{\max}\big)^{k_1}
\cdot\frac{(b_{\max}\!+\!k_2)!}{b_{\max}!}
\cdot\frac{(N\!+\!k_1\!+\!k_2\!+\!k_3)!}{(N\!+\!k_1\!+\!k_2)!}\\
&\hspace{2.2in}
\le C\cdot C'^{k_1}\cdot2^{b_{\max}+k_2}\cdot(N\!+\!k_1\!+\!k_2\!+\!k_3)!\,.
\end{split}\end{equation*}
This implies the bound in Theorem~\ref{GWbound_thm}.\\

\noindent
In the remainder of this section, we treat the $|\a|\!<\!n$ cases.

\subsection{Outline of proof}
\label{GWboundprelim_subs}

\noindent
By \e_ref{genus0_e} and \e_ref{Zdfn_e}, 
the GW-invariant in Theorem~\ref{GWbound_thm} is the coefficient of
$$Q^d\bfH^{\p}\un\hb^{-\b-\1}\equiv
Q^d\prod_{s=1}^{s=N}H_s^{p_s}\hb_s^{-b_s-1}\,,
\qquad\hbox{where}\quad p_s=n\!-\!1\!-\!c_s\,,$$
of the right-hand side of the identity in \e_ref{Z1pt_e} if $N\!=\!1$,
in \e_ref{Z2pt_e} if $N\!=\!2$, and in~\e_ref{mainthm_e} if $N\!\ge\!3$.
In particular, we need to bound the growth of the coefficients~of 
$$e^{-J(q)H/\hb}H^p\frac{F_p(H/\hb,q)}{I_{p-l}(q)}
\in \Q[H][[\hb^{-1},Q]],
\quad\hbox{where}\quad qe^{\de_{0\nua}J(q)}\!=\!Q/H^{\nua}\,.$$
By \e_ref{Fdfn_e}, \e_ref{bDdfn_e}-\e_ref{Flpdfn_e},
for every $p\!\in\!\Z^+$ there exists $\hat{F}_p\!\in\!\Q(w)[[q]]$  such~that 
\BE{hatFdfn_e} e^{-J(q)H/\hb}
H^p\frac{F_p(H/\hb,q)}{I_{p-l}(q)} 
=\hb^p\hat{F}_p\big(H/\hb,Q/\hb^{\nua})\,,\EE
and the coefficient of each power of~$q$ is holomorphic at $w\!=\!0$.\\

\noindent
If $b_1\!+\!c_1\!=\!\nua d\!+\!n\!-\!3\!-\!l$, 
\e_ref{genus0_e}, \e_ref{Zdfn_e}, \e_ref{Z1pt_e}, and~\e_ref{hatFdfn_e} give 
$$\blr{\tau_{b_1}H^{c_1}}_{0,d}^{X_{\a}}
=\LRbr{\LRbr{\LRbr{Z(\hb_1,H_1,Q)}_{Q;d}}_{\hb_1^{-1};b_1+1}}_{H_1;p_1}
=\lr\a \LRbr{\LRbr{\hat{F}_l(w,q)}_{q;d}}_{w;p_1}\,,$$
where $p_1\!=\!n\!-\!1\!-\!c_1$ as before.
Thus, by Corollary~\ref{Fpest_crl} below,
\begin{equation*}\begin{split}
\big|\blr{b_1!\,\tau_{b_1}H^{c_1}}_{0,d}^{X_{\a}}\big|
\le \lr\a C_{\a}^d\frac{b_1!}{(\nua d)!}
&\le \lr\a C_{\a}^d(n\!-\!3\!-\!l)!\binom{\nua d\!+\!n\!-\!3\!-\!l}{\nua d} \\
&\le (n\!-\!3\!-\!l)!\lr\a C_{\a}^d\cdot2^{\nua d+n-3-l}\,;
\end{split}\end{equation*}
this confirms the statement of Theorem~\ref{GWbound_thm} for $N\!=\!1$.\\

\noindent
If $b_1\!+\!c_1\!+\!b_2\!+\!c_2=\!\nua d\!+\!n\!-\!3\!-\!l$, 
\e_ref{genus0_e}, \e_ref{Zdfn_e}, \e_ref{Z2pt_e}, and~\e_ref{hatFdfn_e}  give 
\begin{equation*}\begin{split}
\sum_{\begin{subarray}{c}\de_1+\de_2=1\\ \de_1,\de_2\ge0\end{subarray}}\!\!\!\!
\blr{\tau_{b_1+\de_1}H^{c_1},\tau_{b_2+\de_2}H^{c_2}}_{0,d}^{X_{\a}}
&=\sum_{\begin{subarray}{c}\de_1+\de_2=1\\ \de_1,\de_2\ge0\end{subarray}}\!\!
\LRbr{\LRbr{\LRbr{Z(\un\hb,\bfH,Q)}_{Q;d}}_{\un\hb^{-1};(b_1+1+\de_1,b_2+1+\de_2)}}_{\bfH;(p_1,p_2)}\\
&=\lr\a 
\sum_{\begin{subarray}{c}d_1+d_2=d\\ d_1,d_2\ge 0\\
\nua d_s\ge l+1+b_s-p_s\end{subarray}}\!\!\!\!\!\!
\prod_{s=1}^{s=2}\LRbr{\LRbr{\hat{F}_{\nua d_s+p_s-b_s-1}(w,q)}_{q;d_s}}_{w;p_s}\,,
\end{split}\end{equation*}
with $p_s\!=\!n\!-\!1\!-\!c_s$, $\un\hb\!=\!(\hb_1,\hb_2)$, and $\bfH\!=\!(H_1,H_2)$.
This gives
$$\blr{\tau_{b_1+1}H^{c_1},\tau_{b_2}H^{c_2}}_{0,d}^{X_{\a}}
=\lr\a\hspace{-.1in}
\sum_{\begin{subarray}{c}b_1'+b_2'=b_1+b_2+2\\ 0\le b_2'\le b_2\end{subarray}}
\hspace{-.2in}
(-1)^{b_2-b_2'}
\hspace{-.3in}
\sum_{\begin{subarray}{c}d_1+d_2=d\\ d_1,d_2\ge 0\\
\nua d_s\ge l+b_s'-p_s\end{subarray}}\!\!\!\!\!
\prod_{s=1}^{s=2}\LRbr{\LRbr{\hat{F}_{\nua d_s+p_s-b_s'}(w,q)}_{q;d_s}}_{w;p_s}\,.$$
Thus, by Corollary~\ref{Fpest_crl} below,
\begin{equation*}\begin{split}
\big|\blr{(b_1\!+\!1)!\,\tau_{b_1+1}H^{c_1},b_2!\,\tau_{b_2}H^{c_2}}_{0,d}^{X_{\a}}\big|
&\le \lr\a(b_2\!+\!1)\, C_{\a}^d\frac{(b_1\!+\!1)!b_2!}{(\nua d)!}
\sum_{d_1=0}^{d_1=d}\binom{\nua d}{\nua d_1}\\
&\le \lr\a C_{\a}^d\cdot (n\!-\!1\!-\!l)!\binom{\nua d\!+\!n\!-\!1\!-\!l}{\nua d}
\cdot 2^{\nua d} \\
&\le (n\!-\!1\!-\!l)!\lr\a C_{\a}^d\cdot2^{2\nua d+n-1-l}\,;
\end{split}\end{equation*}
this confirms the statement of Theorem~\ref{GWbound_thm} for $N\!=\!2$.\\

\noindent
Finally, we consider the $N\!\ge\!3$ case. 
For each $p\!\in\!\llfloor{n}\rrfloor_l$, let
\BE{hatFdfn_e2}\hat{F}_{(p)}\big(w,q)
=\frac{\hat{F}_p\big(w,q)}{\prod\limits_{r=p-l+1}^{n-l-1}\!\!\!\!I_r(q)}\,.\EE
It is sufficient to assume that the tuples $\b\!\equiv\!(b_s)_{s\in[N]}$ and
$\c\!\equiv\!(c_s)_{s\in[N]}$ in the statement of Theorem~\ref{GWbound_thm} satisfy
$$|\b|+|\c|=\nua d+n-4-l+N, \qquad b_s,c_s\ge0, \qquad c_s\le n\!-\!1\!-\!l.$$
Let $p_s\!=\!n\!-\!1\!-\!c_s$.
If $\bfd,\b'\!\in\!(\bar\Z^+)^N$, define
$$\p'(\bfd,\b')\in(\bar\Z^+)^N \qquad\hbox{by}\qquad
p_s'(\bfd,\b')=\nua d_s+p_s-b_s+b_s'\,.$$ 
By~\e_ref{genus0_e}, \e_ref{Zdfn_e}, \e_ref{mainthm_e}, \e_ref{Depdfn_e},
 and~\e_ref{hatFdfn_e},
\BE{GWcoeff_e}\blr{\tau_{b_1}H^{c_1},\ldots,\tau_{b_N}H^{c_N}}_{0,d}^{X_{\a}}
=\lr\a\!\!\!\!
\sum_{\begin{subarray}{c} 0\le d'\le d\\ \bfd\in\cP_N(d-d')\\ \b'\in(\bar\Z^+)^N\end{subarray}}
\!\!\!\!\!\!
\nc_{\p'(\bfd,\b'),\b'}^{(d',0)}
\prod_{s=1}^{s=N}
\LRbr{\LRbr{\hat{F}_{(p_s'(\bfd,\b'))}(w,q)}_{q;d_s}}_{w;p_s};\EE
the above summand vanishes unless $l\!\le\!p_s'(\bfd,\b')\!\le\!n\!-\!1$ for all $s\!\in\![N]$.
Since 
$\nc_{\p',\b'}^{(d',0)}\!=\!0$ unless $|\b'|\!\le\!N\!-\!3$,
Corollary~\ref{Fpest_crl} and Proposition~\ref{coeffbound_prp} thus give 
\begin{equation*}\begin{split}
&\big|\blr{b_1!\,\tau_{b_1}H^{c_1},\ldots,b_N!\,\tau_{b_N}H^{c_N}}_{0,d}^{X_{\a}}
\le\lr\a N!C_{\a}^{N+d}
\sum_{\begin{subarray}{c} 0\le d'\le d\\ \bfd\in\cP_N(d-d')\end{subarray}}
\!\sum_{\begin{subarray}{c} \b'\in(\bar\Z^+)^N\\ |\b'|\le N-3\\
b_s'\ge b_s-\nua d_s-p_s\end{subarray}}\hspace{-.1in}
\prod_{s=1}^{s=N}\!\!\bigg(p_s!\frac{b_s!}{b_s'!(\nua d_s)!p_s!}\bigg)\\
&\hspace{.7in}\le \lr\a N!C_{\a}^{N+d}
\sum_{\begin{subarray}{c} 0\le d'\le d\\ \bfd\in\cP_N(d-d')\end{subarray}}
\sum_{\begin{subarray}{c} \b'\in(\bar\Z^+)^N\\ |\b'|\le N-3\\ b_s'\ge b_s-\nua d_s-p_s\end{subarray}}
\hspace{-.1in}
\prod_{s=1}^{s=N}\big(n!3^{b_s}\big)\\
&\hspace{.7in}\le \lr\a N!C_{\a}^{N+d}\cdot (n!)^N3^{\nua d+n+N}\cdot
\binom{d\!+\!N}{N}\binom{N\!-\!3\!+\!N}{N}
\le N!C_{\a}'^{N+d}\cdot 2^{d+N}\cdot 2^{2N-3}\,.
\end{split}\end{equation*}
This confirms the statement of Theorem~\ref{GWbound_thm} for $N\!\ge\!3$.

\begin{rmk}\label{GWeq0_rmk}
For any non-vanishing summand on the right-hand side of \e_ref{GWcoeff_e}, $p_s'(\bfd,\b')\!\le\!n\!-\!1$
and so $b_s\!+\!c_s\!\ge\!\nua d_s$.
Thus, $d_s\!=\!0$ if $b_s\!+\!c_s\!<\!\nua$. 
Since the coefficient of~$q^0$ in $\hat{F}_{(p)}(w,q)$ is~$w^p$, it follows that 
$p_s'(\bfd,\b')\!=\!p_s$ and $b_s'\!=\!b_s$ in such a case.
Since $|\b'|\!\le\!N\!-\!3$, this implies Theorem~\ref{GWeq0_thm}.
\end{rmk}

\subsection{Bounds on the coefficients of generating functions}
\label{Fpest_subs}

\noindent
In this section, we obtain the bounds on the coefficients of 
the power series $F_p,\hat{F}_p\in\Q(w)[[q]]$ defined in~\e_ref{hatFdfn_e}
and~\e_ref{hatFdfn_e2} that are used in the proof of Theorem~\ref{GWbound_thm} 
above.

\begin{lmm}\label{Fpest_lmm}
There exists $C_{\a}\!\in\!\R^+$ such that 
$$\left|\LRbr{\LRbr{F_{p'}(w,q)}_{q;d}}_{w;\nua d-p'+p}\right|
\le \frac{C_{\a}^d}{(\nua d)!}$$
for all $p,p'\!=\!0,1,\ldots,n\!-\!1$ and $d\!\in\!\bar\Z^+$.
\end{lmm}

\begin{proof} By \e_ref{Fdfn_e}, \e_ref{bDdfn_e}-\e_ref{Fpdfn_e0}, and~\e_ref{Flpdfn_e}, 
it is sufficient to show that there exists $C\!\in\!\R^+$ such that 
$$\left|\LRbr{\LRbr{F_0(w,q)}_{q;d}}_{w;\nua d+p}\right|,
\left|\LRbr{\LRbr{F(w,q)}_{q;d}}_{w;\nua d-l+p}\right|
\le \frac{C^d}{(\nua d)!}$$
for all $p\!=\!0,1,\ldots,n\!-\!1$ and $d\!\in\!\bar\Z^+$.
Both numbers on the left-hand side vanish for $p\!<\!l$
(unless $d,p\!=\!0$ in the case of the first number).
If $l\!\le\!p\!<\!n$,
\begin{equation*}\begin{split}
&\left|\LRbr{\LRbr{F(w,q)}_{q;d}}_{w;\nua d-l+p}\right|
=\frac{\prod\limits_{k=1}^l(a_kd)!}{(d!)^n}
\left|\LRbr{\frac{\prod\limits_{k=1}^l\prod\limits_{r=1}^{a_kd}(1\!+\!(a_k/r)w)}
{\prod\limits_{r=1}^d(1\!+\!w/r)^n}}_{w;p-l}\right|\\
&\hspace{.3in} \le n^{nd}\frac{(|\a|d)!}{(nd)!}\cdot
\LRbr{\frac{(1\!+\!|\a|w)^{(|\a|-l)d}}{(1\!-\!w)^{(n-l)d}}}_{w;p-l}\\
&\hspace{.3in}
\le \frac{n^{nd}}{(\nua d)!} 
\sum_{\begin{subarray}{c}r+s=p-l\\ r,s\ge0\end{subarray}}\!\!\!
\binom{(n\!-\!l)d\!+\!r\!-\!1}{r}\binom{(|\a|\!-\!l)d}{s}|\a|^s
\le \frac{n^{nd}}{(\nua d)!}2^{(n-l)d+p-l}(|\a|\!+\!1)^{(|\a|-l)d}\,.
\end{split}\end{equation*}
The first inequality above follows from Stirling's formula  \cite[Section~15.22]{A}, 
\BE{Stirling_e} 1<\frac{e^d}{\sqrt{2\pi}d^{d+\frac{1}{2}}}d!<e^{\frac{1}{8d}}~~~
\forall \,d\!\in\!\Z^+;\EE
the following statement uses the Binomial Theorem.  
The desired bound for $F_0(w,q)$ is obtained similarly.
\end{proof}


\begin{crl}\label{Fpest_crl}
There exists $C_{\a}\!\in\!\R^+$ such that 
$$\left|\LRbr{\LRbr{\hat{F}_{p'}(w,q)}_{q;d}}_{w;p}\right|,
\left|\LRbr{\LRbr{\hat{F}_{(p')}(w,q)}_{q;d}}_{w;p}\right|
\le \frac{C_{\a}^d}{(\nua d)!}$$
for all $p,p'\!=\!0,1,\ldots,n\!-\!1$ and $d\!\in\!\bar\Z^+$.
\end{crl}

\begin{proof} If $\nua\!\ge\!2$,
$$\LRbr{\LRbr{\hat{F}_{p'}(w,q)}_{q;d}}_{w;p}
=\LRbr{\LRbr{F_{p'}(w,q)}_{q;d}}_{w;\nua d-p'+p}\,,$$
and the claim follows immediately from Lemma~\ref{Fpest_lmm}.
If $\nua\!=\!1$, by \e_ref{Icdfn_e}
\begin{gather*}
\LRbr{\LRbr{\hat{F}_{p'}(w,q)}_{q;d}}_{w;p}
=\sum_{\begin{subarray}{c}d_1+d_2=d\\ d_1,d_2\ge0\end{subarray}}
\frac{(-\a!)^{d_1}}{d_1!}\LRbr{\LRbr{F_{p'}(w,q)}_{q;d_2}}_{w;d_2-p'+p}\,\\
\Lra\qquad 
\left|\LRbr{\LRbr{\hat{F}_{p'}(w,q)}_{q;d}}_{w;p}\right|
\le\frac{(\a!\!+\!C_{\a})^d}{d!}\,,
\end{gather*}
where $C_{\a}$ is as in Lemma~\ref{Fpest_lmm}.
Finally, suppose $\nua\!=\!0$.
Define
$$\ti{J}\in Q\cdot\Q[[Q]]\qquad\hbox{by}\qquad q=Qe^{\ti{J}(Q)}\,.$$
By Lemma~\ref{Fpest_lmm},
$$\left|\LRbr{I_0(q)}_{q;d}\right|,
\left|\LRbr{I_1(q)}_{q;d}\right|,\ldots,
\left|\LRbr{I_{n-1}(q)}_{q;d}\right|,
\left|\LRbr{J(q)}_{q;d}\right|\le C^d
\quad\Lra\quad
\left|\LRbr{\ti{J}(q)}_{q;d}\right|\le C'^d\,;$$
the last implication follows from the Inverse Function Theorem.
Since
$$\LRbr{\hat{F}_{p'}(w,Q)}_{w;p}
=\sum_{\begin{subarray}{c}p_1+p_2=p-p'\\ p_1,p_2\ge0\end{subarray}}
\frac{\ti{J}(Q)^{p_1}}{p_1!}
\frac{\LRbr{F_{p'}(w,q)}_{w;p_2}}{\prod\limits_{r=p-l}^{n-l-1}\!\!\!\!I_r(q)}\,,$$
the claim again follows from Lemma~\ref{Fpest_lmm}.
\end{proof}

\subsection{Bounds on the structure constants in Theorem~\ref{main_thm}}
\label{coeffbound_subs}

\noindent
In this section, we obtain an upper bound for the coefficients $\nc_{(\p,\b)}^{(d,0)}$
in Theorem~\ref{main_thm}.
This is one the two key ingredients in the proof of Theorem~\ref{GWbound_thm}.

\begin{prp}\label{coeffbound_prp}
If $n,N\!\in\!\Z^+$ with $N\!\ge\!3$ and $\a\!\in\!(\Z^+)^l$
with $|\a|\!\le\!n$, there exists $C_{\a}\!\in\!\R^+$ such~that 
$$\big| \nc_{\p,\b}^{(d,0)}\big|\le \frac{N!}{\b!}C_{\a}^{N+d}
\qquad\forall\,d\!\in\!\bar\Z^+,\,\p\!\in\!\llfloor{n}\rrfloor^N,
\,\b\!\in\!(\bar\Z^+)^N.$$
\end{prp}

\begin{lmm}\label{Lest_lmm}
If $n\!\in\!\Z^+$, $a\!\in\!(0,n)$, and $L\!\in\!1\!+\!q\Q[[q]]$ is defined~by
\BE{Lbs_e} L(q)^n-qL(q)^a=1,\EE
then there exists $C_a\!\in\!\R^+$ such that 
$$\Bigg|\LRbr{\frac{L(q)^{1-n+k}}
{(1\!-\!q)^{\de}\left(a\!+\!(n\!-\!a)L(q)^n\right)^{k'}}}_{q;d}\Bigg|\le C_a
\quad\forall~k,k'\!\in\bar\Z^+,~k\!\le\!2n^2,\,k'\!\le\!2n\!+\!1,\,\de\!=\!0,1.$$
\end{lmm}

\begin{proof}
Let $\nu\!=\!n\!-\!a$.
We show that \e_ref{Lbs_e} defines a  holomorphic map $q\!\lra\!L(q)$ on a neighborhood
of the closed unit disk  $\bar{D}\!\subset\!\C$ such that 
$$L(q),a\!+\!\nu L(q)\neq0 \qquad\forall~q\!\in\!\bar{D}.$$
Thus, the radius of convergence of the Cauchy series around $q\!=\!0$ for the holomorphic function
$$q\lra\frac{L(q)^k}
{\left(a\!+\!\nu L(q)^n\right)^{k'}}$$
is greater than~1.
Let
$$S=\big\{(q,z)\!\in\!\C^2\!:~z^n\!-\!qz^a\!=\!1\big\}.$$
Since the differential of the defining equation is surjective for $z\!\neq\!0$,
$S$ is a smooth curve in~$\C^2$.
The projection map $\pi_1\!: S\!\lra\!\C$ to the first coordinate is an $n$-fold cover 
branched at the points $(q,z)\!\in\!S$ such that 
\begin{equation*}\begin{split}
nz^{n-1}-qaz^{a-1}=0 \quad\Lra\quad q=\frac{n}{a}z^{\nu} \quad&\Lra\quad
z^n=-\frac{a}{\nu}\\ 
\quad&\Lra\quad
|q|=\frac{n}{a}\cdot\bigg(\frac{a}{\nu}\bigg)^{\nu/n}
>\bigg(\frac{n}{a}\bigg)^{a/n}>1.
\end{split}\end{equation*}
Thus, $\pi_1$ is an unramified cover of an open neighborhood $U$ of $\bar{D}$,
and its restriction to the component of $\pi^{-1}(0)$ containing $(0,1)$
induces a holomorphic map 
$$U\lra\C,\qquad q\lra L(q),$$
solving \e_ref{Lbs_e}.
It is immediate from \e_ref{Lbs_e} that $L(q)\!\neq\!0$ for all $q$, if $a\!>\!0$.
On the other hand,
\begin{equation*}\begin{split}
1+\frac{\nu}{n}qL(q)^a=0 \quad\Lra\quad q=-\frac{n}{\nu}L(q)^{-a}
\quad&\Lra\quad L(q)^n=-\frac{a}{\nu}\\
\quad&\Lra\quad |q|=\frac{n}{\nu}\cdot\bigg(\frac{\nu}{a}\bigg)^{a/n}
>\bigg(\frac{n}{\nu}\bigg)^{\nu/n}>1,
\end{split}\end{equation*}
as claimed.
\end{proof}


\begin{lmm}\label{Phibound_lmm}
Let $\Phi_0,\Phi_1,\ldots\in\Q[[q]]$ be as in Proposition~\ref{Fexp_prp}.
There exists $C_{\a}\!\in\!\R^+$ such that 
$$\bigg|\LRbr{\frac{\Phi_b(q)}{\Phi_0(q)}}_{q;d} \bigg|
\le b!C_{\a}^b \LRbr{(1\!-\!C_{\a}q)^{-b}}_{q;d} 
\qquad\forall\,b,d\!\in\!\bar\Z^+\,. $$
\end{lmm}

\begin{proof}
For $k\!=\!1,2,\ldots,n$, define 
$$\wt\fL_k\!: \Q[[q]]\lra\Q[[q]] \qquad\hbox{by}\quad
\wt\fL_k(\Phi)=
\frac{1}{L(q)^{k-1}\Phi_0(q)(|\a|\!+\!\nua L(q)^n)}\fL_k(\Phi_0\Phi)\,,$$
with $\fL_k$ and $\Phi_0$ given by \e_ref{fLkdfn_e} and \e_ref{F0exp_e2}, respectively.
These differential operators are of the form
\BE{tifLk_e}\wt\fL_k=\sum_{i=0}^{i=k}\ti{h}_{k,k-i}(q)D^i
\qquad\hbox{with}\quad \ti{h}_{k,i}\in\Q[[q]].\EE
Note that by \e_ref{Ldfn_e0} and 
\BE{Lder_e} 
\frac{L'}{L}=\frac{L^n-1}{|\a|+\nua L^n}=\frac{\a^{\a}qL(q)^{|\a|}}{|\a|+\nua L^n}\,.\EE
We now consider three separate cases.\\

\noindent
(1) Suppose $0\!<\!|\a|\!<\!n$. We show that there exists $C_{\a}\!\in\!\R^+$ such that 
\BE{Phibound_e} 
\bigg|\LRbr{\frac{\Phi_b(q)}{\Phi_0(q)}}_{q;d} \bigg|
\le b!C_{\a}^b \LRbr{(1\!-\!\a^{\a}q)^{-b}}_{q;d} 
\qquad\forall\,b,d\!\in\!\bar\Z^+\,. \EE
By \e_ref{F0exp_e2} and \e_ref{Lder_e}, for each $j\!\in\!\Z^+$ 
there exists $p_j\!\in\Q[u]$ such that  
\BE{Phi0der_e}
\frac{D^j\Phi_0}{\Phi_0}=\frac{(L^n\!-\!1)p_j(L^n)}{(|\a|\!+\!\nua L^n)^{2j}}\
=\frac{\a^{\a}qL(q)^{|\a|} p_j(L^n)}{(|\a|\!+\!\nua L^n)^{2j}}\,,
\quad\deg p_j\le 2j\!-\!1.\EE
By~\e_ref{Hdfn_e}, for each $j\!\in\!\Z^+$ there exist $p_{m,j}\!\in\!\Q[u]$ such~that 
$$\H_{m,j}(u)=\frac{(u\!-\!1)p_{m,j}(u)}{(|\a|+\nua u)^{2j-1}}\,,
\quad \deg p_{m,j}\le 2j\!-\!2,$$
where $\H_{m,j}\!\in\!Q(u)$ is the function defined in Section~\ref{HG_subs}.
Thus, by \e_ref{fLkdfn_e} and \e_ref{Phi0der_e}, there exist $\ti{p}_{k,i}\!\in\!\Q[u]$ such~that
$$\ti{h}_{k,i}=\frac{1}{L^{k-1}}\cdot
\frac{(qL^{|\a|})^{\de_{i,k}}\ti{p}_{k,i}(L^n)}{(|\a|\!+\!\nua L^n)^{2i+1}}\,,
\qquad \deg \ti{p}_{k,i}\le 2i\!+\!1-\de_{i,k}.$$
Let $C\!\ge\!1$ be the maximum of the absolute values of the coefficients of the polynomials 
$(2i\!+\!1)\ti{p}_{k,i}$, with $i\!=\!0,1,\ldots,k$ and $k\!=\!2,3,\ldots,n$.
Thus,
\BE{tiHkl_e}
\Big|\LRbr{(1\!-\!\a^{\a}q)^{-b}\ti{h}_{k,i}(q)}_{q;d}\Big|
\le CC_{|\a|} \LRbr{q^{\de_{i,k}}(1\!-\!\a^{\a}q)^{-b}}_{q;d}
\qquad\forall\,k\!=\!2,3,\ldots,n,\,b\!\in\!\Z^+,\EE
where $C_{|\a|}$ is as in Lemma~\ref{Lest_lmm}.
We show that \e_ref{Phibound_e} holds with 
$$C_{\a}=n^2CC_{|\a|}\a^{\a}.$$
This is indeed the case for $b\!=\!0$.
Suppose $b^*\!\ge\!1$ and the bound holds for all $b\!<\!b^*$.
By~\e_ref{PhiODE}, \e_ref{tifLk_e}, \e_ref{tiHkl_e}, and the inductive assumption,
\begin{equation*}\begin{split}
\bigg|\LRbr{D\bigg(\frac{\Phi_{b^*}(q)}{\Phi_0(q)}\bigg)}_{q;d} \bigg|
&\le \sum_{k=2}^{k=n} 
\bigg|\LRbr{\wt\fL_k\bigg(\frac{\Phi_{b^*-k+1}(q)}{\Phi_0(q)}\bigg)}_{q;d}\bigg|\\
&\le n^2CC_{|\a|}\cdot C_{\a}^{b^*-1}b^*!b^*(\a^{\a})^2
\LRbr{q(1\!-\!\a^{\a}q)^{-b^*-1}}_{q;d}\,.
\end{split}\end{equation*}
Integrating this inequality, we find that \e_ref{Phibound_e} holds for $b\!=\!b^*$ as well.\\

\noindent
(2) Suppose next that $|\a|\!=\!n$. We show that \e_ref{Phibound_e} still holds.
Since $nDL/L\!=\!(L^n-1)$ in this case, for each $j\!\in\!\Z^+$ 
there exists $p_j\!\in\Q[u]$ such~that  
\BE{Phi0dern_e}
\frac{D^j\Phi_0}{\Phi_0}=(L^n\!-\!1)p_j(L^n)\
=\a^{\a}qL(q)^n p_j(L^n)\,,\quad\deg p_j\le j\!-\!1\,.\EE
On the other hand, by~\e_ref{Hdfn_e} for each $j\!\in\!\Z^+$ 
there exist $p_{m,j}\!\in\!\Q[u]$ such~that 
$$\H_{m,j}(u)=(u\!-\!1)p_{m,j}(u), \quad \deg p_{m,j}\le j\!-\!1\,.$$
It follows that there exists $\ti{p}_{k,i}\!\in\!\Q[u]$ such that
\BE{tiHkln_e}
\ti{h}_{k,i}=\frac{1}{L^{k-1}}\big(qL(q)^n\big)^{\de_{i,k}}\ti{p}_{k,i}(L^n)\,,
\quad \deg \ti{p}_{k,i}\le i\!-\!\de_{i,k}\,.\EE
Let $C\!\ge\!1$ be the maximum of the absolute values of the coefficients of the polynomials 
$(i\!+\!1)\ti{p}_{k,i}$, with $i\!=\!0,1,\ldots,k$ and $k\!=\!2,3,\ldots,n$.
Thus,
\BE{tiHkl_e2}
\Big|\LRbr{(1\!-\!\a^{\a}q)^{-b}\ti{h}_{k,i}(q)}_{q;d}\Big|
\le C \LRbr{q^{\de_{i,k}}(1\!-\!\a^{\a}q)^{-b-k}}_{q;d}
\quad\forall\,k\!=\!2,3,\ldots,n,\,b\!\in\!\Z^+\,;\EE
see \e_ref{Ldfn_e2}.
The same inductive argument as at the end of~(1) now shows that 
 \e_ref{Phibound_e} holds with $C_{\a}=n^2C\a^{\a}$.\\

\noindent
(3) Finally, suppose $|\a|\!=\!0$, i.e.~$\a\!=\!()$.
We show that there exist $C_{\eset},C_{b,r}\!\in\!\Q$ for $b,r\!\in\!\bar\Z^+$ such~that 
\BE{Phibound_e0}
\frac{\Phi_b}{\Phi_0}=\sum_{r=0}^{(n+1)b}C_{b,r}L^{-r}\,,
\qquad \sum_{r=0}^{(n+1)b}|C_{b,r}|\le b!C_{\eset}^b\,\qquad\forall\,b\!\in\!\Z^+.\EE 
This implies the claim, since 
$$\big|\lrbr{L(q)^{-r}}_{q;d}\big|\le  
\big|\lrbr{L(q)^{-2nb}}_{q;d}\big|
=\bigg|\binom{-2b}{d}\bigg|=\binom{2b\!+\!d\!-\!1}{d}
\le 2^{2b+d} \le 2^{2b}\lrbr{(1\!-\!2q)^b}_{q;d}$$
for all $r\!\le\!2nb$ and $b\!\in\!\Z^+$.\\

\noindent
Since $nDL/L\!=\!(1-L^{-n})$ in this case,  there exist $C_{r;i}^{(j)}\!\in\!\Q$ such that 
$$D^iL^{-r}=L^{-r}\frac{DL}{L}\sum_{j=0}^{i-1}(r\!+\!nj)C_{r;i}^{(j)}L^{-nj}\,,
\quad \sum_{j=0}^{i-1}\big|C_{r;i}^{(j)}\big|\le
2^{i-1}\prod_{j=0}^{i-2}\frac{r\!+\!nj}{n} 
\qquad\forall\,r\!\in\!\bar\R^+\,,i\!\in\!\Z^+.$$
On the other hand, by~\e_ref{Hdfn_e} for each $j\!\in\!\Z^+$ 
there exist $p_{m,j}\!\in\!u\cdot\Q[u]$ 
$$\H_{m,j}(u)=(u\!-\!1)p_{m,j}(1/u), \quad \deg p_{m,j}\le j\,.$$
It follows that there exist $\ti{p}_{k,i}\!\in\!\Q[u]$ such that
\BE{tiHkl_e0}
\ti{h}_{k,i}=\frac{1}{L^{k-1}}\ti{p}_{k,i}(L^{-n})\bigg(\frac{DL}{L}\bigg)^{\de_{i,k}}\,,
\quad \deg \ti{p}_{k,i}\le i\!-\!\de_{i,k} \qquad\forall\,i\in\bar\Z^+\,.\EE
Thus,  there exist $\ti{C}_{r;k}^{(j)}\!\in\!\Q$ such that 
\BE{tifLkr_e}\ti\fL_k L^{-r}=L^{-r-k}DL
\sum_{j=0}^{k-1}(r\!+\!nj\!+\!1)\ti{C}_{r;k}^{(j)}L^{-nj}\,,
\qquad \sum_{j=0}^{k-1}\big|\ti{C}_{r;k}^{(j)}\big|\le
2^kC\prod_{j=1}^{k-1}\frac{r\!+\!nj}{n} \EE
for all $r\!\in\!\bar\R^+$ and $k\!\in\!\Z^+$,
where $C\!\ge\!1$ is the maximum of the absolute values of the coefficients of 
the polynomials $(k\!+\!1)\ti{p}_{k,i}$
with $i\!=\!0,1,\ldots,k$ and $k\!=\!1,2,\ldots,n$.
We show that \e_ref{Phibound_e0} holds with  
$$C_{\eset}=4\bigg(\frac{2n\!+\!2}{n}\bigg)^nC.$$
This is indeed the case for $b\!=\!0$.
Suppose $b^*\!\ge\!1$ and the claim holds for all $b\!<\!b^*$.
By~~\e_ref{PhiODE}, the inductive assumption, and~\e_ref{tifLkr_e},  
there exist $C_{b^*,r}'\!\in\!\Q$ such that 
\begin{alignat}{1}
\label{ODEa0_e}
D\bigg(\frac{\Phi_{b^*}}{\Phi_0}\bigg)
&= -\sum_{k=2}^{k=n}\wt\fL_k\bigg(\frac{\Phi_{b^*-k+1}}{\Phi_0}\bigg)
=-\frac{DL}{L}\sum_{r=1}^{(n+1)b^*}\!\!\!\!rC_{b^*,r}L^{-r}\,,\\
\sum_{r=1}^{(n+1)b^*}\!\!\!\!\big|C_{b^*,r}\big|
&\le C\sum_{k=2}^n\sum_{r=0}^{(n+1)(b^*-1+k)}\sum_{j=0}^{k-1}
\big|\ti{C}_{r;k}^{(j)}\big|\big|C_{b^*-k+1,r}\big|\notag\\
&\le C\sum_{k=2}^n\Bigg(2^k
\prod_{j=1}^{k-1}\bigg(\frac{(n\!+\!1)(b^*\!-\!k\!+\!1)+nj}{n}\bigg)
\cdot(b^*\!-\!k\!+\!1)!C_{\eset}^{b^*-k+1}\Bigg) \notag \\
&\le 2CC_{\eset}^{b^*-1}\sum_{k=2}^n\Bigg(\bigg(2\frac{(n\!+\!1)}{n}\bigg)^{\!k-1}\,
\prod_{j=1}^{k-1}\big(b^*\!-\!k\!+\!1+j\big)\cdot(b^*\!-\!k\!+\!1)!\Bigg)
\le \frac{C_{\eset}^{b^*}}{2}b^*!\,.\notag 
\end{alignat}
Thus, integrating \e_ref{ODEa0_e} and using $\Phi_{b^*}\!\in\!q\cdot\Q[[q]]$,
we find that \e_ref{Phibound_e0} holds for $b\!=\!b^*$ as well.
\end{proof}

\begin{rmk}\label{bound_rmk} 
The above arguments are based on 
the fact that all coefficients of $(1\!-\!q)^{-\al}$ are nonnegative (actually positive)
if $\al\!>\!0$, non-decreasing with~$\al$, are at least as large in the absolute values
as the coefficients of $(1\!\pm\!q)^{\al}$, and non-decreasing with $d$ if $\al\!\ge\!1$.
\end{rmk}

\begin{crl}\label{Phibound_crl}
Let $\Phi_{p;b},\Phi_{m,\c}(q)\!\in\!\Q[[q]]$ be as in~\e_ref{Fpexp_e} 
and~\e_ref{Phimcdfn_e}.
There exists $C_{\a}\!\in\!\R^+$ such~that 
\begin{alignat*}{2}
\bigg|\LRbr{\frac{\Phi_{p;b}(q)}{\Phi_0(q)}}_{q;d} \bigg|
&\le  b!C_{\a}^b \LRbr{(1\!-\!C_{\a}q)^{-b-1}}_{q;d} 
&\quad&\forall\,b,d\!\in\!\bar\Z^+,\,p\!\in\!\llfloor{n}\rrfloor;\\
\bigg|\LRbr{\Phi_{m,\c}(q)}_{q;d} \bigg|&
\le \frac{(m\!+\!|\c|)!}{|\c|!}\binom{|\c|}{\c}
\prod_{r=1}^{\i}\bigg(\frac{1}{r\!+\!1}\bigg)^{c_r}
 C_{\a}^{\|\c\|} \LRbr{(1\!-\!C_{\a}q)^{-\|\c\|-1}}_{q;d}  
&\quad& \forall\,m,d\!\in\!\bar\Z^+,\,\c\!\in\!(\bar\Z^+)^{\i}.
\end{alignat*}
\end{crl}

\begin{proof}
It is sufficient to obtain the first bound for the power series 
$\hat\Phi_{p;b}\in\Q[[q]]$, $-l\!\le\!p\!\le\!n\!-\!1\!-\!l$, defined in~\e_ref{Psipr_e}.
If $0\!<\!|\a|\!<\!n$, it follows by induction on $b\!\in\!\bar\Z^+$ and $p$
(from~0 up to $n\!-\!1\!-\!l$ and down to~$-l$) from Lemma~\ref{Phibound_lmm},
the $j\!=\!1$ case of~\e_ref{Phi0der_e}, and Lemma~\ref{Lest_lmm}.
For $|\a|\!=\!n$, Lemma~\ref{Fpest_lmm} implies that there exists $C\!\in\!\R^+$
such that 
\BE{Ipbound_e}\big|\lrbr{I_0(q)^{k_0}I_1(q)^{k_1}\ldots I_{n-l}(q)^{k_{n-l}}}_{q;d}\big|\le C^d
\qquad\forall\,d\!\in\!\bar\Z^+,\,k_0,k_1,\ldots,k_{n-l}\!\in\!\{0,\pm1\}.\EE
By induction on $b$ and $|p|$ (with the base case being Lemma~\ref{Phibound_lmm})
along with~\e_ref{Ldfn_e2} and the $j\!=\!1$ case of~\e_ref{Phi0dern_e}, 
this implies that 
$$\Bigg|\LRbr{\frac{\hat\Phi_{l+p;b}(q)}{\Phi_0(q)}}_{q;d} \Bigg|
\le  C_{\a;p}^bb! \LRbr{(1\!-\!C_{\a;p}q)^{-b-|p|/n}}_{q;d}
\qquad\forall\,b,d\!\in\!\bar\Z^+\,,$$
for some $C_{\a;p}\!\in\!\R^+$.
The same estimate holds if $|\a|\!=\!0$, by Lemma~\ref{Phibound_lmm}
and~\e_ref{Ldfn_e2}.
The second bound follows directly from Lemma~\ref{Phibound_lmm} and
\e_ref{F0exp_e2}, along with Lemma~\ref{Lest_lmm} if $0\!<\!|\a|\!<\!n$
and~\e_ref{Ipbound_e} if $|\a|\!=\!n$.
\end{proof}

\begin{proof}[Proof of Proposition~\ref{coeffbound_prp}]
By Corollary~\ref{Phibound_crl}, the absolute value of each nonzero factor~$\lrbr{\cdot}$ in 
\e_ref{ncCdfn2_e} is bounded above~by
\begin{gather*}
\frac{(m_v\!+\!|\c_v|)!}{|\c_v|!}\binom{|\c_v|}{\c_v}\!
\prod_{r=1}^{\i}\!\!\bigg(\frac{1}{r\!+\!1}\bigg)^{\!c_{v;r}}
\!\!\!\!\!\!\prod_{s\in\eta^{-1}(v)}\!\!\!\frac{1}{b_s!}\cdot
\prod_{e\in\tnE_{\Ga}^-(v)}\!\!\!\!\!\!\frac{(b_e^-\!+\!1\!+\!b_e')!}{b_e^-!}
\cdot\frac{(b_{e_v}^+\!-\!b_{e_v}')!}{b_{e_v}^+!}
C_{\a}^{\De_v(\b')}\!\!\LRbr{(1\!-\!C_{\a}q)^{-\De_v(\b')}}_{q;d_v}\\
\hbox{where}\qquad \De_v(\b')
=4m_v+8-|\b|_{\eta^{-1}(v)}+|\b'|_{\tnE_{\Ga}^-(v)}-b_{e_v}'\,.
\end{gather*}
Thus, by \e_ref{mvsum_e}, the absolute value of each nonzero summand 
(product of factors over $v\!\in\!\Ver$)
in \e_ref{ncCdfn2_e} is bounded above~by
$$\frac{C_{\a}^{8N}\!\!\LRbr{(1\!-\!C_{\a}q)^{-8N}}_{q;d}}{\b!}
\prod_{v\in\Ver}\!\!\Bigg(\!\frac{(m_v\!+\!|\c_v|)!}{|\c_v|!}\binom{|\c_v|}{\c_v}
\prod_{r=1}^{\i}\!\!\bigg(\frac{1}{r\!+\!1}\bigg)^{\!c_{v;r}}\!\!\Bigg)
\cdot \prod_{e\in\Edg}\!\!\!\frac{(b_e^-\!+\!1\!+\!b_e')!(b_e^+\!-\!b_e')!}{b_e^-!b_e^+!}\,.$$
Note that 
\begin{equation*}\begin{split}
\sum_{b_e^-+b_e^+=b_e^{\pm}}\!\!\!\!\!\!\!
\frac{(b_e^-\!+\!1\!+\!b_e')!(b_e^+\!-\!b_e')!}{b_e^-!b_e^+!}
\le \sum_{b_e^-+b_e^+=b_e^{\pm}}\!\!\!\!\!\!\!
\frac{(b_e^-\!+\!1\!+\!b_e^+)!}{b_e^-!b_e^+!}
&= (b_e^{\pm}\!+\!1)\!\!\!\sum_{b_e^-+b_e^+=b_e^{\pm}}\!\!\!
\binom{b_e^-\!+\!b_e^+}{b_e^-}\\
&= (b_e^{\pm}\!+\!1)2^{b_e^{\pm}}\le 4^{b_e^{\pm}}\,.
\end{split}\end{equation*}
Since each tuple $\b''$ is a partition of $N\!-\!3\!-\!|\Edg|\!-\!|\b^-|\!-\!|\b^+|\!-\!\|\c\|$
into $N$ ordered parts, where
$$\|\c\|=\sum_{v\in\Ver}\!\!\!\|\c_v\|\,,$$
the number of such tuples with $|\b^-|\!+\!|\b^+|$ and $\|\c\|$ fixed
is at~most
$$\binom{N\!-\!3\!-\!|\Edg|\!-\!|\b^-|\!-\!|\b^+|\!-\!\|\c\|+N\!-\!1}{N\!-\!1}
\le 2^{2(N-2)-|\b^-|-|\b^+|-\|\c\|}\,.$$
Thus, the absolute value of the sum in~\e_ref{ncCdfn2_e} with $\Ga$, $(\p',\b',\bft)$,
and~$\c$ fixed is bounded above~by
$$\frac{C_{\a}'^{8N}\LRbr{(1\!-\!C_{\a}q)^{-8N}}_{q;d}}{\b!}
2^{-\|\c\|}\!\!\!
\prod_{v\in\Ver}\!\!\Bigg(\!\frac{(m_v\!+\!|\c_v|)!}{|\c_v|!}\binom{|\c_v|}{\c_v}
\prod_{r=1}^{\i}\!\!\bigg(\frac{1}{r\!+\!1}\bigg)^{\!c_{v;r}}\!\!\Bigg)
.$$
Since $|1\!-\!2\ln2|<1$, by the Binomial Theorem
\begin{equation*}\begin{split}
&\sum_{(\c)_{v\in\Ver}\in((\bar\Z^+)^{\i})^{\Ver}}\!\!\!
2^{-\|\c\|}\!\!\!
\prod_{v\in\Ver}\!\!\Bigg(\!\frac{(m_v\!+\!|\c_v|)!}{m_v!}\binom{|\c_v|}{\c_v}
\prod_{r=1}^{\i}\!\!\bigg(\frac{1}{r\!+\!1}\bigg)^{\!c_{v;r}}\!\!\Bigg)\\
&\hspace{.5in}
=\prod_{v\in\Ver}\bigg(\!1-\sum_{r=1}^{\i}\frac{w^r}{r\!+\!1}\bigg)^{-m_v-1}
\Bigg|_{w=1/2}
=\bigg(\!2+\frac{\ln(1\!-\!w)}{w}\bigg)^{-(N-2)}\Bigg|_{w=1/2}
=\big(2(1\!-\!\ln2)\big)^{N-2}\,.
\end{split}\end{equation*}
Since $b_e'\!\le\!b_e^+$ for $e\!\in\!\Edg$ and nonzero summands in \e_ref{ncCdfn2_e}, 
$|\b'|\!\le\!N\!-\!3\!-\!|\Edg|$.
The number of such tuples~is
$$\binom{N\!-\!3\!-\!|\Edg|+|\Edg|}{|\Edg|}\le 2^{N-3}.$$
Thus, the absolute value of the contribution of each trivalent $N$-marked tree~$\Ga$ 
to $\nc_{\p,\b}^{(d,0)}$ is bounded above~by
$$\frac{\ti{C}_{\a}^N}{\b!}\binom{-8N}{d}C_{\a}^d\cdot
\prod_{v\in\Ver}\!\!\!m_v!
=\frac{\ti{C}_{\a}^N}{\b!}\binom{8N\!+\!d\!-\!1}{d}C_{\a}^d\cdot
\prod_{v\in\Ver}\!\!\!m_v!
\le \frac{\ti{C}_{\a}^N}{\b!}2^{8N+d}C_{\a}^d\cdot
\prod_{v\in\Ver}\!\!\!m_v!\,.$$
Combining this with Lemma~\ref{graphnum_lmm} below,
we obtain the claimed bound for~$\nc_{\p,\b}^{(d,0)}$. 
\end{proof}

\begin{rmk}\label{projbound_rmk}
In the $|\a|\!=\!0$ case (projective space), a bound of the form 
$(N!/\b!)C^{N-3-|\b|}$ can be obtained using the last description of~$c_{\p,\b}^{(d,0)}$
in Section~\ref{graph_subs} and~\e_ref{Phibound_e0}.
\end{rmk}

\begin{lmm}\label{graphnum_lmm}
There exist $C\!\in\!\R^+$ such that 
$$a_{N-1}\equiv \sum_{\Ga}\prod_{v\in\Ver}\!\!\!m_v!\le C^NN! \qquad\forall\,N\ge3,$$
where the sum is taken over all trivalent $N$-marked trees.
\end{lmm}

\begin{proof}
Let $a_1\!=\!1$ and
$$f(q)=\sum_{N=1}^{\i}\frac{a_N}{N!}q^N\in \Q[[q]].$$
Considering the vertex of an $(N\!+\!1)$-marked tree $\Ga$ to which the last marked point 
is attached, we observe that 
\begin{equation*}\begin{split}
a_N&=\sum_{k=2}^{k=N}
\frac{1}{k!}\sum_{(N_1,\ldots,N_k)\in\cP_k(N)}
\binom{N}{N_1,\ldots,N_k}(k\!-\!2)!a_{N_1}\ldots a_{N_k}\\
&=N!\sum_{k=2}^{k=N}\bigg(\frac{1}{k\!-\!1}\!-\!\frac{1}{k}\bigg)
\sum_{(N_1,\ldots,N_k)\in\cP_k(N)}\frac{a_{N_1}}{N_1!}\ldots\frac{a_{N_k}}{N_k!}\,.
\end{split}\end{equation*}
This recursion is equivalent to the condition that 
\BE{fcond_e} f(q)=q+f(q)+\big(f(q)\!-\!1\big)\sum_{k=1}^{\i}\frac{f(q)^k}{k}
\quad\Llra\quad 
\big(1-f(q)\big)\ln\big(1-f(q)\big)=-q.\EE
By the Inverse Function Theorem, $f(q)$ is an analytic function on a neighborhood of $q\!=\!0$
and so $a_N/N!\le C^N$ for some $C\!\in\!\R^+$.
\end{proof}

\begin{rmk}\label{graphnum_rmk}
As noticed by the author for small values of~$N$ and confirmed in general by P.~Johnson on 
{\it Math Overflow}, $a_{N-1}\!=\!(N\!-\!2)^{N-2}$.
By~\e_ref{fcond_e},
\BE{fwrel_e} f(q)=1-e^{W(-q)}\,,\EE
where $W\!\in\!\Q[[q]]$ is the Lambert $W$ function, i.e.~the analytic function on 
a neighborhood of $0\!\in\!\C$ defined~by 
$$W(q)e^{W(q)}=q, \qquad W(0)=0.$$
As can be seen from the Lagrange inversion formula,
\BE{Wexp_e} e^{W(q)}=1+q-\sum_{N=2}^{\i} \frac{(N\!-\!1)^{(N-1)}}{N!}(-q)^N\,.\EE
Along with \e_ref{fwrel_e}, this implies the claim.
\end{rmk}

\appendix

\section{Existence of Asymptotic Expansions}
\label{HGprp_pf}

\noindent
In this appendix, we show that power series
\BE{cY0dfn_e}\cY_0(\hb,\x,q)\equiv
\sum_{d=0}^{\i}q^d\frac{\prod\limits_{k=1}^{k=l}\prod\limits_{r=0}^{a_kd-1}(a_k\x+r\hb)}
{\prod\limits_{r=1}^{r=d}\left(\prod\limits_{k=1}^{k=n}\!\!(\x\!-\!\al_k\!+\!r\hb)
-\prod\limits_{k=1}^{k=n}\!\!(\x\!-\!\al_k)\right)}\in\Q_{\al}(\x,\hb)\big[\big[q\big]\big]\EE
admits an expansion of the form~\e_ref{cYexp_e} and then prove Proposition~\ref{Fexp_prp}.
The arguments here are motivated by \cite[Section~2]{ZaZ}.

\begin{lmm}\label{cY0_lmm}
The power series $\cY_0(\hb,\x,q)$ admits an expansion of the form
\BE{cY0exp_e}
\cY_0(\hb,\x,q)=e^{\xi(\x,q)/\hb}\sum_{b=0}^{\i}\Phi_{0;b}(\x,q)\hb^b\,\EE
with $\xi,\Phi_{0;1},\Phi_{0;2},\ldots\in q\Q_{\al}(\x)[[q]]$
and $\Phi_{0;0}\in 1+q\Q_{\al}(\x)[[q]]$.
\end{lmm}


\begin{proof}
Since $\cY_0\!\in\!1+q\Q_{\al}(\hb,\x)[[q]]$, there is an expansion
\BE{cY0ln_e}
\ln\cY_0(\hb,\x,q)=\sum_{d=1}^{\i}\sum_{b=b_{\min}(d)}^{\i}\!\!\!\!\!C_{d,b}(\x)\hb^bq^d\,\EE
around $\hb\!=\!0$, with $C_{d,b}(\x)\!\in\!\Q_{\al}(\x)$;
we can assume that $C_{d,b_{\min}(d)}\!\neq\!0$ if $b_{\min}(d)\!<\!0$.
The claim of Lemma~\ref{cY0_lmm} is equivalent to the statement 
$b_{\min}(d)\!\ge\!-1$ for all $d\!\in\!\Z^+$; in such a case
$$\xi(\x,q)=\sum_{d=1}^{\i}C_{d,-1}(\x)q^d\,.$$
Suppose instead $b_{\min}(d)\!<\!-1$ for some $d\!\in\!\Z^+$.
Let
\BE{dmin_e} d^*=\min\big\{d\!\in\!\Z^+\!:\,b_{\min}(d)\!<\!-1\big\}\ge1,
\qquad b^*=b_{\min}(d^*)\le-2.\EE
The power series $\cY_0$ satisfies the differential equation
\BE{cY0ode_e}\Bigg\{\prod_{k=1}^{k=n}\big(\x\!-\!\al_k+\hb D\big)
-q\prod_{k=1}^{l=1}\prod_{r=0}^{a_k-1}\big(a_k\x\!+\!a_k\hb D\!+\!r\hb\big)\Bigg\}\cY_0(\hb,\x,q)
=\prod\limits_{k=1}^{k=n}(\x\!-\!\al_k)\cdot\cY_0(\hb,\x,q),\EE
where $D\!=\!q\frac{\tnd}{\tnd q}$.
By \e_ref{cY0ln_e}, \e_ref{dmin_e}, and induction on the number of derivatives taken,
\BE{cY0lnchng_e}\begin{split}
\frac{\left\{\prod\limits_{k=1}^{k=n}\big(\x\!-\!\al_k+\hb D\big)\right\}\cY_0(\hb,\x,q)}
{\prod\limits_{k=1}^{k=n}(\x\!-\!\al_k)\cdot\cY_0(\hb,\x,q)}
&=1+\sum_{k=1}^{k=n}\frac{d^*C_{d^*,b^*}}{\x\!-\!\al_k}\hb^{b^*+1}q^{d^*}+A(\hb,\x,q)\,,\\
q\,\frac{\left\{\prod\limits_{k=1}^{l=1}\prod\limits_{r=0}^{a_k-1}
\big(a_k\x\!+\!a_k\hb D\!+\!r\hb\big)\right\}\cY_0(\hb,\x,q)}
{\prod\limits_{k=1}^{l=1}\prod\limits_{r=0}^{a_k-1}(a_k\x\!+\!r\hb)\cdot\cY_0(\hb,\x,q)}
&= B(\hb,\x,q)\,,
\end{split}\EE
for some
$$A,B\in q\Q_{\al}(\hb,\x)_0\big[\big[q\big]\big]
+q^{d^*}\hb^{b^*+2}\Q_{\al}(\hb,\x)_0\big[\big[q\big]\big]
+q^{d^*+1}\Q_{\al}(\hb,\x)\big[\big[q\big]\big],$$
where $\Q_{\al}(\hb,\x)_0\subset\Q_{\al}(\hb,\x)$ is the subring of rational functions
in $\al$, $\hb$, and $\x$ that are regular at $\hb\!=\!0$.
Combining \e_ref{cY0ode_e} and \e_ref{cY0lnchng_e}, we conclude that $C_{d^*,b^*}\!=\!0$,
contrary to the assumption.
\end{proof}

\begin{crl}\label{F0_crl}
The power series $F_0\!\in\!\Q(w)[[q]]$ defined by \e_ref{F0dfn_e} 
admits an asymptotic expansion of the form~\e_ref{F0exp_e}.
\end{crl}

\begin{proof}
The existence of an asymptotic expansion~\e_ref{F0exp_e} is equivalent to
the existence of an expansion of the form~\e_ref{cYexp_e} for
$$F_0(\hb^{-1},q)\equiv\cY_0(\hb,1,q)\big|_{\al=0}\,.$$
Thus, Corollary~\ref{F0_crl} follows from Lemma~\ref{cY0_lmm}.
\end{proof}

\begin{rmk}\label{F0exp_rmk}
It is possible to give a somewhat different proof of Corollary~\ref{F0_crl}, 
without using Lemma~\ref{cY0_lmm}, which is more in line with~\cite{ZaZ}.
By \cite[Lemma~1.3]{ZaZ}, an element $H\!\in\!\cP$ admits 
an asymptotic expansion~\e_ref{F0exp_e} if $\bM^kH\!=\!H$ 
for some $k\!\in\!\Z^+$.
By \cite[Lemma~4.1]{Po2},  $\bM^nF\!=\!F$ if $|\a|\!=\!n$.
In the $\nua\!>\!0$ case, the coefficients $\ti\nc_{p,s}^{(d)}$ in \e_ref{Flpdfn_e} 
with $d\!\ge\!1$ and $\nua d\le p\!-\!s$ are determined by the requirement that 
the resulting function $F_p(w,q)$ is holomorphic at $w\!=\!0$ with value $1\!\in\!Q[[q]]$;
see~\e_ref{ncrec_e}.
On the other hand, $F_n\!=\!F_0$ if these coefficients are given~by
$$\sum_{s=0}^{|\a|-l}\ti\nc_{n,l+s}^{(1)}w^s
=-\lr{\a}\prod_{k=1}^{k=l}\prod_{r=1}^{a_k-1}(a_kw\!+\!r),\qquad
\ti\nc_{n,s}^{(d)}=0~~\forall\,d\!\ge\!2.$$
Since $F_0$ is holomorphic at $w\!=\!0$ with value $1\!\in\!Q[[q]]$, 
it follows that indeed $F_n\!=\!F_0$.
The proof of \cite[Lemma~1.3]{ZaZ} can be adjusted to show that 
this in turn implies that $F_0$ admits an asymptotic expansion of the form~\e_ref{F0exp_e}.
\end{rmk}

\noindent
In the remainder of this appendix, we prove Proposition~\ref{Fexp_prp}.
Since $F\!=\!\bD^lF_0$ and $F_0$ admits an asymptotic expansion of the form~\e_ref{F0exp_e},
so does~$F$.
The function $F(w,q)$ defined by~\e_ref{Fdfn_e} satisfies the ODE 
$$\bigg\{D_w^n-w^n-qw^{\nua}\,\prod_{k=1}^{k=l}\prod_{r=1}^{r=a_k}(a_kD_w+r)\bigg\}F=0\,,$$
where $D_w\!=\!w\!+\!q\frac{\tnd}{\tnd q}$.
Thus, the power series $\xi,\Phi_0,\Phi_1,\ldots$ introduced in Proposition~\ref{Fexp_prp} satisfy
\BE{HGpf_e3}
\bigg\{\ti{D}_w^n-w^n-qw^{\nua}\,\prod_{k=1}^{k=l}\prod_{r=1}^{r=a_k}(a_k\ti{D}_w+r)\bigg\}
\sum_{b=0}^{\i}\Phi_b(q)w^{-b}=0,\EE
where $\ti{D}_w\!=\!(1\!+\!\xi'(q))w\!+\!q\frac{\tnd}{\tnd q}$.
The resulting equation for the coefficient of $w^n$ gives
\BE{HGpf_e5} 
\big\{(1\!+\!\xi'(q))^n-1-{\a}^{\a}q(1\!+\!\xi'(q))^{|\a|}\big\}\Phi_0(q)=0.\EE
Since $\Phi_0(0)\!=\!1$, combining \e_ref{HGpf_e5} with 
the condition $\xi'(0)\!=\!0$ and comparing with \e_ref{Ldfn_e0},
we obtain the first equation in~\e_ref{PhiODE}.\\

\noindent
By the above, $\ti{D}_w\!=L(q)w\!+\!q\frac{\tnd}{\tnd q}$.
Proceeding as in \cite[Section~2.4]{ZaZ}, but using~\e_ref{Lder_e},
we find that
$$\ti{D}_w^s=\sum_{k=0}^{k=s}\sum_{i=0}^{i=k}\binom{s}{i}\H_{s-i,k-i}(L^n)(Lw)^{s-k}D^i\,,$$
where $\H_{m,j}$ are the rational functions defined by \e_ref{Hdfn_e}.
Thus, 
$$L(q)^n\bigg\{\ti{D}_w^n-w^n-qw^{\nua}\,\prod_{k=1}^{k=l}\prod_{r=1}^{r=a_k}(a_k\ti{D}_w+r)\bigg\}
=\sum_{k=1}^n(Lw)^{n-k}\fL_k\,,$$
where $\fL_k$ is the differential operator of order $k$ given by \e_ref{fLkdfn_e}.
It follows that the second equation in~\e_ref{PhiODE} is the coefficient 
of~$(Lw)^{n-1-b}$ in~\e_ref{HGpf_e3} multiplied by~$L(q)^n$.


\section{Some Combinatorics}
\label{reslmm_pf}

\begin{lmm}\label{comb_l0}
The following identities hold:
\begin{alignat*}{2}
\sum_{\b'\in\cP_m(b')}
\prod_{i=1}^{i=m}\binom{b_i}{b_i'} &=\binom{b_1\!+\!\ldots\!+b_m}{b'}
&\qquad&\forall~m\!\in\!\Z^+,\,b_1,\ldots,b_m,b'\!\in\!\bar\Z^+\,,\\
\sum_{b=0}^{\i}(-1)^b\binom{p}{b}\prod_{t=B-s+1}^{t=B}\!\!\!\!\!(t\!+\!b)
&=(-1)^p s!\binom{B}{s\!-\!p}
&\qquad&\forall~B,p,s\!\in\!\bar\Z^+\,,\\ 
\sum_{p=0}^{\i}(-1)^p\binom{m\!+\!p}{p}\Psi^p &=\frac{1}{(1\!+\!\Psi)^{m+1}}
&\qquad&\forall~m\!\in\!\bar\Z^+.
\end{alignat*}
\end{lmm}

\noindent
The first two statements of this lemma are proved in \cite[Appendix~A]{bcov1}.
The last statement is a special case of the Binomial Theorem; 
here is a direct argument:
\begin{equation*}\begin{split}
\sum_{p=0}^{\i}(-1)^p\binom{m\!+\!p}{p}\Psi^p 
&=\frac{1}{m!}\bigg\{\frac{\tnd}{\tnd\Psi}\bigg\}^m
\sum_{p=0}^{\i}(-1)^p\Psi^{m+p}
=\frac{(-1)^m}{m!}\bigg\{\frac{\tnd}{\tnd\Psi}\bigg\}^m
\sum_{p=0}^{\i}(-1)^p\Psi^p\\
&=\frac{(-1)^m}{m!}\bigg\{\frac{\tnd}{\tnd\Psi}\bigg\}^m \frac{1}{1\!+\!\Psi}
=\frac{1}{(1\!+\!\Psi)^{m+1}}\,.
\end{split}\end{equation*}

\begin{lmm}\label{res_lmm}
If $\ze,\Psi_0,\Psi_1,\ldots\!\in\!Q\Q_{\al}(\hb)[[Q]]$ and
\BE{reslmm_e0}1+\cZ^*(\hb,Q)=e^{\ze(Q)/\hb}\bigg(1+\sum_{b=0}^{\i}\Psi_b(Q)\hb^b\bigg)\,,\EE
then
\BE{res_e}\begin{split}
&\sum_{m'=0}^{\i}\frac{(m'\!+\!m)!}{m'!} 
\sum_{\b\in\cP_{m'}(m-B+m')}
\!\!\!\Bigg(\prod_{k=1}^{k=m'}\frac{(-1)^{b_k}}{b_k!}
\Res{\hb=0}\Big\{\hb^{-b_k}\cZ^*(\hb,Q)\Big\}\Bigg)\\
&~~
=\!\!\sum_{\c\in(\bar\Z^+)^{\i}}\!\!\!\Bigg(\!\!
(-1)^{|\c|+\|\c\|}  (m\!+\!|\c|)!\frac{\ze(Q)^{B-m+\|\c\|}}{(1\!+\!\Psi_0(Q))^{m+1}}
\!\binom{B}{m\!-\!\|\c\|}\!
\prod_{r=1}^{\i}\frac{1}{c_r!}
\bigg(\!\frac{\Psi_r(Q)}{(r\!+\!1)!\,(1\!+\!\Psi_0(Q))}\!\bigg)^{\!c_r}\!\Bigg)
\end{split}\EE
for all $m,B\!\in\!\bar\Z^+$.
\end{lmm}

\begin{proof}
If $\c\!\in\!(\bar\Z^+)^{\i}$, let
$$\Psi^{\c}=\prod_{r=1}^{\i}\Psi_r^{c_r} \,, \qquad
\om(\c)=\prod_{r=1}^{\i}\big((r\!+\!1)!\big)^{c_r}\,.$$
We show that each $\Psi_0^{c_0}\Psi^{\c}$, with $c_0\!\in\!\bar\Z^+$,
enters with the same coefficient on the two sides of~\e_ref{res_e}.\\

\noindent
For $c_0\!\in\!\bar\Z^+$ and $\c\!\in\!(\bar\Z^+)^{\i}$, let
$$S(c_0,\c)=\big\{(r,j)\!\in\!\bar\Z^+\!\times\!\Z^+\!:
(r,j)\!\in\!\{r\}\!\times\![c_r]~\forall\,r\!\in\!\bar\Z^+\big\}.$$
This is a finite set of cardinality $c_0\!+\!|\c|$.
By~\e_ref{reslmm_e0}, for all $b\!\in\!\bar\Z^+$
$$\Res{\hb=0}\Big\{\hb^{-b} \cZ^*(\hb,Q)\Big\}
=\sum_{r=\max(b-1,0)}^{\i}\!\frac{\ze(Q)^{r+1-b}}{(r\!+\!1\!-\!b)!}\Psi_r(Q)
+\begin{cases}
\ze(Q),&\hbox{if}~b\!=\!0;\\
0,&\hbox{if}~b\!>\!0.\end{cases}$$
Thus, for each $\b\!\in\!(\bar\Z^+)^{S(c_0,\c)}$ and every choice of 
disjoint subsets $S_0,S_1,\ldots$ of $[m']$, where
$$m'=B-m+|\b|\,,$$ 
of cardinalities $c_0,c_1,\ldots$, the term $\Psi_0^{c_0}\Psi^{\c}$ appears in the 
$m'$-th summand on the left-hand side of~\e_ref{res_e} with the coefficient
\BE{split_e}\begin{split}
&\frac{(m'\!+\!m)!}{m'!}\,\ze^{m'-c_0-|\c|} \prod_{(r,j)\in S(c_0,\c)}\!\!\!
\bigg(\frac{(-1)^{b_{r,j}}}{b_{r,j}!}\cdot
\frac{\ze^{r+1-b_{r,j}}}{(r\!+\!1\!-\!b_{r,j})!}\bigg)\\
&\qquad\qquad\qquad\qquad\qquad
=\frac{\ze^{B-m+\|\c\|}}{\om(\c)}\,
(-1)^{|\b|}\frac{(m'\!+\!m)!}{m'!}
\prod_{(r,j)\in S(c_0,\c)}\!\!\!\binom{r\!+\!1}{b_{r,j}} 
\,.\,\footnotemark
\end{split}\EE
\footnotetext{The factors in the $m'$-fold product in \e_ref{res_e} 
that contribute $\Psi_r$ are indexed by the elements of~$S_r$;
the $j$-th such factor arises from $\Res{\hb=0}\big\{\hb^{-b_{r,j}}\cZ^*(\hb,Q)\big\}$
with $r\!\ge\!b_{r,j}\!-\!1$.
This leaves $m'\!-\!c_0\!-\!|\c|$ factors that contribute~$\ze(Q)$
from $\Res{\hb=0}\big\{\cZ^*(\hb,Q)\big\}$.
The first expression in \e_ref{split_e} is defined to be $0$ if 
$b_{r,j}\!>\!r\!+\!1$ for some $(r,j)\!\in\!S(c_0,\c)$.}Since the number of above choices is
$$\binom{m'}{c_0,\c,m'\!-\!c_0\!-\!|\c|}
\equiv \frac{m'!}{c_0!\c!(m'\!-\!c_0\!-\!|\c|)!} \,,$$
it follows that the coefficient of $\Psi_0^{c_0}\Psi^{\c}$ 
on the left-hand side of~\e_ref{res_e} is
\begin{equation}\label{combpf_e3}
\frac{\ze^{B-m+\|\c\|}}{\om(\c)c_0!\c!}\,
\sum_{\b\in(\bar\Z^+)^{S(c_0,\c)}}\!\!\Bigg( (-1)^{|\b|}\!\!\!\!\!\!
\prod_{t=B-m-c_0-|\c|+1}^{t=B}\!\!\!\!\!\!\!\!\!\!\!\!\!\!(t\!+\!|\b|)
\prod_{(r,j)\in S(c_0,\c)}\!\!\binom{r\!+\!1}{b_{r,j}}\Bigg).
\end{equation}
If $(c_0,\c)\!=\!(0,\0)$ and thus $(\bar\Z^+)^{S(c_0,\c)}\!\equiv\!\{\0\}$, 
this expression
reduces to $m!\binom{B}{m}\ze^{B-m}$.
Otherwise, \e_ref{combpf_e3} becomes
\begin{equation*}\begin{split}
&\frac{\ze^{B-m+\|\c\|}}{c_0!\c!\om(\c)}\,
\sum_{b=0}^{\i}\Bigg((-1)^b\binom{c_0\!+\!|\c|\!+\!\|\c\|}{b}
\prod_{t=B-m-c_0-|\c|+1}^{t=B}\!\!\!\!\!\!\!\!\!\!\!\!\!\!(t\!+\!b)\Bigg)\\
&\qquad\qquad\qquad\qquad\qquad
=\frac{\ze^{B-m+\|\c\|}}{c_0!\c!\om(\c)}
(-1)^{c_0+|\c|+\|\c\|}(m\!+\!c_0\!+\!|\c|)!\binom{B}{m\!-\!\|\c\|}\,,
\end{split}\end{equation*}
by the first two statements of Lemma~\ref{comb_l0}.
Lemma~\ref{res_lmm} now follows from the last statement of Lemma~\ref{comb_l0}.
\end{proof}

\noindent
For any $d\!\in\!\bar\Z^+$ and $t\!\in\!\Z$, let
\BE{binomdfn_e} \binom{t}{d}=\frac{\prod\limits_{r=0}^{d-1}(t\!-\!r)}{d!}\,.\EE
For $r\!\in\!\bar\Z^+$ and $\p\!\in\!(\bar\Z^+)^n$, 
define $w_r\!\in\!\Q[\al]$ and $C_{r;\p}\!\in\!\Q$ by  
$$ w_r\equiv \sum_{i=1}^{i=n}\al_i^r \equiv 
\sum_{\p\in(\bar\Z^+)^n}\!\!\!\!C_{r;\p}\hat\si_1^{p_1}\hat\si_2^{p_2}\ldots\hat\si_n^{p_n}\,.$$
If $r_1,r_2\!\in\![n]$ with $r_1\!\neq\!r_2$ and $b_1,b_2\!\in\!\bar\Z^+$, let
\BE{symmpol_e}  \p=(p_1,\ldots,p_n),  \quad 
p_r=\begin{cases} b_i,&\hbox{if}~r\!=\!r_i;\\
0,&\hbox{otherwise};\end{cases}
\qquad C^{(b_1,b_2)}_{r_1,r_2}=C_{b_1r_1+b_2r_2;\p}\,.\EE
Thus,  $C^{(b_1,b_2)}_{r_1,r_2}$ is the coefficient of $\hat\si_{r_1}^{b_1}\hat\si_{r_2}^{b_2}$
in the expansion of $w_{b_1r_1+b_2r_2}$ in terms of products of 
the modified (by sign) elementary symmetric polynomials~$\hat\si_r$.
If $b_1\!<\!0$ or $b_2\!<\!0$, set $C^{(b_1,b_2)}_{r_1,r_2}\!=\!0$.

\begin{lmm}\label{Newt_lmm}
If $r_1,r_2\!\in\![n]$ with $r_1\!\neq\!r_2$ and $b_1,b_2\!\in\!\bar\Z^+$ with $b_1\!+\!b_2\!\neq\!0$,
$$C^{(b_1,b_2)}_{r_1,r_2}=\binom{b_1\!+\!b_2\!-\!1}{b_2}r_1+\binom{b_1\!+\!b_2\!-\!1}{b_1}r_2. $$
\end{lmm}

\begin{proof}
If $b_1\!\in\!\Z^+$ and $\al_1,\ldots,\al_n$ are the $n$~roots of the polynomial 
$\al^n\!-\!\al^{n-r_1}=\al^{n-r_1}\big(\al^{r_1}\!-\!1\big)$,
$$C^{(b_1,0)}_{r_1,r_2}=\sum_{i=1}^{i=n}\al_i^{b_1r_1}=r_1\cdot1^{b_1}+(n\!-\!r_1)\cdot0^{b_1r_1}
=r_1\,;$$
thus, the claim holds when either $b_1\!=\!0$ or $b_2\!=\!0$.
If $b_1,b_2\!\in\!\Z^+$, 
$$w_{b_1r_1+b_2r_2}
=\sum_{r=1}^{b_1r_1+b_2r_2-1}\!\!\!\!\!\!\hat\si_rw_{b_1r_1+b_2r_2-r}
~+~(b_1r_1\!+\!b_2r_2)\hat\si_{b_1r_1+b_2r_2}$$
by Newton's identity \cite[p577]{Artin}.
This gives
$$C^{(b_1,b_2)}_{r_1,r_2}=C^{(b_1-1,b_2)}_{r_1,r_2}+C^{(b_1,b_2-1)}_{r_1,r_2}
\qquad\forall\,b_1,b_2\!\in\!\Z^+.$$
Along with the $b_1\!=\!0$ or $b_2\!=\!0$ case, this implies the claim by induction.
\end{proof}

\begin{lmm}\label{Lbinom_lmm}
The power series $L\!\in\!1\!+\!q\Q[[q]]$ defined by \e_ref{Ldfn_e0} satisfies
\BE{Lbinom_e}\LRbr{\frac{nL(q)^{\nua d+nt}}{|\a|+\nua L(q)^n}}_{q;d}
=\big(\a^{\a}\big)^d\binom{d\!+\!t\!-\!1}{d}\EE
for all $d\!\in\!\bar\Z^+$ and $t\!\in\!\Z$.
\end{lmm}

\begin{proof}
In the two extremal cases, by \e_ref{Ldfn_e2}
$$\frac{nL(q)^{\nua d+nt}}{|\a|+\nua L(q)^n}=
\begin{cases}(1+q)^{d+t-1},&\hbox{if}~|\a|\!=\!0;\\
(1-\a^{\a} q)^{-t},&\hbox{if}~|\a|\!=\!n.\end{cases}$$
Thus, the claim in these two cases follows from the binomial theorem;
so, we can assume that $0\!<\!|\a|\!<\!n$.
Replacing $\a^{\a}q$ by~$q$ in~\e_ref{Ldfn_e0}, we observe that it is sufficient
to prove~\e_ref{Lbinom_lmm} with $L$ defined by~\e_ref{Ldfn_e0} with $\a^{\a}$
replaced by~1 and $|\a|$ by some $a\!\in\!\Z^+$ with $a\!<\!n$;
thus, $\nua\!=\!\nu\!\equiv\!n\!-\!a$.\\

\noindent
With these reductions, for each $n$-th root of unity $\ze\!\in\!\C$ let
$$L_{\ze}(q)=\ze L(\ze^a q)\in \Q[[q]].$$ 
Then,
\begin{alignat*}{1}
L_{\ze}(q)^n-qL_{\ze}(q)^a =1 \quad&\Lra\quad
\frac{1}{a+\nu L_{\ze}(q)^n}=
\frac{L_{\ze}'(q)}{q L_{\ze}(q)^{a+1}}\\
\ze^{\nu d+nt}\cdot\big(\ze^a q\big)^d=q^d \quad&\Lra\quad
\LRbr{\frac{nL(q)^{\nu d+nt}}{a+\nu L(q)^n}}_{q;d}
=\sum_{\ze^n=1}\LRbr{\frac{L_{\ze}(q)^{\nu d+nt}}{a+\nu L_{\ze}(q)^n}}_{q;d}\,,
\end{alignat*}
where $'$ denotes $q\frac{\tnd}{\tnd q}$ as before.
Combining these two conclusions, we find that 
\BE{Lsum_e} \LRbr{\frac{nL(q)^{\nu d+nt}}{a+\nu L(q)^n}}_{q;d}
=\sum_{\ze^n=1}
\LRbr{L_{\ze}(q)^{\nu(d+1)+n(t-1)}\frac{L_{\ze}'(q)}{L_{\ze}(q)}}_{q;d+1}\,.\EE
If $\nu(d\!+\!1)\!+\!n(t\!-\!1)\!=\!0$, this gives
\begin{equation*}\begin{split}
\LRbr{\frac{nL(q)^{\nu d+nt}}{a+\nu L(q)^n}}_{q;d}
=(d\!+\!1)\sum_{\ze^n=1}\LRbr{\ln L_{\ze}(q)}_{q;d+1}
&=(d\!+\!1)\LRbr{\ln\bigg(\prod_{\ze^n=1}\!\!L_{\ze}(q)\bigg)}_{q;d+1}\\
&=(d\!+\!1)\LRbr{\ln(-1)^{n-1}}_{q;d+1}=0,
\end{split}\end{equation*}
since $\{L_{\ze}\}_{\ze^n=1}$ is the set of the roots of 
$\ell^n-q\ell^a-1=0$.
Since $\nu\!<\!n$, our assumption on $(d,t)$ implies that $0\!\le\!d\!+\!t\!-\!1\!<\!d$, 
and so the right-hand side of~\e_ref{Lbinom_e} also vanishes.
If $\nu(d\!+\!1)\!+\!n(t\!-\!1)\!>\!0$, \e_ref{Lsum_e} and Lemma~\ref{Newt_lmm} give
\begin{equation*}\begin{split}
\LRbr{\frac{nL(q)^{\nu d+nt}}{a+\nu L(q)^n}}_{q;d}
&=\frac{d\!+\!1}{\nu(d\!+\!1)+n(t\!-\!1)}
\sum_{\ze^n=1}\LRbr{L_{\ze}(q)^{\nu(d+1)+n(t-1)}}_{q;d+1}\\
&=\frac{d\!+\!1}{\nu(d\!+\!1)+n(t\!-\!1)}C_{\nu,n}^{(d+1,t-1)}
=\binom{d\!+\!t\!-\!1}{d}\,,
\end{split}\end{equation*}
as claimed (the last equality holds even if $t\!\le\!0$).
If $\nu(d\!+\!1)\!+\!n(t\!-\!1)\!<\!0$, \e_ref{Lsum_e} and Lemma~\ref{Newt_lmm} give
\begin{equation*}\begin{split}
\LRbr{\frac{nL(q)^{\nu d+nt}}{a+\nu L(q)^n}}_{q;d}
&=\frac{d\!+\!1}{\nu(d\!+\!1)+n(t\!-\!1)}
\sum_{\ze^n=1}\LRbr{\bigg(\frac{1}{L_{\ze}(q)}\bigg)^{a(d+1)-n(d+t)}}_{q;d+1}\\
&=\frac{d\!+\!1}{\nu(d\!+\!1)+n(t\!-\!1)}
C_{a,n}^{(d+1,-(d+t))}(-1)^{d+1} =(-1)^d\binom{-t}{d}\,,
\end{split}\end{equation*}
since $\{1/L_{\ze}\}_{\ze^n=1}$ is the set of the roots of 
$\ell^n+q\ell^{\nu}-1=0$;
the last equality holds even if $d\!+\!t\!>\!0$.
Since 
$$(-1)^d\binom{-t}{d}=\binom{d\!+\!t\!-\!1}{d}\,,$$
\e_ref{Lbinom_e} holds in this last case as well.
\end{proof}

\begin{crl}\label{Lbinom_crl}
The power series $L\!\in\!1\!+\!q\Q[[q]]$ defined by \e_ref{Ldfn_e0} satisfies
\BE{Lbinomder_e}
\LRbr{\frac{nL(q)^{\nua d+nt}}{(|\a|\!+\!\nua L(q)^n)^k}\cdot\frac{L'(q)}{L(q)}}_{q;d}
=\frac{(\a^{\a})^d}{n^k}\sum_{r=0}^{d-1}\binom{k\!-\!1\!+\!r}{r}\binom{d\!-\!1+\!t}{d\!-\!1\!-\!r}
\bigg(\!-\frac{\nua}{n}\bigg)^r\EE
for all $d\!\in\!\bar\Z^+$ and $k,t\!\in\!\Z$.
\end{crl}

\begin{proof} For $d\!=\!0$, both sides of \e_ref{Lbinomder_e} vanish.
By~\e_ref{Lder_e}, the $k\!=\!0$ case of \e_ref{Lbinomder_e} reduces to Lemma~\ref{Lbinom_lmm}.
For $k\!\neq\!0$, by~\e_ref{Lder_e} and the Binomial Theorem 
\begin{equation*}\begin{split}
\LRbr{\frac{nL(q)^{\nua d+nt}}{(|\a|\!+\!\nua L(q)^n)^k}\cdot\frac{L'(q)}{L(q)}}_{q;d}
&=\a^{\a}\LRbr{\frac{nL(q)^{\nua(d-1)+n(t+1)}}{|\a|\!+\!\nua L(q)^n}
\cdot \frac1{(n\!+\!\nua\a^{\a}qL(q)^{|\a|})^k}}_{q;d-1}\\
&=\frac{\a^{\a}}{n^k}\sum_{r=0}^{d-1}\binom{-k}{r}
\LRbr{\frac{nL(q)^{\nua(d-1-r)+n(t+1+r)}}{|\a|\!+\!\nua L(q)^n}}_{q;d-1-r}
\bigg(\frac{\nua}{n}\a^{\a}\bigg)^r\,.
\end{split}\end{equation*}
The claim now follows from Lemma~\ref{Lbinom_lmm}.
\end{proof}

\noindent
For $p,d\!\in\!\Z$, let $\lrbr{p}_d,\lrbr{\hat{p}}_d,\tau_d(p),t_d(p)\in\Z$ be 
as in \e_ref{pmoddfn_e}.
In particular,
\begin{alignat}{1}\label{pairvalues_e}
&(\tau_{d-1}(p)\!-\!\tau_d(p),t_d(p))\in\{(0,0),(1,0),(0,1)\},\\
\label{shift_e}
&1-t_1(p)-\tau_0(p)+\tau_1(p)=
\begin{cases}
1,&\hbox{if}~t_1(p)\!=\!0~\hbox{and}~\tau_0(p)\!=\!\tau_1(p);\\
0,&\hbox{otherwise}.\end{cases}
\end{alignat}
Let $A\!=\!\a^{\a}$ for the remainder this section. 

\begin{lmm}\label{ntcsym_lmm}
For all $d\!\in\!\bar\Z^+$, $p\!\in\!\Z$, and $f\!:\Z^2\lra\R$,
\BE{ntcsym_e}\begin{split}
&\sum_{\begin{subarray}{c}d_1+d_2=d\\ d_1,d_2\ge0 \end{subarray}} 
\ntc^{(d_1)}_{\lrbr{p}_{d_2},\lrbr{p}_{d_2}-\nua d_1}
\ntc^{(d_2)}_{\lrbr{\hat{p}}_{d_2},\lrbr{\hat{p}}_{d_2}-\nua d_2}f\big(\tau_{d_2}(p),t_{d_2}(p)\big)\\
&\hspace{1in}
=\begin{cases}
f\big(\tau_0(p),t_0(p)\big),&\hbox{if}~d\!=\!0;\\
-A(1\!-\!\tau_0(p)\!+\!\tau_1(p)\!-\!t_1(p)) f(\tau_1(p),t_1(d)),&\hbox{if}~d\!=\!1;\\
0,&\hbox{if}~d\!\ge\!2.\end{cases}
\end{split}\EE
\end{lmm}

\begin{proof}
The $d\!=\!0$ case of \e_ref{ntcsym_e} is immediate from  $\ntc^{(0)}_{p,s}\!=\!\de_{p,s}$.
If $\tau_0(p)\!=\!\tau_d(p)$ and $t_d(p)\!=\!0$, 
\e_ref{ntcsym_e} reduces to \cite[(2.9)]{PoZ}.
In general, let $d_1^*,\ldots,d_k^*\!\in\!\Z^+$ be such that 
$$\tau_0(p)=\tau_{d_1^*-1}(p)>\tau_{d_1^*}(p)=\tau_{d_2^*-1}(p)>\tau_{d_2^*}(p)=\tau_{d_3^*-1}(p)
\ldots >\tau_{d_k^*}(p)=\tau_d(p);$$
if $\tau_0(p)\!=\!\tau_d(p)$, $k\!\equiv\!0$. Let $d_0^*\!=\!0$ and $d_{k+1}^*\!=\!d\!+\!1$.
If $1\!\le\!i\!\le\!k$, then $\lrbr{p}_{d_i^*-1}\!<\!\nua$, $\lrbr{\hat{p}}_{d_i^*}\!<\!l\!+\!\nua$,
and~so 
\begin{alignat*}{2}
d_{i-1}^*\!\le\!d_2\!<\!d_i^*  \qquad&\Lra\qquad \lrbr{p}_{d_2}-\nua(d\!-\!d_2)<0 
&\qquad&\Lra\qquad \ntc^{(d-d_2)}_{\lrbr{p}_{d_2},\lrbr{p}_{d_2}-\nua(d-d_2)}=0\,;\\
d_i^*\!\le\!d_2\!<\!d_{i+1}^* \qquad&\Lra\qquad \lrbr{\hat{p}}_{d_2}-\nua d_2<l
&\qquad&\Lra\qquad \ntc^{(d_2)}_{\lrbr{\hat{p}}_{d_2},\lrbr{\hat{p}}_{d_2}-\nua d_2}=0\,.
\end{alignat*}
Thus, all summands on the left-hand side of \e_ref{ntcsym_e} vanish if $k\!\neq\!0$.
Finally, if $d\!>\!0$ and $k\!=\!0$, but $t_d(p)\!=\!1$,
then $\lrbr{p}_d,\lrbr{\hat{p}}_d\!<\!l$, and~so
$$\ntc^{(d)}_{\lrbr{\hat{p}}_d,\lrbr{\hat{p}}_d-\nua d}=0;\quad
\lrbr{p}_{d_2}-\nua(d\!-\!d_2)<l ~~~\Lra~~~
\ntc^{(d-d_2)}_{\lrbr{p}_{d_2},\lrbr{p}_{d_2}-\nua(d-d_2)}=0~~\forall\,d_2=0,1,\ldots,d-1.$$
Thus, all summands on the left-hand side of \e_ref{ntcsym_e} vanish in this case as well.
In light of~\e_ref{shift_e}, this confirms~\e_ref{ntcsym_e}.
\end{proof}

\begin{lmm}\label{sumsplit_lmm}
For all $d\!\in\!\bar\Z^+$ and $p\!\in\!\Z$,
\BE{sumsplit_e}\begin{split}
&\sum_{\bfd\in\cP_4(d)}\Bigg\{
\ntc^{(d_1)}_{\lrbr{p}_{d_2+d_3},\lrbr{p}_{d_2+d_3}-\nua d_1}
\ntc^{(d_2)}_{\lrbr{\hat{p}}_{d_2+d_3},\lrbr{\hat{p}}_{d_2+d_3}-\nua d_2}\\
&\hspace{1.2in}
A^{d_3}\binom{d_3\!+\!\tau_{d_2+d_3}(p)\!-\!t_{d_2+d_3}(p)}{d_3}
\LRbr{\frac{nL(q)^{\nua d_4-n\tau_{d_2+d_3}(p)}}{|\a|\!+\!\nua L(q)^n}}_{q;d_4}\Bigg\}
=\de_{d,0}\,.
\end{split}\EE
\end{lmm}

\begin{proof}
The $d\!=\!0$ case is clear; so we assume $d\!>\!0$.
Using Lemma~\ref{ntcsym_lmm} to sum over $d_1\!+\!d_2\!=\!d'$ with $d'$ fixed, 
we find that the left-hand side of~\e_ref{sumsplit_lmm} equals
\begin{equation*}\begin{split}
&\LRbr{\frac{nL(q)^{\nua d-n\tau_0(p)}}{|\a|\!+\!\nua L(q)^n}}_{q;d}\\
&~~+\!\!\!\sum_{\begin{subarray}{c}d_3+d_4=d\\ 1\le d_3\le d\end{subarray}}\!\!\!\!\!
A^{d_3}\binom{d_3\!-\!1\!+\!\tau_{d_3}(p)\!-\!t_{d_3}(p)}{d_3\!-\!1}
\frac{d_3\tau_{d_3-1}(p)\!+\!(d_3\!-\!1)(t_{d_3}(p)\!-\!\tau_{d_3}(p))}{d_3}
\LRbr{\frac{nL(q)^{\nua d_4-n\tau_{d_3}(p)}}{|\a|\!+\!\nua L(q)^n}}_{q;d_4}.
\end{split}\end{equation*}
By \e_ref{pairvalues_e},
$$\binom{d_3\!-\!1\!+\!\tau_{d_3}(p)\!-\!t_{d_3}(p)}{d_3\!-\!1}
\frac{d_3\tau_{d_3-1}(p)\!+\!(d_3\!-\!1)(t_{d_3}(p)\!-\!\tau_{d_3}(p))}{d_3}
=\binom{d_3\!-\!1\!+\!\tau_{d_3-1}(p)}{d_3}\,.$$
It follows that  the left-hand side of~\e_ref{sumsplit_e} equals
\BE{sumsplit_e3}\begin{split}
&A^d\binom{d\!-\!1\!-\!\tau_0(p)}{d}
+A^d\!\!\!\sum_{\begin{subarray}{c}d_3+d_4=d\\ 1\le d_3\le d\end{subarray}}\!\!
\binom{d_3\!-\!1\!+\!\tau_{d_3-1}(p)}{d_3}\binom{d_4\!-\!1\!-\!\tau_{d_3}(p)}{d_4}\,;
\end{split}\EE
see also Lemma~\ref{Lbinom_lmm}.
By induction on $s\!=\!0,1,\ldots,d\!-\!1$,
$$\sum_{\begin{subarray}{c}d_3+d_4=d\\ d-s\le d_3\le d\end{subarray}}\!\!
\binom{d_3\!-\!1\!+\!\tau_{d_3-1}(p)}{d_3}\binom{d_4\!-\!1\!-\!\tau_{d_3}(p)}{d_4}
=(-1)^s\binom{d\!-\!1}{s}\binom{d\!-\!1\!-\!s+\!\tau_{d-1-s}(p)}{d}\,.$$
Setting $s\!=\!d\!-\!1$ in the last identity, we conclude that the sum
in~\e_ref{sumsplit_e3} vanishes.
\end{proof}

\begin{crl}\label{sumsplit_crl}
For all $d\!\in\!\bar\Z^+$, $p,t\!\in\!\Z$, and $f\!\in\!\R[[q]]$,
\begin{equation*}\begin{split}
\sum_{\bfd\in\cP_4(d)}\!\!\Bigg\{
\ntc^{(d_1)}_{\lrbr{p}_{d_2+d_3},\lrbr{p}_{d_2+d_3}-\nua d_1}
\ntc^{(d_2)}_{\lrbr{\hat{p}}_{d_2+d_3},\lrbr{\hat{p}}_{d_2+d_3}-\nua d_2}
(\a^{\a})^{d_3}\binom{d_3\!+\!\tau_{d_2+d_3}(p)\!-\!t_{d_2+d_3}(p)\!-\!t}{d_3}\qquad&\\
\times\LRbr{\frac{nL(q)^{\nua d_4+n(t-\tau_{d_2+d_3}(p))}}{|\a|\!+\!\nua L(q)^n}f(q)}_{q;d_4}\Bigg\}
=\LRbr{\frac{nL(q)^{\nua d}f(q)}{|\a|\!+\!\nua L(q)^n}}_{q;d}\,.&
\end{split}\end{equation*}
\end{crl}

\begin{proof} Replacing $p$ with $p\!-\!nt$, we can assume that $t\!=\!0$.
The $d\!=\!0$ case is clear; so we assume that $d\!\ge\!1$ and that 
the above identity holds with~$d$ replaced by any nonnegative 
integer $d'\!<\!d$.
The left-hand side of this identity is given~by
\begin{gather*}
\textnormal{LHS}_d=
\sum_{\begin{subarray}{c}d'+d''=d\\ d',d''\ge0\end{subarray}}
\!\!\!\!\!C_{d',d''}\LRbr{f(q)}_{q;d''}, \qquad\hbox{where}\\
\begin{split}
C_{d',d''}=\sum_{\bfd\in\cP_4(d')}\!\!\Bigg\{
\ntc^{(d_1)}_{\lrbr{p}_{d_2+d_3},\lrbr{p}_{d_2+d_3}-\nua d_1}
\ntc^{(d_2)}_{\lrbr{\hat{p}}_{d_2+d_3},\lrbr{\hat{p}}_{d_2+d_3}-\nua d_2}
(\a^{\a})^{d_3}\binom{d_3\!+\!\tau_{d_2+d_3}(p)\!-\!t_{d_2+d_3}(p)}{d_3}\quad&\\
\times\LRbr{\frac{nL(q)^{\nua (d_4+d'')-n\tau_{d_2+d_3}(p)}}{|\a|\!+\!\nua L(q)^n}}_{q;d_4}
\Bigg\}.&
\end{split}\end{gather*}
So, it is sufficient to show that 
$$C_{d',d''}=\LRbr{\frac{nL(q)^{\nua d}}{|\a|\!+\!\nua L(q)^n}}_{q;d'}$$
for $d'\!=\!0,1,\ldots,d$.
For $d'\!<\!d$, this is the case by the inductive assumption applied with
$f\!=\!L^{\nua d''}$.
For $d'\!=\!d''$, this is the case by Lemmas~\ref{sumsplit_lmm} and~\ref{Lbinom_lmm}.
\end{proof}

\vspace{.2in}

\noindent
{\it Department of Mathematics, SUNY Stony Brook, Stony Brook, NY 11794-3651\\
azinger@math.sunysb.edu}

\end{document}